\documentclass[a4paper,11pt]{amsart}

\usepackage[left=2.7cm,right=2.7cm,top=3.5cm,bottom=3cm]{geometry}
\usepackage{graphicx}
\newcommand{\Bmu}{\mbox{$\raisebox{-0.59ex}
  {$l$}\hspace{-0.18em}\mu\hspace{-0.88em}\raisebox{-0.98ex}{\scalebox{2}
  {$\color{white}.$}}\hspace{-0.416em}\raisebox{+0.88ex}
  {$\color{white}.$}\hspace{0.46em}$}{}}

\numberwithin{equation}{section}

\usepackage{amsthm,amssymb,amsmath,amsfonts,mathrsfs,amscd,amsbsy,dsfont,verbatim,relsize,extarrows}
\usepackage{stmaryrd,dsfont}
\usepackage[utf8]{inputenc}
\usepackage[T1]{fontenc}
\usepackage[all,cmtip]{xy}  
\usepackage{latexsym}
\usepackage{longtable}
\usepackage{mathtools}
\usepackage{marginnote}

\usepackage{graphicx} 
\usepackage{stmaryrd}
\usepackage{amsmath}
\usepackage{amsthm}
\usepackage{amssymb}
\usepackage{mathrsfs}
\usepackage{tikz-cd}
\usepackage{verbatim}

\usepackage[pagebackref]{hyperref}

\mathtoolsset{showonlyrefs}

\theoremstyle{plain}
\newtheorem{theorem}{Theorem}[section]
\newtheorem{proposition}[theorem]{Proposition}
\newtheorem{lemma}[theorem]{Lemma}
\newtheorem{corollary}[theorem]{Corollary}

\newcounter{lettera}

\newtheorem{itheorem}[lettera]{Theorem}

\theoremstyle{definition}
\newtheorem{definition}[theorem]{Definition}
\newtheorem{assumption}[theorem]{Assumption}

\theoremstyle{remark}
\newtheorem{remark}[theorem]{Remark}

\newtheorem{remark/notation}[theorem]{Remark/Notation}
\newtheorem{notation/convention}[theorem]{Notation/Convention}

\newcommand{\longmono}{\mbox{\;$\lhook\joinrel\longrightarrow$\;}}

\newcommand{\longepi}{\mbox{\;$\relbar\joinrel\twoheadrightarrow$\;}}

\newfont{\cyr}{wncyr10 scaled 1100}
\newcommand{\Sha}{\mbox{\cyr{X}}}

\newcommand{\cyc}{\mathrm{cyc}}
\newcommand{\unr}{\mathrm{unr}}

\DeclareMathOperator{\Fitt}{Fitt}
\DeclareMathOperator{\Gal}{Gal}
\DeclareMathOperator{\Sel}{Sel}
\DeclareMathOperator{\ind}{ind}
\DeclareMathOperator{\ord}{ord}
\DeclareMathOperator{\deff}{def}
\DeclareMathOperator{\loc}{loc}
\DeclareMathOperator{\length}{length}
\DeclareMathOperator{\Frob}{Frob}

\DeclareMathOperator{\Frac}{Frac}
\DeclareMathOperator{\divv}{div}
\DeclareMathOperator{\Hom}{Hom}
\DeclareMathOperator{\coker}{coker}
\DeclareMathOperator{\Cor}{Cor}

\DeclareMathOperator{\im}{im}
\DeclareMathOperator{\M}{M}
\DeclareMathOperator{\charr}{char}
\DeclareMathOperator{\Supp}{Supp}
\DeclareMathOperator{\tors}{tors}
\DeclareMathOperator{\Tor}{Tor}
\DeclareMathOperator{\Gr}{Gr}
\DeclareMathOperator{\GL}{GL}
\DeclareMathOperator{\Weil}{Weil}

\newcommand{\defeq}{\vcentcolon=}

\newcommand{\Q}{\mathds Q}
\newcommand{\N}{\mathds N}
\newcommand{\Z}{\mathds Z}

\newcommand{\F}{\mathds F}

\newcommand{\p}{\mathfrak p}

\makeatletter
\@namedef{subjclassname@2020}{%
  \textup{2020} Mathematics Subject Classification}
\makeatother

\newcommand{\enrico}[1]{{\color{red} $\clubsuit \clubsuit \clubsuit$ Enrico: [#1]}}

\begin{document}

\title[Higher Fitting ideals of anticyclotomic Shafarevich--Tate groups]{Higher Fitting ideals and the structure\\of anticyclotomic Shafarevich--Tate groups}
\author{Enrico Da Ronche, Matteo Longo and Stefano Vigni}

\thanks{The authors are partially supported by PRIN 2022 ``The arithmetic of motives and $L$-functions'' and by the GNSAGA group of INdAM. The research by the first and third authors is partially supported by the MUR Excellence Department Project awarded to Dipartimento di Matematica, Universit\`a di Genova, CUP D33C23001110001.}


\address{Dipartimento di Matematica, Universit\`a di Genova, Via Dodecaneso 35, 16146 Genova, Italy}
\email{enrico.daronche@edu.unige.it}
\address{Dipartimento di Matematica, Universit\`a di Padova, Via Trieste 63, 35121 Padova, Italy}
\email{matteo.longo@unipd.it}
\address{Dipartimento di Matematica, Universit\`a di Genova, Via Dodecaneso 35, 16146 Genova, Italy}
\email{stefano.vigni@unige.it}



\begin{abstract}
Let $p$ be a prime number. We investigate a refined version of the Iwasawa main conjectures for rational elliptic curves (and more general Galois representations) over anticyclotomic $\Z_p$-extensions of imaginary quadratic fields, both in the definite and in the indefinite settings. In order to do this, we describe (under mild arithmetic assumptions) all the higher Fitting ideals of Pontryagin duals of Selmer and Shafarevich--Tate groups over anticyclotomic $\Z_p$-extensions in terms of the bipartite Euler systems introduced by Bertolini and Darmon. As an application of our work on Fitting ideals, we offer new results on the structure of (Pontryagin duals of) anticyclotomic Selmer and Shafarevich--Tate groups of elliptic curves.
\end{abstract}

\maketitle


\section{Introduction}

Anticyclotomic Iwasawa main conjectures for elliptic curves (and, more generally, abelian varieties of $\GL_2$-type) have a long and distinguished history that dates back at least to the mid '80s of the past century. Here we would like to mention (in rough chronological order) work of Mazur (\cite{mazur}), Perrin-Riou (\cite{PR4}, \cite{PR3}, \cite{PR1}, \cite{PR2}), Rubin (\cite{Ru2}, \cite{Ru1}), Bertolini (\cite{Bertolini}), Bertolini--Darmon (\cite{BD5}, \cite{BD3}, \cite{BD-IMC}), Darmon--Iovita (\cite{DI}), Howard (\cite{HoHeeg}, \cite{howard-duke}, \cite{howard2012bipartite}), Burungale--Castella--Kim (\cite{BCK}), Burungale--B\"{u}y\"{u}kboduk--Lei (\cite{BBL}) and, recently, Bertolini--Longo--Venerucci (\cite{BLV}); in particular, the contribution by Bertolini--Longo--Venerucci represents the culmination of the techniques (and the underlying philosophy) originally introduced by Bertolini--Darmon and offers (under mild technical assumptions) a proof of many cases of the anticyclotomic main conjectures for elliptic curves. The goal of our article is to investigate a refined version of these conjectures by describing all the higher Fitting ideals of Pontryagin duals of Selmer and Shafarevich--Tate groups of elliptic curves (and more general $p$-adic Galois representations) over anticyclotomic $\Z_p$-extensions of imaginary quadratic fields, where $p$ is a prime number; classical anticyclotomic Iwasawa main conjectures can then be seen as a special case of our results when the Fitting ideals we consider are the initial ones. 

The strategy towards the anticyclotomic Iwasawa main conjectures proposed by Bertolini--Darmon in \cite{BD-IMC} uses level raising arguments to produce classes in the Galois cohomology of an elliptic curve. Later on, the approach of Bertolini--Darmon was formalized by Howard in his theory of \emph{bipartite Euler systems} (\cite{howard2012bipartite}). In the present paper, we adopt the abstract viewpoint of Howard and use bipartite Euler systems to study higher Fitting ideals. More precisely, in Section \ref{sec::4} we obtain results for a broader class of $p$-adic Galois representations admitting a bipartite Euler system, which we specialize to the case of elliptic curves in Section \ref{sec::5} (notice that the general formulation in Section \ref{sec::4} can also be adapted, for example, to case of the $p$-adic Galois representations attached to higher weight modular forms). In order to simplify the exposition, in this introduction we focus on our results for elliptic curves exclusively. 

Let $E$ be an elliptic curve over $\Q$ of conductor $N$ and let $K$ be an imaginary quadratic field of class number $h_K$ and discriminant $D_K$ coprime to $N$. Fix a prime number $p\geq5$ of good ordinary reduction for $E$ such that $p\nmid D_Kh_K$ and let $T$ be the $p$-adic Tate module of $E$. Write $N=N^+N^-$, where a prime number divides $N^+$ if it splits in $K$ and divides $N^-$ if it is inert in $K$. Let $G_\Q$ be the absolute Galois group of $\Q$; we assume throughout that $N^-$ is square-free and that 
\begin{itemize}
\item the representation of $G_\Q$ on the $p$-torsion $E[p]$ of $E$ is surjective;
\item $E[p]$ is ramified at all the prime factors of $N^+$;
\item $E[p]$ is ramified at all the prime factors $\ell$ of $N^-$ with $\ell\equiv\pm 1\pmod p$.
\end{itemize}
Let $\bar E_p$ be the reduction of $E$ at $p$ and set $a_p(E)\defeq1+p-\#\bar E_p(\F_p)\in\Z$; we also assume 
\begin{itemize}
\item $a_p(E)\not\equiv 1 \pmod p$ if $p$ splits in $K$;
\item $a_p(E)\not\equiv\pm1 \pmod p$ if $p$ is inert in $K$
\end{itemize}
Notice that conditions analogous to those appearing in the lists of assumptions above are natural to impose when studying questions in the arithmetic of Galois representations attached to abelian varieties or modular forms (see, \emph{e.g.}, \cite{BD-IMC}, \cite{BLV}, \cite{Zhang}).

Write $\nu(N^-)$ for the number of prime factors of $N^-$: we say that we are in the \emph{definite} (respectively, \emph{indefinite}) case if $\nu(N^-)$ is \emph{odd} (respectively, \emph{even}). Finally, let $K_\infty$ be the anticyclotomic $\Z_p$-extension of $K$, whose $m$-th finite layer will be denoted by $K_m$, and write $G_\infty\defeq\Gal(K_\infty/K)$ for its Galois group and $\Lambda\defeq\Z_p[\![G_\infty]\!]$ for its Iwasawa algebra over $\Z_p$. 

Consider the Greenberg Selmer group $\Sel_{\Gr}\bigl(K_\infty,E[p^\infty]\bigr)$, where $E[p^\infty]$ is the $p$-primary torsion subgroup of $E$. In the definite (respectively, indefinite) setting, the Pontryagin dual of $\Sel_{\Gr}\bigl(K_\infty,E[p^\infty]\bigr)$ is a torsion $\Lambda$-module (respectively, $\Lambda$-module of rank $1$) whose $\Lambda$-torsion submodule $\mathcal X$ is isomorphic to the Pontryagin dual of the Shafarevich--Tate group $\Sha_{\Gr}\bigl(K_\infty,E[p^\infty]\bigr)$, \emph{i.e.}, the quotient of $\Sel_{\Gr}\bigl(K_\infty,E[p^\infty]\bigr)$ by its maximal $\Lambda$-divisible submodule. It turns out that $\mathcal X$ is pseudo-isomorphic to $M\oplus M$ for a suitable (finitely generated, torsion) $\Lambda$-module $M$. The constructions in \cite{BD-IMC} and \cite{BLV} produce a bipartite Euler system: we have a distinguished Iwasawa element $\lambda_1\in\Lambda$ in the definite case and a distinguished class $\kappa_1\in\Sel_{\Gr}(K_\infty,T)$ in the indefinite case. More precisely, denoting by $T$ the $p$-adic Tate module of $E$ and by $\Sel_{\Gr}(K_\infty,T)$ the corresponding Selmer group \emph{\`a la} Greenberg, $\lambda_1$ is an anticyclotomic $p$-adic $L$-function, whereas $\kappa_1$ is defined in terms of a trace-compatible system of Heegner points of $p$-power conductor. Given a torsion $\Lambda$-module $M$, denote by $\charr(M)$ the characteristic ideal of $M$; in the present context, the Iwasawa main conjecture is the (predicted) equality 
\[
\charr(\mathcal X)=\begin{cases}\charr\bigl(\Lambda/(\lambda_1)\bigr)^2 & \text{if $\nu(N^-)$ is odd},\\[3mm]\charr\bigl(\Sel_{\Gr}(K_\infty,T)/\Lambda\cdot \kappa_1\bigr)^2 & \text{if $\nu(N^-)$ is even}.
\end{cases}
\]
The main result of this paper is a description of all the higher Fitting ideals of the $\Lambda$-module $\mathcal X$, which allows us to determine the pseudo-isomorphism class of $\mathcal X$.

To state our main result, for all integers $k\geq 0$ let $\mathcal{N}_k$ be the set of products of distinct $k$-admissible prime numbers (in the sense of \cite[p. 18]{BD-IMC}) for the $E$, $p$ and $K$. Define $\epsilon\defeq1$ (respectively, $\epsilon\defeq0$) in the definite (respectively, indefinite) setting and denote by $\mathcal{N}_k^\epsilon$ the subset of $\mathcal{N}_k$ consisting of $n$ such that $\nu(n)\equiv \epsilon\pmod2$, where, as before, $\nu(n)$ is the number of distinct prime factors of $n$; by convention, $1\in\mathcal{N}_k$ and $\nu(1)\defeq0$. As alluded to before, we consider the bipartite Euler system introduced in \cite{BD-IMC} and \cite{BLV}, which consists of two sets 
\[ 
\kappa\defeq\Bigl\{\kappa_n \in\textstyle{\varprojlim_m} H^1(K_m,{T}/p^k{T})\;\Big|\;n \in \mathcal{N}_k^\epsilon\Bigr\},\quad\lambda\defeq\Bigl\{\lambda_n\in\Lambda/p^k\Lambda\;\Big|\;n\in\mathcal{N}_k^{\epsilon+1}\Bigr\},
\]
the inverse limit being taken with respect to the corestriction maps. The sets $\kappa$ and $\lambda$ are related to each other via the two explicit reciprocity laws described in \cite[Theorems 4.1 and 4.2]{BD-IMC} and \cite[\S6.2]{BLV} (see also Definition \ref{defBES}). For all integers $k\geq 1$ and even integers $i\geq 0$, let 
\[
\mathfrak{C}_i(k)\defeq(\lambda_n)_{n \in \mathcal{N}_{2k,}\nu(n) \leq i}\subset \Lambda/(p^k) 
\]
be the ideal of $\Lambda/(p^k)$ that is generated by all the $\lambda_n$ for $n$ a product of at most $i$ distinct $p^k$-admissible primes; the inverse limit $\mathfrak{C}_i\defeq \varprojlim_{k}\mathfrak{C}_i(k)$ is an ideal of $\Lambda$. Moreover, define $\Phi\defeq\Hom_{\Lambda/(p^k)}\bigl(\,\varprojlim_m H^1(K_m,T/p^kT),\Lambda/(p^k)\bigr)$, then for all integers $k\geq1$ and even integers $i\geq0$ set  
\[
\mathfrak{D}_i(k)\defeq\bigl(\bigl\{f(\kappa_n)\mid\text{$f\in\Phi$, $n\in\mathcal{N}_{2k}$, $\nu(n)\leq i$}\bigr\}\bigr)\subset\Lambda/(p^k).
\]
The inverse limit $\mathfrak{D}_i\defeq\varprojlim_k \mathfrak{D}_i(k)$ is an ideal of $\Lambda$.

Now, given proper ideals $I,J$ of $\Lambda$, write $I \sim J$ if there are ideals $\mathfrak{a}$ and $\mathfrak{b}$ of height $2$ of $\Lambda$ such that $\mathfrak{a}I \subset J$ and $\mathfrak{b}J\subset I$. The pseudo-isomorphism class of a finitely generated torsion $\Lambda$-module is uniquely determined by the classes modulo $\sim$ of its (higher) Fitting ideals (see \cite[Lemma 9.2]{KU2012} or Theorem \ref{Thm-pseudo}; the theory of Fitting ideals that is relevant for this paper is reviewed in \S \ref{sec:FittingLambda}). Taking crucial advantage of the pseudo-isomorphism $\mathcal X\sim M\oplus M$, in Proposition \ref{even-odd-prop} we show that the only Fitting ideals of $X$ that are genuinely interesting are the even index ones, as the odd index ones can be recovered from them. In proving this property, we are led to introduce an apparently new algebraic notion (Definition \ref{pseudo-def}) that we call \emph{pseudo square-root} (of certain ideals of $\Lambda$), which we believe may be of independent interest. 

Our main result is a description of the higher Fitting ideals of $\mathcal X$ in terms of the ideals $\mathfrak{C}_i$ and $\mathfrak{D}_i$ introduced above. In the definite case, we prove

\begin{itheorem} \label{iLambda} 
Let $i\geq0$ be an integer. In the definite case, there is an equivalence
\[
\Fitt_i(\mathcal X)\sim\begin{cases}\mathfrak{C}_i^2&\text{if $i$ is even},\\[3mm] \mathfrak{C}_{i-1}\cdot\mathfrak{C}_{i+1}&\text{if $i$ is odd}. \end{cases}
\]
\end{itheorem}

The next result takes care of the indefinite case.

\begin{itheorem} \label{iKappa} 
Let $i\geq0$ be an integer. In the indefinite case, there is an equivalence
\[
\Fitt_i(\mathcal X)\sim\begin{cases}\mathfrak{D}_i^2&\text{if $i$ is even},\\[3mm] \mathfrak{D}_{i-1}\cdot\mathfrak{D}_{i+1}&\text{if $i$ is odd}. \end{cases}
\]
\end{itheorem}

Theorems \ref{iLambda} and \ref{iKappa} are essentially a reformulation of Corollary \ref{final-coro}, which in turn is a consequence of Theorems \ref{final-thm1} and \ref{final-thm2}.

A few remarks are in order. First, of all, we would like to bring to the reader's attention that we obtain results analogous to Theorems \ref{iLambda} and \ref{iKappa} in an abstract situation in which we describe Fitting ideals via the theory of $\Lambda$-adic bipartite Euler systems for a rather general free $\Lambda$-module $\mathbf{T}$ of ranks $2$ equipped with a continuous action of the absolute Galois group of $K$. This abstract result is Theorem \ref{highFittLambda}, which is really the crucial novelty of this paper, from which Theorem \ref{iLambda} follows immediately (upon setting the stage and verifying the various assumptions needed to develop our general strategy). From our point of view, the interest in such an abstract formulation lies in future possible applications to other contexts, \emph{e.g.}, to the case of higher (even) weight modular forms or to the study of supersingular elliptic curves. As will be apparent, our strategy towards Theorem \ref{highFittLambda} is inspired by ideas of Mazur--Rubin (\cite{MR}) and of Howard (\cite{HoHeeg}); in a nutshell, we combine information on Selmer groups of the specializations $T_\mathfrak{P}$ of $\mathbf{T}$ at (all but finitely many) height $1$ prime ideals $\mathfrak{P}$ of $\Lambda$, where $T_\mathfrak{P}\defeq\mathbf{T}\otimes_{\Lambda}\mathcal{O}_\mathfrak{P}$ with $\mathcal{O}_\mathfrak{P}$ the integral closure of the domain $\Lambda/\mathfrak{P}$. Thus, part of this paper is devoted to obtaining (again, using an approach via bipartite Euler systems) structure theorems for Selmer groups of Galois representations over artinian local rings and over discrete valuations rings: our main results, which are similar in flavour to results of Kim (\cite{kim2024higher}), are Theorem \ref{artSel} and Corollary \ref{artKappa} (respectively, Theorems \ref{thmDVR} and \ref{dvrKappa}) in the artinian (respectively, discrete valuation ring) case. The reader is then encouraged to compare our results with \cite[Theorem 4.19]{kim2024higher} and \cite[Proposition 4.5.8]{MR} in the case of artinian local rings and with \cite[Theorem 4.19]{kim2024higher} and \cite[Proposition 4.5.8]{MR} in the case of discrete valuation rings. 

A second remark is the following. Although we insisted, in this introduction, on separating the definite case and the indefinite case, it is worth observing that one interesting feature of the formalism of bipartite Euler system that we develop 
in this paper is that, actually, our results in the indefinite setting can be alternatively stated in terms of the ideals $\mathfrak{C}_i$; more precisely, in both settings the equivalence $\Fitt_i(\mathcal X)\sim \mathfrak{C}_i^2$ holds true for all even $i\geq0$ (with $\nu(n) \leq i+1$ in the definition of $\mathfrak{C}_i(k)$ in the indefinite case), and then (as pointed out above) a formula for odd Fitting ideals involving the $\mathfrak C_i$ follows. In other words, quite surprisingly, the theta elements $\lambda_n$ control Fitting ideals not only in the definite case (as it is natural to expect), but also in the indefinite setting. Unfortunately, as of today we have no theoretical explanation of this phenomenon. 

We conclude this introduction by briefly commenting on the existing literature on higher Fitting ideals in an Iwasawa-theoretic setting. The first source of inspiration for this project was the paper \cite{ohshita2021higher} by Ohshita, which combines techniques due to Kurihara (\cite{KU2012}, \cite{kurihara2014refined}) with others coming from foundational work of Mazur--Rubin (\cite{MR}). One important difference to bear in mind when approaching our paper is that all results in \cite{KU2012}, \cite{kurihara2014refined} and \cite{ohshita2021higher} can be applied only to representations that are \emph{not} self-dual, while self-duality (which is responsible, for example, for the pseudo-isomorphism $\mathcal X\sim M\oplus M$) is a crucial feature of the representations studied in our article. Moreover, in \cite{KU2012}, \cite{kurihara2014refined} and \cite{ohshita2021higher} the statement of the relevant main conjecture is taken as an \emph{assumption} and results on the ideals $\Fitt_i$ are proved by induction on $i$. On the contrary, here we \emph{do not} assume anticyclotomic main conjectures: rather, we extend the approach in \cite{BD-IMC}, \cite{BLV} and \cite{howard2012bipartite} to prove these conjectures and show that our extension can be fruitfully applied to study higher Fitting ideals as well. Our general impression is that, when compared to the approach of Kurihara and Ohshita, the Euler systems from \cite{BD-IMC} offer more flexibility and strength. Finally, for recent results on the structure of Selmer groups and the Iwasawa main conjecture for elliptic curves along a different line of investigation, see \cite{kim-AJM}.

\subsection*{Structure of the paper}

In Section \ref{sec::2} we review the main properties of Fitting ideals, with a special focus on Fitting ideals of finitely generated modules over a discrete valuation ring or over $\Lambda$. Section \ref{sec::3} is devoted to the study of bipartite Euler systems and to their application to structure theorems for Selmer and Shafarevich--Tate groups of Galois representations over principal artinian local rings and discrete valuation rings. Here we also introduce the main definitions and properties for bipartite Euler systems over $\Lambda$. In Section \ref{sec::4} we give our main result, namely, we describe (in terms of bipartite Euler systems) all the higher Fitting ideals of (Pontryagin duals of) Shafarevich--Tate groups attached to anticyclotomic Galois representations. Finally, in Section \ref{sec::5} we apply our general result to the case of elliptic curves.

\subsection*{Acknowledgements}
We thank Beatrice Ostorero Vinci for several helpful discussions on some of the topics of this paper.

\section{Background on Fitting ideals} \label{sec::2}

In this section, we review some facts about Fitting ideals; although the theory can be developed in the context of finitely generated modules over a commutative ring, for the sake of simplicity we shall stick to finitely presented modules. For details, we refer the reader to, \emph{e.g.}, \cite[\S 20.2]{Eis}, \cite[Ch. XIX, \S 2]{lang} and \cite[Ch. 3]{northcott}.

\subsection{Definition and basic properties}

From here on, $R$ is a commutative ring and $M$ is a finitely presented $R$-module; if $R$ is noetherian (which will always be the case in our applications), then our condition on $M$ is equivalent to $M$ being finitely generated over $R$. Let $i\geq0$ be an integer and let $n\geq i$ be an integer such that $M$ can be generated over $R$ by $n$ elements. Choose a presentation
\[
\bigoplus_{j=1}^k R \overset\varphi\longrightarrow \bigoplus_{i=1}^n R \longrightarrow M \longrightarrow 0
\]
of $M$. Denote by $A\in\M_{n\times k}(R)$ the matrix associated with $\varphi$ with respect to the canonical basis; finally, given an integer $r\geq0$, write $I_r(A)$ for the ideal of $R$ generated by the minors of order $r$ of $A$, with the convention that $I_r(A)\defeq(0)$ if $r>\min\{n,k\}$ (recall that a minor of order $0$ is defined to be $1$, so that $I_0(A)=R$).

In the definition that follows, the notation above is in force.

\begin{definition} \label{fitting-def}
Let $i\geq0$ be an integer. The \emph{$i$-th Fitting ideal} of $M$ over $R$ is 
\[
\Fitt_i(M)\defeq\begin{cases} I_{n-i}(A) & \text{for $i\leq n$},\\[2mm]R& \text{for $i>n$}.\end{cases}
\]
\end{definition}

The ideal $\Fitt_0(M)$ is sometimes called \emph{the (initial) Fitting ideal} of $M$, in which case the ideals $\Fitt_i(M)$ for $i\geq1$ are referred to as the \emph{higher Fitting ideals} of $M$. A Fitting ideal $\Fitt_i(M)$ with $i$ even (respectively, odd) is called \emph{even} (respectively, \emph{odd}). When we need to specify the ring $R$, we write $\Fitt_{R,i}(M)$ in place of $\Fitt_i(M)$.

The next result asserts that Definition \ref{fitting-def} does indeed make sense.

\begin{proposition}
The Fitting ideals of $M$ are well defined, i.e., their definition is independent of the choice of a presentation of $M$ as an $R$-module.
\end{proposition}

\begin{proof} See, \emph{e.g.}, \cite[Corollary--Definition 20.4]{Eis} or \cite[Ch. 3, Theorem 1]{northcott}. \end{proof}

The next result collects some properties of Fitting ideals.

\begin{proposition} \label{pfi}
The following properties hold for a finitely presented module $M$ over $R$:
\begin{enumerate}
\item $\Fitt_i(M)$ is finitely generated for all $i\geq0$;
\item $\Fitt_i(M)\subset\Fitt_{i+1}(M)$ for all $i\geq0$;
\item if $M$ can be generated by $i$ elements, then $\Fitt_i(M)=R$;
\item for any ring homomorphism $f:R\to S$, the equality
\[
\Fitt_i(M\otimes_RS)=f\bigl(\Fitt_i(M)\bigr)\cdot S
\]
holds for all $i\geq0$, the tensor product being taken with respect to $f$;
\item if $R$ is local and $\Fitt_i(M)=R$, then $M$ can be generated by $i$ elements;
\item if $I$ is an ideal of $R$, then $\Fitt_0(R/I)=I$;
\item if $M\twoheadrightarrow N$ is a surjection of finitely presented $R$-modules, then the inclusion
\[
\Fitt_i(M)\subset\Fitt_i(N)
\]
holds for all $i\geq0$;
\item given an exact sequence $M \overset{f}\to N \rightarrow B \rightarrow 0$ of finitely presented $R$-modules, there is an inclusion
\[
\Fitt_i(M) \Fitt_j(B) \subset \Fitt_{i+j}(N)
\]
for all $i,j\geq0$;
\item if $\mathfrak{p}$ is a prime ideal of $R$, then $\Fitt_i(M) $ is not contained in $\mathfrak{p}$ if and only if $M_\mathfrak{p}$ can be generated by $i$ elements over $R_\mathfrak{p}$;
\item the equality
\[
\Fitt_i\Biggl(\bigoplus_{j=1}^n M_j\Biggr)=\sum_{\substack{(s_1,\dots,s_n)\in\N^n\\s_1 + \dots + s_n=i}} \prod_{j=1}^n \Fitt_{s_j}(M_j)
\]
holds for all $i\geq0$.
\end{enumerate}
\end{proposition}

\begin{proof} Parts (1), (2) and (3) are \cite[Ch. XIX, Proposition 2.4]{lang}. Part (4) is \cite[Corollary 20.5]{Eis}, while part (5) is \cite[Proposition 20.6]{Eis} and part (6) is \cite[Ch. XIX, Corollary 2.6]{lang}. On the other hand, part (7) is immediate from the definitions. As for part (8), the corresponding result when $f$ is injective (\emph{i.e.}, for short exact sequences) is \cite[Ch. XIX, Proposition 2.7]{lang}; our statement can then be obtained by considering the induced short exact sequence
\[
0\longrightarrow M/\ker(f)\longrightarrow N\longrightarrow B\longrightarrow0
\]
and using part (7). To prove part (9), notice that, by parts (3) and (5), $M_\mathfrak p$ can be generated by $i$ elements over $R_\mathfrak p$ if and only if $\Fitt_i(M_\mathfrak p)=R_\mathfrak p$; since $M_\mathfrak p\simeq M\otimes_RR_\mathfrak p$, part (4) gives $\Fitt_i(M_\mathfrak p)=\Fitt_i(M)\cdot R_\mathfrak p$, and the claim follows. Finally, part (10) can be deduced from \cite[Ch. XIX, Proposition 2.8]{lang} by induction on $n$. \end{proof}

For a survey describing applications of Fitting ideals (especially initial Fitting ideals) to algebraic number theory and arithmetic geometry, the reader may wish to consult \cite{greither}.

\subsection{Fitting ideals of modules over a DVR}

Let $R$ be a discrete valuation ring with uniformizer $\pi$. We recall the following well-known structure theorem.

\begin{theorem} \label{DVR-isom-thm}
Let $M$ be a finitely generated torsion $R$-module. If $M$ is non-trivial, then there is an isomorphism of $R$-modules
\[
M \simeq \bigoplus_{j=1}^n R\big/\bigl(\pi^{k_j}\bigr),
\]
where $k_j \geq 0$ are integers uniquely determined up to order.
\end{theorem}

\begin{proof} See, \emph{e.g.}, \cite[Ch. III, Theorem 7.5]{lang}, of which this result is a special case. \end{proof}

As a consequence of Theorem \ref{DVR-isom-thm}, the structure of $M$ is determined (up to isomorphism) by the sequence of ideals $\bigl(\Fitt_i(M)\bigr)_{i\geq0}$  , as now we show.

\begin{theorem}\label{FittingDVR}
Let $M$ be a non-trivial finitely generated torsion $R$-module and suppose that
\begin{equation} \label{M-isom}
M \simeq \bigoplus_{j=1}^n R\big/\bigl(\pi^{k_j}\bigr)
\end{equation}
for integers $n\geq1$ and $1\leq k_1\leq\dots\leq k_n$. Then for each $i\in\{1,\dots,n\}$ there is an equality 
\[
\Fitt_i(M)=(\pi^{c_i})
\]
of ideals of $R$, where $c_i\defeq\sum_{j=1}^{n-i}k_j$.
\end{theorem}

\begin{proof} To begin with, observe that if $i\geq0$ and $k\geq1$ are integers, then there is an equality
\begin{equation} \label{Fitt-pi-eq}
\Fitt_i\Bigl(R\big/\bigl(\pi^k\bigr)\!\Bigr)=\begin{cases}(\pi^k)&\text{for $i=0$},\\[2mm]R&\text{for $i\geq1$}.\end{cases}
\end{equation}
Now fix $i\in\{1,\dots,n\}$. Combining isomorphism \eqref{M-isom} with equality \eqref{Fitt-pi-eq} and with parts (6) and (10) of Proposition \ref{pfi}, we obtain
\begin{equation} \label{Fitt-i-eq}
\begin{aligned}
\Fitt_i(M) &=\Fitt_i\Biggl(\bigoplus_{j=1}^n R\big/\bigl(\pi^{k_j}\bigr)\!\Biggr)=\sum_{\substack{(s_1,\dots,s_n)\in\N^n\\s_1 + \dots + s_n=i}}\prod_{j=1}^n \Fitt_{s_j}\Bigl(R\big/\bigl(\pi^{k_j}\bigr)\!\Bigr)\\
&=\sum_{\substack{(s_1,\dots,s_n)\in\N^n\\s_1 + \dots + s_n=i}}\prod_{\substack{j\in\{1,\dots,n\}\\s_j=0}} (\pi^{k_j})=(\pi^{c_i}),
\end{aligned}
\end{equation}
as desired, where the last equality in \eqref{Fitt-i-eq} follows from Definition \ref{fitting-def} using the fact that the integers $k_j$ are ordered so that $1\leq k_1\leq \dots\leq k_n$. \end{proof}

\subsection{Fitting ideals of $\Lambda$-modules}\label{sec:FittingLambda}

Let $\mathcal{O}$ be a complete discrete valuation ring with finite residue field. The power series ring $\Lambda\defeq \mathcal{O} \llbracket T \rrbracket$ in one variable over $\mathcal{O}$ is a $2$-dimensional complete noetherian regular local ring. A finitely generated $\Lambda$-module $M$ is \emph{pseudo-null} if $M_\mathfrak{p}=0$ for all prime ideals $\mathfrak{p}$ of $\Lambda$ of height at most $1$; equivalently, $M$ is pseudo-null if and only if $M$ is finite (see, \emph{e.g.}, \cite[Remark 4, p. 269]{NSW}). Two finitely generated $\Lambda$-modules $M$ and $N$ are \emph{pseudo-isomorphic} if there exists a \emph{pseudo-isomorphism} between them, \emph{i.e.}, a homomorphism of $\Lambda$-modules $f:M\to N$ such that $\ker(f)$ and $\coker(f)$ are pseudo-null.

The following well-known result describes (up to pseudo-isomorphism) the structure of finitely generated torsion $\Lambda$-modules. 

\begin{theorem} \label{pseudo-thm}
Let $M$ be a non-trivial finitely generated torsion $\Lambda$-module. Then $M$ is pseudo-isomorphic to a $\Lambda$-module of the form
\[
\bigoplus_{j=1}^n \bigoplus_{t=1}^{\ell} \Lambda\big/ \mathfrak{p}_j^{k_{j,t}},
\]
where $\mathfrak{p}_j$ is a height $1$ prime ideal of $\Lambda$ for all $j$ and $k_{j,t} \geq 0$ are integers. All these ideals and integers are uniquely determined up to order.
\end{theorem}

\begin{proof} See, \emph{e.g.}, \cite[(5.1.10)]{NSW}. \end{proof}

From here on, given ideals $I$ and $J$ of $\Lambda$, we write $I \prec J$ if there is an ideal $\mathfrak{a}$ of height $2$ of $\Lambda$ such that $\mathfrak{a}I \subset J$; we write $I \sim J$ if $I \prec J$ and $J \prec I$. It is immediate to see that $\sim$ is an equivalence relation on the set of ideals of $\Lambda$. Furthermore, $\sim$ respects sums and products of ideals: if $I_r\sim J_r$ for $r\in\{1,2\}$, then $I_1+I_2\sim J_1+J_2$ and $I_1I_2\sim J_1J_2$. In what follows, we denote by $[I]$ the equivalence class of the ideal $I$ of $\Lambda$ with respect to $\sim$.

\begin{proposition}\label{prop-prec}
Let $I$ and $J$ be ideals of $\Lambda$. Then $I \prec J$ if and only if
\[ 
I \Lambda_\mathfrak{P} \subset J \Lambda_\mathfrak{P}
\]
for every height $1$ prime ideal $\mathfrak{P}$ of $\Lambda$.
\end{proposition}

\begin{proof} Suppose that $I \prec J$. Then $\mathfrak{a}I \subset J$ with $\mathfrak{a}$ an ideal of height $2$ of $\Lambda$. 
Since $\mathfrak{a}$ has height $2$, for any prime ideal $\mathfrak{P}$ 
of of height $1$ we have $(\Lambda/\mathfrak{a})_\mathfrak{P}=0$, and therefore $\mathfrak{a}_\mathfrak{P}=\Lambda_\mathfrak{P}$. 
Hence, given a height $1$ prime ideal $\mathfrak{P}$ of $\Lambda$, one has
\[
I \Lambda_\mathfrak{P} =\mathfrak{a}I \Lambda_\mathfrak{P} \subset J \Lambda_\mathfrak{P}.
\]
Conversely, assume that $I \Lambda_\mathfrak{P} \subset J \Lambda_\mathfrak{P}$ for every height $1$ prime ideal $\mathfrak{P}$ of $\Lambda$ and define 
\[
\mathfrak{a}\defeq(J:I)=\bigl\{\lambda\in\Lambda\mid\lambda I\subset J\bigr\},
\]
which is an ideal of $\Lambda$. Let $\mathfrak m$ be the maximal ideal of $\Lambda$, which has height $2$. If $\mathfrak{a}=\Lambda$, then $\mathfrak{m} I \subset J$, whence $I\prec J$. Otherwise, thanks to the inclusion $\mathfrak{a}I \subset J$, we just need to check that the height of $\mathfrak{a}$ is $2$. Let $\mathfrak{P}$ be a prime ideal of height $1$ of $\Lambda$. Choose a system of generators $\{i_1,\dots,i_n\}$ of $I$. Since $I \Lambda_\mathfrak{P} \subset J \Lambda_\mathfrak{P}$, there exist $y_j \in \Lambda \smallsetminus \mathfrak{P}$ for $j\in\{1,\dots,n\}$ such that $y_j i_j \in J $ for each $j$. Therefore, setting $y_\mathfrak{P}\defeq y_1 \cdot \ldots \cdot y_n$, we have $iy_\mathfrak{P} \in J$ for all $i\in I$, so $y_\mathfrak{P}\in\mathfrak a$. But $y_\mathfrak{P}\notin\mathfrak{P}$, which shows that $\mathfrak a\not\subset\mathfrak P$. It follows that $\mathfrak{a}$ is a non-trivial (because $y_\mathfrak{P}\neq 0$ for some $\mathfrak{P}$ as before) ideal of $\Lambda$ that is not contained in any height $1$ prime ideal $\mathfrak{P}$ of $\Lambda$. We conclude that $\mathfrak a$ has height $2$. \end{proof}

In the proof of the following result, where we generalize the arguments in the proof of \cite[Lemma 9.2]{KU2012}, we implicitly exploit Theorem \ref{pseudo-thm}.

\begin{theorem} \label{Thm-pseudo}
Let $M$ be a non-trivial finitely generated torsion $\Lambda$-module. The pseudo-isomorphism class of $M$ is uniquely determined by the classes modulo $\sim$ of the ideals $\Fitt_i(M)$ for $i\geq0$.
\end{theorem}

\begin{proof} We divide the proof into four steps.

\texttt{Step 1.} If $M$ and $N$ are two pseudo-isomorphic $\Lambda$-modules, then $\Fitt_i(M)\sim\Fitt_i(N)$ for all $i\geq0$. To see this, observe that, since $M$ and $N$ are pseudo-isomorphic, there are exact sequences
\[ M \longrightarrow N \longrightarrow B \longrightarrow 0,\quad N \longrightarrow M \longrightarrow B' \longrightarrow 0,
\]
where $B$ and $B'$ are pseudo-null $\Lambda$-modules. By part (8) of Proposition \ref{pfi}, we deduce that for all $i\geq0$ there are inclusions
\[
\Fitt_i(M)\Fitt_0(B)\subset\Fitt_i(N),\quad\Fitt_i(N)\Fitt_0(B')\subset\Fitt_i(M).
\]
On the other hand, by part (9) of Proposition \ref{pfi}, the ideals $\Fitt_0(B)$ and $\Fitt_0(B')$ have height at least $2$, and the claim follows from Proposition \ref{prop-prec}.

\texttt{Step 2.}
Let  ${\{\mathfrak{p}_j\}}_{j=1,\dots,n}$ be a set of height $1$ prime ideals of $\Lambda$ and let $k_{j,t} \geq 0$ be integers for $j\in\{1, \dots,n\}$ and $t\in\{1, \dots,\ell\}$. Assume that, for each $j$, the sequence ${(k_{j,t})}_{t=1,\dots,\ell}$ is non-decreasing. Then
\begin{equation} \label{fitt-prod-eq}
\Fitt_i\Biggl(\bigoplus_{j=1}^n \bigoplus_{t=1}^{\ell} \Lambda\big/\mathfrak{p}_j^{k_{j,t}} \Biggr)\sim\prod_{j=1}^{n} \prod_{t=1}^{\ell-i}\mathfrak{p}_j^{k_{j,t}}
\end{equation}
for all $i\in\{0,\dots,\ell\}$. To prove \eqref{fitt-prod-eq}, set $M_j\defeq\bigoplus_{t=1}^{\ell} \Lambda\big/\mathfrak{p}_j^{k_{j,t}}$. By the argument in the proof of Theorem \ref{FittingDVR}, 
we have
\[
\Fitt_i(M_j) =\prod_{t=1}^{\ell-i} \mathfrak{p}_j^{k_{j,t}}
\]
for all $j$. Now, using part (10) of Proposition \ref{pfi}, one has
\[
\begin{aligned}
\Fitt_i\Biggl(\bigoplus_{j=1}^n M_j\Biggr) &= \sum_{\substack{(s_1,\dots,s_n)\in\N^n\\s_1 + \dots + s_n=i}} \prod_{j=1}^n \Fitt_{s_j}(M_j)=\sum_{\substack{(s_1,\dots,s_n)\in\N^n\\s_1 + \dots + s_n=i}} \prod_{j=1}^n \prod_{t=1}^{\ell-s_j}\mathfrak{p}_j^{k_{j,t}}\\
& = \left(\sum_{\substack{(s_1,\dots,s_n)\in\N^n\\s_1 + \dots + s_n=i}} \prod_{j=1}^n \prod_{t=\ell-i+1}^{\ell-s_j}\mathfrak{p}_j^{k_{j,t}}\right)\cdot\prod_{j=1}^n \prod_{t=1}^{\ell-i} \mathfrak{p}_j^{k_{j,t}},
\end{aligned}
\]
where by convention we understand that the ideal $\prod_{t=\ell-i+1}^{\ell-s_j}\mathfrak{p}_j^{k_{j,t}}$ is equal to $\Lambda$ when $s_j=i$. We observe that the ideal in brackets contains an ideal of height at least $2$, which (in light of the definition of $\prec$) concludes this step. Indeed, if $n=1$, then we have $s_1=i$ and so the ideal in brackets is equal to $\Lambda$, whose maximal ideal has height $2$. If $n>1$, then the summand corresponding to the choice $s_j=i$ is not divisible by $\mathfrak{p}_j$, so the sum of these varying $j$ is not contained in any prime ideal of height $1$ of $\Lambda$.  

\texttt{Step 3.} Let ${\{\mathfrak{p}_j\}}_{j=1,\dots,n}$ be a family of height $1$ prime ideals of $\Lambda$ and let $k_j,r_j \geq 0$ be integers for any $j$. We want to prove the implication
\[
\prod_{j=1}^n \mathfrak{p}_j^{k_j}\sim\prod_{j=1}^n \mathfrak{p}_j^{r_j}\;\Longrightarrow\;\text{$k_j=r_j$ for all $j$}.
\]
Arguing by contradiction, without loss of generality we can assume $k_1>r_1$. By definition of $\sim$, there is a height $2$ ideal $\mathfrak{a}$ of $\Lambda$ such that
\begin{equation} \label{inclusions-eq}
\mathfrak{a} \prod_{j=1}^n\mathfrak{p}_j^{r_j}\subset\prod_{j=1}^n \mathfrak{p}_j^{k_j} \subset\mathfrak{p}_1^{k_1}. 
\end{equation}
For all $j$, let $p_j$ be a generator of $\mathfrak{p}_j$; pick $x \in \mathfrak{a} \smallsetminus \mathfrak{p}_1$ (such an $x$ exists since $\mathfrak{a}$ has height $2$). Thanks to inclusions \eqref{inclusions-eq}, there is $\lambda\in\Lambda$ such that $x \prod_{j=1}^n p_j^{r_j}=\lambda p_1^{k_1}$, which yields an equality 
\begin{equation} \label{x-eq}
x\prod_{j=2}^n p_j^{r,j}=\lambda p_1^{k_1-r_1}.
\end{equation}
In particular, the left-hand side in equality \eqref{x-eq} lies in $\mathfrak{p}_1$, which is a contradiction. 

\texttt{Step 4.} Eventually, the theorem follows immediately from the three steps above. 
If $M$ is pseudo-isomorphic to the 
$\Lambda$-module $N=\bigoplus_{j=1}^n \bigoplus_{t=1}^{\ell} \Lambda\big/ \mathfrak{p}_j^{k_{j,t}}$ 
(for suitable height $1$ prime ideals $\mathfrak{p}_j$ and integers $k_{j,t}\geq 0$), 
then the classes modulo $\sim$ of the Fitting ideals $\Fitt_i(M)$ and $\Fitt_i(N)$ of 
$M$ and $N$ respectively are the same, by \texttt{Step 1}. To conclude the proof, suppose that the two $\Lambda$-modules 
$M=\bigoplus_{j=1}^n \bigoplus_{t=1}^{\ell} \Lambda\big/ \mathfrak{p}_j^{k_{j,t}}$
and $N=\bigoplus_{j=1}^{n} \bigoplus_{t=1}^{\ell} \Lambda\big/ \mathfrak{p}_{j}^{r_{j,t}}$ 
(for suitable height $1$ prime ideals $\mathfrak{p}_j$ 
and integers $k_{j,t}\geq 0$ and $r_{j,t}\geq 0$) have Fitting ideals $\Fitt_i(M)$ 
and $\Fitt_i(N)$ in the same $\sim$ class. By  
\texttt{Step 2}, we have $\Fitt_i(M)\sim \prod_{j=1}^{n} \prod_{t=1}^{\ell-i}\mathfrak{p}_j^{k_{j,t}}$ 
and $\Fitt_i(N)\sim \prod_{j=1}^{n} \prod_{t=1}^{\ell-i}\mathfrak{p}_{j}^{r_{j,t}}$; finally, since $\Fitt_i(M)\sim \Fitt_i(N)$ by assumption, we see that $\prod_{j=1}^{n}\prod_{t=1}^{\ell-i}\mathfrak{p}_j^{k_{j,t}}\sim \prod_{j=1}^{n} \prod_{t=1}^{\ell-i}\mathfrak{p}_{j}^{r_{j,t}}$, and the conclusion follows from \texttt{Step 3}. \end{proof}

Now we need two auxiliary results in commutative algebra.

\begin{lemma} \label{princ-lemma}
Let $I$ be a non-trivial ideal of $\Lambda$. Then there exists a principal ideal $(f)$ of $\Lambda$ such that $I \sim (f)$. Furthermore, this principal representative of $[I]$ is unique: if $(g) \subset \Lambda$ satisfies $I \sim (g)$, then $(f) = (g)$.
\end{lemma}

\begin{proof} Since $\Lambda$ is a unique factorization noetherian domain, every height $1$ prime ideal of $\Lambda$ is principal (see, \emph{e.g.}, \cite[Theorem 20.1]{matsumura}). Moreover, $I$ is contained in at most finitely many height $1$ prime ideals of $\Lambda$ (see, \emph{e.g.}, \cite[Theorem 6.5, (i)]{matsumura} applied to $M=\Lambda/I$), which we call $\p_1,\dots,\p_n$. For each $i\in\{1,\dots,n\}$, fix $p_i\in\Lambda$ such that $\mathfrak{p}_i=(p_i)$.

For each $i$, the localization $\Lambda_{\p_i}$ is a discrete valuation ring, so the extended ideal $I\Lambda_{\p_i}$ is a power of the maximal ideal of $\Lambda_{\p_i}$. Thus, for each $i$ there is an integer $k_i \ge 1$ such that $I\Lambda_{\p_i} = p_i^{k_i}\Lambda_{\p_i}$. Now set 
\[
f\defeq\prod_{i=1}^n p_i^{k_i}\in\Lambda
\]
and consider the principal ideal $(f)$: we claim that $I \sim (f)$. 

To verify this equivalence, let $\mathfrak{q}$ be any height $1$ prime ideal of $\Lambda$. If $\mathfrak{q}=\p_i$ for some $i\in\{1,\dots,n\}$, then $p_j\notin\p_i$ for $j\neq i$, so $p_j$ is invertible in $\Lambda_{\mathfrak{p}_i}$ for $j\neq i$. Therefore, $(f)\Lambda_{\mathfrak q}=(f)\Lambda_{\p_i}=p_i^{k_i}\Lambda_{\p_i}=I\Lambda_{\p_i}=I\Lambda_{\mathfrak q}$. If $\mathfrak{q}\notin\{\p_1,\dots,\p_n\}$, then $I \not\subset\mathfrak{q}$, which implies $I\Lambda_{\mathfrak{q}} = \Lambda_{\mathfrak{q}}$; similarly, no prime factor $p_i$ of $f$ belongs to $\mathfrak{q}$, so $(f)\Lambda_{\mathfrak{q}}=\Lambda_{\mathfrak{q}}$. Since $I\Lambda_{\mathfrak{q}} = (f)\Lambda_{\mathfrak{q}}$ for every height $1$ prime ideal $\mathfrak{q}$ of $\Lambda$, Proposition \ref{prop-prec} allows us to conclude that $I \sim (f)$.

To establish uniqueness, suppose $(g)$ is a principal ideal of $\Lambda$ such that $I \sim (g)$. By transitivity, $(f) \sim (g)$, \emph{i.e.}, $(f)\Lambda_{\mathfrak{p}} = (g)\Lambda_{\mathfrak{p}}$ for all height $1$ prime ideals $\mathfrak{p}$ of $\Lambda$. Since $\Lambda$ is a Krull domain, every principal ideal of $\Lambda$ is uniquely determined by its localizations at height $1$ prime ideals (see, \emph{e.g.}, \cite[Theorem 12.3]{matsumura}). Therefore, there are equalities
\[
(f)=\bigcap_{\mathrm{ht}(\mathfrak{p})=1} (f)\Lambda_{\mathfrak{p}} = \bigcap_{\mathrm{ht}(\mathfrak{p})=1} (g)\Lambda_{\mathfrak{p}}=(g),
\]
as desired. \end{proof}

\begin{remark} \label{princ-rem}
Lemma \ref{princ-lemma} allows us to identify the monoid of equivalence classes of non-zero ideals of $\Lambda$ under $\sim$ with the monoid of principal ideals of $\Lambda$. In particular, if an ideal $J$ of $\Lambda$ satisfies $J\sim I^2$ for some ideal $I$, then the unique principal representative of $[J]$ is a square.
\end{remark}


\begin{proposition} \label{square-prop}
Let $J$ be a non-zero ideal of $\Lambda$ and suppose that there exists a non-zero ideal $I$ of $\Lambda$ such that $J\sim I^2$. Then $[I]$ is uniquely determined by $J$.
\end{proposition}

\begin{proof} Let $I_1$ and $I_2$ be non-zero ideals of $\Lambda$ such that $J\sim I_1^2$ and $J\sim I_2^2$: we wish to show that $[I_1]=[I_2]$, \emph{i.e.}, $I_1\sim I_2$. By Lemma \ref{princ-lemma}, the classes $[I_1]$ and $[I_2]$ admit (unique) principal representatives, which we denote by $(f_1)$ and $(f_2)$, respectively. It follows that $I_1^2\sim(f_1^2)$ and $I_2^2\sim(f_2^2)$. By transitivity, we get $(f_1^2)\sim(f_2^2)$, and then the uniqueness part of Lemma \ref{princ-lemma} yields the equality
\begin{equation} \label{square-princ-eq}
(f_1^2)=(f_2^2).    
\end{equation} 
Since $\Lambda$ is a domain, equality \eqref{square-princ-eq} implies that there is $u\in\Lambda^\times$ such that $f_1^2=u\cdot f_2^2$. Furthermore, unique factorization in $\Lambda$ shows that there is $v\in\Lambda^\times$ such that $f_1=v\cdot f_2$, whence $(f_1)=(f_2)$. It follows immediately that $I_1\sim(f_1)=(f_2)\sim I_2$, \emph{i.e.}, $[I_1]=[I_2]$. \end{proof}

In light of Proposition \ref{square-prop}, we can introduce the following notion.

\begin{definition} \label{pseudo-def}
A non-zero ideal $J$ of $\Lambda$ \emph{admits a pseudo-square root} if there is an ideal $I$ of $\Lambda$ such that $J\sim I^2$. In this case, the \emph{pseudo-square root} of $J$ is the (uniquely determined) equivalence class $\sqrt{[J]}\defeq[I]$.
\end{definition}

Equivalently, if $(f^2)$ is the unique principal representative of $[J]$, then $\sqrt{[J]}=[(f)]$ (\emph{cf.} Remark \ref{princ-rem}).

\begin{remark} \label{pseudo-rem} (1) If $I$ is a non-zero ideal of $\Lambda$, then $\sqrt{[I^2]}=[I]$.

(2) Let $I$ and $J$ be non-zero ideals of $\Lambda$ admitting pseudo-square roots. If $I\sim J$, then $\sqrt{[I]}=\sqrt{[J]}$.
\end{remark}

The next result, which describes odd Fitting ideals of certain torsion $\Lambda$-modules in terms of pseudo-square roots (in the sense of Definition \ref{pseudo-def}) of products of even Fitting ideals, will play a crucial (albeit often tacit) role in the rest of the paper.

\begin{proposition} \label{even-odd-prop}
Let $M$ be a non-trivial finitely generated torsion $\Lambda$-module and set $Y\defeq M\oplus M$. The equality
\[ 
\bigl[\Fitt_{2i+1}(Y)\bigr]=\sqrt{\bigl[\Fitt_{2i}(Y)\cdot\Fitt_{2i+2}(Y)\bigr]} 
\]
holds for all integers $i\geq0$. In particular, the classes modulo $\sim$ of the odd Fitting ideals of $Y$ are uniquely determined by the even Fitting ideals of $Y$.
\end{proposition}

\begin{proof} Let $\p$ be a height $1$ prime ideal of $\Lambda$; the localization $\Lambda_\p$ is a discrete valuation ring, and then a direct computation using Theorem \ref{FittingDVR} shows that
\begin{enumerate}
\item $\Fitt_{2i}(Y_\p)=\Fitt_i(M_\p)^2$ for all $i\geq0$;
\item $\Fitt_{2i+1}(Y_\p)=\Fitt_i(M_\p)\cdot\Fitt_{i+1}(M_\p)$ for all $i\geq0$.
\end{enumerate}
Squaring both sides of (2) and applying (1), for all $i\geq0$ we get an equality
\begin{equation} \label{fitt-loc-eq}
\Fitt_{2i+1}(Y_\p)^2=\Fitt_{2i}(Y_\p)\cdot\Fitt_{2i+2}(Y_\p)
\end{equation}
of ideals of $\Lambda_\p$. On the other hand, by part (4) of Proposition \ref{pfi}, Fitting ideals commute with localization, so equality \eqref{fitt-loc-eq} can be rewritten as
\begin{equation} \label{fitt-loc-eq2}
\Fitt_{2i+1}(Y)^2\Lambda_\p=\bigl(\Fitt_{2i}(Y)\cdot{\Fitt_{2i+2}(Y)}\bigr)\Lambda_\p.
\end{equation}
Since equality \eqref{fitt-loc-eq2} holds for all height $1$ prime ideals $\p$ of $\Lambda$, Proposition \ref{prop-prec} shows that 
\[ 
\Fitt_{2i+1}(Y)^2\sim\Fitt_{2i}(Y)\cdot\Fitt_{2i+2}(Y) 
\]
for all integers $i\geq0$, as was to be shown. \end{proof}

We immediately obtain

\begin{corollary} \label{pseudo-M-coro}
Let $M$ be a non-trivial finitely generated torsion $\Lambda$-module. The pseudo-isomorphism class of $M\oplus M$ is uniquely determined by the classes modulo $\sim$ of the ideals $\Fitt_i(M)$ for even $i\geq0$.
\end{corollary}

\begin{proof} Combine Theorem \ref{Thm-pseudo} and Proposition \ref{even-odd-prop}. \end{proof}
 
\section{Bipartite Euler systems} \label{sec::3}

Starting from the definitions given in \cite{howard2012bipartite}, in this section we develop a general theory of bipartite Euler systems and use it to study the structure of certain Selmer groups; in doing so, we follow arguments from \cite{kim2024higher}. Let us fix throughout an imaginary quadratic field $K$ of discriminant $D_K$, choose an algebraic closure $\overline K$ of $K$ and set $G_K\defeq\Gal(\overline{K}/K)$ for the absolute Galois group of $K$.

\subsection{Bipartite Euler systems} \label{bipartite}

Let $R$ be a complete noetherian local ring with maximal ideal $\mathfrak{m}$ and finite residue field of characteristic $p\nmid D_K$; moreover, fix a free module $T$ of rank $2$ over $R$ endowed with a continuous, $R$-linear action of $G_K$. For every prime $\mathfrak{l}$ of $K$, set
\[
H^1_{\unr}(K_\mathfrak{l},T)\defeq\ker\Bigl(H^1(K_\mathfrak{l},T) \longrightarrow H^1(I_\mathfrak{l},T)\Bigr).
\]

\begin{definition}\label{SelStructureT}
(1) A \emph{Selmer structure} for $T$ over $R$ is an ordered pair $(\mathcal{F}, \Sigma_\mathcal{F})$ where
\begin{itemize}
\item $\Sigma_\mathcal{F}$ is a finite set of places of $K$ containing the primes over $p$, the ramification locus of $T$ and the archimedean primes;
\item $\mathcal{F}$ is the choice, for each prime $\mathfrak{l}$ of $K$, of a subgroup
\[ 
H^1_\mathcal{F}(K_\mathfrak{l},T) \subset H^1(K_\mathfrak{l},T) 
\]
such that $H^1_\mathcal{F}(K_\mathfrak{l},T)=H^1_{\unr}(K_\mathfrak{l},T)$ for all $\mathfrak{l} \notin \Sigma_\mathcal{F}$.
\end{itemize}

(2) The \emph{Selmer group} attached to a Selmer structure $(\mathcal{F},\Sigma_\mathcal{F})$ is
\[
\Sel_\mathcal{F}(K,T)\defeq\ker\Biggl(H^1(K,T) \longrightarrow \prod_{\mathfrak{l}} \frac{H^1(K_\mathfrak{l},T)}{H^1_\mathcal{F}(K_\mathfrak{l},T)}\Biggr),
\]
the map in brackets being the obvious one.

(3) Given an $R$-bilinear pairing $(\cdot,\cdot):T \times T \rightarrow R(1)$, a Selmer structure $(\mathcal{F},\Sigma_\mathcal{F})$ for $T$ is $\emph{self-dual}$ with respect to $(\cdot,\cdot)$ if $H^1_\mathcal{F}(K_v,T)$ is maximal isotropic under the induced local Tate pairing for each finite $v\in\Sigma_\mathcal{F}$.
\end{definition}

From now on, assume that $T$ is equipped with a self-dual Selmer structure $(\mathcal{F},\Sigma_\mathcal{F})$.

\begin{definition}\label{admissibleprimes}
A prime number $\ell$ is \emph{admissible} for $T$ with respect to $\mathcal{F}$ if
\begin{itemize}
\item $\ell\not=p$;
\item $\ell$ is inert in $K$;
\item $\ell^2-1 \in R^\times$;
\item $T$ is unramified at $\ell$;
\item $\ell\mathcal O_K \notin \Sigma_\mathcal{F}$.
\end{itemize}
\end{definition}

Let $\mathcal{P}^\mathrm{adm}$ denote the set of admissible primes. Fix a subset $\mathcal{P}\subset\mathcal{P}^\mathrm{adm}$ and denote by $\mathcal{N}=\mathcal N(\mathcal P)$ the set of square-free products of elements of $\mathcal{P}$. For an $n \in \mathcal{N}$, we write $\nu(n)$ for the number of prime divisors of $n$. We include $1$ in $\mathcal{N}$ and set $\nu(1)\defeq0$. For each $\ell \in \mathcal{P}$, the characteristic polynomial of $\Frob_\ell$ is $P_\ell(X)\defeq\det(X-\Frob_\ell|T)\in R[X]$; write $P_\ell(X)=X^2- t_\ell X+d_\ell$. For $\ell\in\mathcal P$ and $n \in \mathcal{N}$, define the ideals
\begin{equation} \label{I-eq}
I_\ell\defeq\bigl(t_\ell-(\ell^2+1), d_\ell-\ell^2\bigr)\subset R,\quad I_n\defeq\sum_{\ell|n}I_\ell.
\end{equation}
Finally, for $\epsilon\in\Z/2\Z$ set
\[
\mathcal{N}^\epsilon\defeq\bigl\{n \in \mathcal{N}\mid\nu(n)\equiv\epsilon\pmod2\bigr\}.
\]
Let $n \in \mathcal{N}$. For each prime $\ell\,|\,n$, since $T$ is unramified at $\ell$ and $P_\ell(X)$ splits modulo $I_n$ as $(X-\ell^2)(X-1)$, there is an isomorphism of $G_{K_\ell}$-modules  
\[
T/I_nT \simeq R/I_n \oplus R/I_n(1).
\]
From now on, for any integer $r\geq1$ denote by $\Bmu_r$ the group of $r$-th roots of unity (in a suitable field that will be clear from context).

For lack of a convenient reference, we explicitly recall the following standard result. 

\begin{lemma}\label{lemmaKummer} 
Let $S$ be a complete noetherian local ring with finite residue field. The $S$-module $H^1\bigl(K_\ell,S(1)\bigr)$ is free of rank $1$. 
\end{lemma}

\begin{proof} Write $\mathfrak m$ for the maximal ideal of $S$ and $p$ for its residual characteristic; we divide the proof into two steps.

\texttt{Step 1.} We first assume that $S$ is finite. If $S\simeq \Z/p^s\Z$, then $S(1)\simeq\Bmu_{p^s}$, where $\Bmu_{p^s}$ is the $G_{K_\ell}$-module 
of $p^s$-roots of unity in $\overline{K}_\ell$. By Kummer theory, there is an isomorphism  
\[
H^1(K_\ell,\Bmu_{p^s})\simeq K_\ell^\times\big/{(K_\ell^\times)}^{p^s};
\]
moreover, since $K_\ell^\times \simeq \ell^\Z \times\Bmu_{(\ell^2-1)} \times \bigl(1+\mathfrak{m}_{\mathcal{O}_{K_\ell}}\bigr)$, we have ${(K_\ell^\times)}^{p^s}\simeq \ell^{p^s\Z} \times\Bmu_{(\ell^2-1)} \times\bigl(1+\mathfrak{m}_{\mathcal{O}_{K_\ell}}\bigr)$ because $p\nmid \ell^2-1$, so $K_\ell^\times\big/\bigl(K_\ell^\times\bigr)^{p^s}\simeq\Z/p^s\Z$. If $S$ is finite, then it is isomorphic (as a group) to a direct sum of cyclic groups, and the result follows in this case. 

\texttt{Step 2.} Now we prove the result for a general $S$. For every integer $n\geq1$, the $S$-module $S_n=S/\mathfrak{m}^n$ is finite, so by the previous 
step there is an isomorphism $H^1(K_\ell,S_n(1))\simeq S_n$. Since $S$ is complete, there is an isomorphism $S\simeq\varprojlim_nS_n$, and then a theorem of Tate (\cite[Proposition B.2.3]{RubinES}) ensures that there is an isomorphism
\[
S\simeq\varprojlim_nH^1\bigl(K_\ell,S_n(1)\bigr)\simeq H^1\bigl(K_\ell,\textstyle{\varprojlim_n}S_n(1)\bigr)=H^1\bigl(K_\ell,S(1)\bigr),
\]
where the isomorphism of $G_{K_\ell}$-modules $\varprojlim_nS_n(1)\simeq S(1)$ is clear from the isomorphism 
of $S$-modules $\varprojlim_nS_n\simeq S$. 
\end{proof}

Define 
\[
H^1_{\ord}(K_\ell,T/I_nT)\defeq\im\Bigl(H^1\bigl(K_\ell,R/I_n(1)\bigr)\longrightarrow H^1(K_\ell,T/I_nT)\Bigr). 
\]


\begin{proposition} \label{freeness-prop}
Let $n \in \mathcal{N}$ and let $\ell$ be a prime divisor of $n$. There is a splitting
\[
H^1(K_\ell,T/I_nT)\simeq H^1_{\unr}(K_\ell,T/I_n T)\oplus  H^1_{\ord}(K_\ell,T/I_nT)
\]
in which each summand is free of rank $1$ over $R/I_n$.
\end{proposition}

\begin{proof} For notational convenience, set $S\defeq R/I_n$ and $M\defeq T/I_nT$. There is a splitting
\[
H^1(K_\ell,M) \simeq H^1(K_\ell,S) \oplus H^1\bigl(K_\ell,S(1)\bigr).
\]
Since $I_\ell$ is pro-$\ell$ and $S$ is pro-$p$, the group $H^1(I_\ell,S)=\Hom(I_\ell,S)$ is trivial and so
\[
H^1_\unr(K_\ell,S)=H^1(K_\ell,S).
\]
By \cite[Lemma 1.2.1]{MR}, we also have
\[
H^1_\unr\bigl(K_\ell,S(1)\bigr) \simeq S(1)/(\Frob_\ell -1) S(1)=S(1)/(\ell^2-1)S(1)=0,
\]
where the last equality follows from the fact that $\ell^2-1$ is invertible in $R$. Then
\[
H^1_{\unr}(K_\ell,M)\simeq H^1_{\unr}(K_\ell,S)\oplus H^1_\unr\bigl(K_\ell,S(1)\bigr)=H^1(K_\ell,S),
\]
where $H^1(K_\ell,S) \simeq S/(\Frob_\ell-1)S \simeq S$. On the other hand, since $M\simeq S\oplus S(1)$, the map $H^1\bigl(K_\ell,S(1)\bigr)\rightarrow H^1(K_\ell,M)$ is injective, so there is an isomorphism $H^1_{\ord}(K_\ell,M)\simeq H^1\bigl(K_\ell,S(1)\bigr)$. Finally, by Lemma \ref{lemmaKummer} the $S$-module $H^1\bigl(K_\ell,S(1)\bigr)$ is free of rank $1$, as was to be shown. \end{proof}

In light of Proposition \ref{freeness-prop}, from here on we fix isomorphisms of $R$-modules
\[
\phi_{\unr}^{(n)}:H^1_\unr(K_\ell,T/I_nT) \overset\simeq\longrightarrow R/I_n,\quad\phi_{\ord}^{(n)}: H^1_{\ord}(K_\ell,T/I_nT) \overset\simeq\longrightarrow R/I_n.
\]
Define a modified Selmer structure $\mathcal{F}(n)$ on $T/I_nT$ by setting
\begin{itemize}
\item $H^1_{\mathcal{F}(n)}(K_\ell,T/I_nT)\defeq H^1_\mathcal{F}(K_\ell,T/I_nT)$ if $\ell \nmid n$;
\item $H^1_{\mathcal{F}(n)}(K_\ell,T/I_nT)\defeq H^1_{\ord}(K_\ell,T/I_nT)$ if $\ell\,|\,n$.
\end{itemize}
Finally, for any $n\in\mathcal{N}$ and $\ell\in\mathcal{P}$ consider the projection
\[ 
\pi_{n,\ell}:R/I_n \longepi R/I_{n \ell}. 
\]
The following notion was first introduced in \cite{howard2012bipartite}.

\begin{definition}\label{defBES}
A \emph{bipartite Euler system} for $(T,\mathcal{F},\mathcal{P})$ of parity $\epsilon \in \Z/2\Z$ is a pair $(\kappa,\lambda)$ of families
\[ 
\kappa=\Bigl\{\kappa_n \in\Sel_{\mathcal{F}(n)}(K,T/I_n T)\;\Big|\;n \in \mathcal{N}^\epsilon\Bigr\},\quad\lambda=\Bigl\{\lambda_n \in  R/I_n\;\Big|\;n \in \mathcal{N}^{\epsilon+1}\Bigr\}
\]
satisfying the following two properties.
\begin{enumerate}
\item \texttt{First explicit reciprocity law}: for $n\in\mathcal{N}^{\epsilon+1}$ and $\ell \in \mathcal{P}$ such that $\ell\nmid n$, there is an equality  
\[
\bigl(\pi_{n,\ell}(\lambda_n)\bigr)=\Bigl(\phi_{\ord}^{(n\ell)}\bigl(\loc_\ell (\kappa_{n\ell})\bigr)\!\Bigr)
\]
of ideals of $R/I_{n\ell}$.
\item \texttt{Second explicit reciprocity law}: for $n \in \mathcal{N}^\epsilon$ and $\ell \in \mathcal{P}$ such that $\ell\nmid n$, there is an equality 
\[
(\lambda_{n\ell})=\biggl(\pi_{n,\ell}\Bigl(\phi_{\unr}^{(n)}\bigl(\loc_\ell(\kappa_n)\bigr)\!\Bigr)\!\!\biggr)
\]
of ideals of $R/I_{n\ell}$.
\end{enumerate}
\end{definition}

\begin{remark}
The abstract formulation of the notion of a bipartite Euler system given in this section is proposed here for the first time; it is similar to that of \cite{kim2024higher} and can be seen as an analogue of the formalism for Kolyvagin systems in \cite{MR} (see also \cite{mastella2025anticyclotomic} for further details). 
\end{remark}

\subsection{Bipartite Euler systems over a principal artinian local ring} \label{BES-Artinian}

We study bipartite Euler systems over principal artinian local rings, mainly following 
\cite{howard2012bipartite} and \cite{kim2024higher}. The main result of this subsection is Theorem \ref{artSel}, which also appears in \cite[Theorem 4.19]{kim2024higher} in a similar form; for the reader's convenience, in our arguments we decided to include a certain number of details, following the approach of Mazur--Rubin (\cite[Proposition 4.5.8]{MR}), 
because some of the steps in the proof of \cite[Theorem 4.19]{kim2024higher} are only implicit or just sketched.

Let $R$ be a principal artinian local ring of length $k$ with maximal ideal $\mathfrak{m}$; let $\pi\in\mathfrak m$ be a uniformizer. For an $R$-module $M$, denote by $\length(M)$ the length of $M$ over $R$. 

We say that $\mathcal{F}$ is \emph{cartesian} if the isomorphism
\[
T/\mathfrak{m}^i T \overset\simeq\longrightarrow T[\mathfrak{m}^i],\quad \bar{x}\longmapsto \pi^{k-i}x
\]
induces an isomorphism
\[
H^1_\mathcal{F}(K_v,T/\mathfrak{m}^i T)\overset\simeq\longrightarrow H^1_\mathcal{F}\bigl(K_v,T[\mathfrak{m}^i]\bigr)
\]
for all $i\in\{1,\dots,k\}$ and all $v\in\Sigma_\mathcal{F}$.

In the list of assumptions that follows, $\tau\in G_\Q$ is a fixed complex conjugation.

\begin{assumption} \label{assArt}
\begin{enumerate}
\item $T/\mathfrak{m}T$ is absolutely irreducible;
\item $T$ admits a perfect and symmetric $R$-bilinear pairing
\[
(\cdot, \cdot):T \times T \longrightarrow R(1)
\]
such that $(x^\sigma,y^{\tau \sigma \tau})=(x,y)^\sigma$ for all $\sigma \in G_K$;
\item $\mathcal{F}$ is cartesian (\cite[\S1.2]{HoHeeg}) and self-dual;
\item for all $c \in H^1(K,T/\mathfrak{m}T) \smallsetminus \{0 \}$, there are infinitely many $\ell \in \mathcal{P}$ such that $\loc_\ell(c) \neq 0$; 
\item $I_\ell=0$ for all $\ell \in \mathcal{P}$.
\end{enumerate}
\end{assumption}

By \cite[Lemma 2.2.6]{howard2012bipartite}, $\mathcal{F}(n)$ is cartesian for all $n \in \mathcal{N}$.

\begin{proposition}[Howard] \label{eSel}
Let $n \in \mathcal{N}$. There exists $e(n)\in\{0,1\}$ such that
\[ \Sel_{\mathcal{F}(n)}(K,T) \simeq R^{e(n)} \oplus M_n \oplus M_n. \]
\end{proposition}

\begin{proof} This is \cite[Proposition 2.2.7]{howard2012bipartite}. \end{proof}

Define
\[ 
\mathcal{N}^{\ind}\defeq\bigl\{n\in\mathcal{N}\mid e(n)=1\bigr\},\quad\mathcal{N}^{\deff}\defeq\bigl\{n\in\mathcal{N}\mid e(n)=0\bigr\}.
\]
By \cite[Corollary 2.2.10]{howard2012bipartite}, there exists $\epsilon\in\Z/2\Z$ such that $\mathcal{N}^{\ind}=\mathcal{N}^\epsilon$ and $\mathcal{N}^{\deff}=\mathcal{N}^{\epsilon +1}$.

\begin{definition} \label{C-def}
A bipartite Euler system $(\kappa, \lambda)$ is \emph{free} if for each $n \in \mathcal{N}^{\ind}$ there is a free $R$-module $C_n \subset \Sel_{\mathcal{F}(n)}(K,T)$ of rank $1$ containing $\kappa_n$.
\end{definition}

\begin{remark}
With these definitions, $n\in \mathcal{N}^{\deff}$, \emph{i.e.}, $e(n)=0$, if and only if 
the parity of the number of primes dividing $n$ is \emph{odd} if $\epsilon=0$ and \emph{even} if $\epsilon=1$; on the other hand, $n\in \mathcal{N}^{\ind}$, \emph{i.e.}, $e(n)=1$, if and only if the parity of the number of primes dividing $n$ is \emph{even} if $\epsilon=0$ and \emph{odd} if $\epsilon=1$.  
\end{remark}

\begin{lemma} \label{lemmaind}
Let $n\in \mathcal{N}^{\ind}$ and suppose that $\ell\in\mathcal{P}$ satisfies the following conditions:
\begin{itemize}
\item $\ell\nmid n$;
\item the restriction of $\loc_\ell:\Sel_{\mathcal{F}(n)}(K,T) \rightarrow H^1_{\unr}(K_{\ell},T)$ to the $R$-component is bijective. 
\end{itemize}
Then there is an isomorphism $M_{n}\simeq M_{n\ell}$.  
\end{lemma}

\begin{proof} There is a short exact sequence
\begin{equation} \label{sel-eq}
0\longrightarrow \Sel_{\mathcal{F}_\ell(n)}(K,T)\longrightarrow \Sel_{\mathcal{F}(n)}(K,T)\longrightarrow H^1_\mathrm{unr}(K_\ell,T)\longrightarrow 0,
\end{equation}
where $\Sel_{\mathcal{F}_\ell(n)}(K,T)$ is the Selmer group $\Sel_{\mathcal{F}(n)}(K,T)$ restricted at $\ell$, \emph{i.e.}, the subgroup of those classes in $\Sel_{\mathcal{F}(n)}(K,T)$ having trivial localization at $\ell$. Then 
\[ 
\length\biggl(\frac{\Sel_{\mathcal{F}(n)}(K,T)}{\Sel_{\mathcal{F}_\ell(n)}(K,T)}\biggr)=
\length\bigl(H^1_{\unr}(K,T)\bigr)=\length(R)=k,
\]
so by \cite[Proposition 2.2.9]{howard2012bipartite} we get
\[
\length\biggl(\frac{\Sel_{\mathcal{F}(n\ell)}(K,T)}{\Sel_{\mathcal{F}_\ell(n)}(K,T)}\biggr)=0.
\] 
It follows that there is an isomorphism
\[ {\Sel_{\mathcal{F}(n\ell)}(K,T)}\simeq {\Sel_{\mathcal{F}_\ell(n)}(K,T)}, \]
so exact sequence \eqref{sel-eq} can be rewritten as  
\[ 0\longrightarrow \Sel_{\mathcal{F}(n\ell)}(K,T)\longrightarrow \Sel_{\mathcal{F}(n)}(K,T)\longrightarrow H^1_\mathrm{unr}(K_\ell,T)\longrightarrow 0.\]
Since $\Sel_{\mathcal{F}(n\ell)}(K,T)\simeq M_{n\ell}\oplus M_{n\ell}$
and $\Sel_{\mathcal{F}(n)}(K,T)\simeq R\oplus M_n\oplus M_n$, we conclude that $M_{n\ell}\simeq M_n$. \end{proof}

\begin{lemma} \label{lemmadef}
Let $n\in \mathcal{N}^{\deff}$ and assume that there is an isomorphism
\[
\Sel_{\mathcal{F}(n)}(K,T)\simeq\bigoplus_{i\geq0}\bigl(R/\mathfrak{m}^{d_i}\bigr)^2
\]
with ${(d_i)}_{i\geq0}$ a non-increasing sequence of non-negative integers such that $d_i=0$ for $i\gg0$. Let $\ell\in\mathcal{P}$ satisfy the following conditions:
\begin{itemize}
\item $\ell\nmid n$;
\item the restriction of
\[
\loc_\ell:\Sel_{\mathcal{F}(n)}(K,T) \longrightarrow H^1_{\unr}(K_{\ell},T)
\]
to a summand of maximal length is injective.
\end{itemize}
Then $\length(M_{n\ell})=\sum_{i\geq 2}d_i$ and there is an isomorphism
\[
\Sel_{\mathcal{F}(n\ell)}(K,T)\simeq R\oplus\bigoplus_{i\geq1}\bigl(R/\mathfrak{m}^{d_i}\bigr)^2.
\]
\end{lemma}

\begin{proof} In our current notation, the length of the summand of maximal length is $d_0$ and 
\[ 
\length\bigl(\Sel_{\mathcal{F}(n)}(K,T)\bigr)=d\defeq2\cdot\sum_{i\geq 0}d_i.
\]
There is a short exact sequence  
\[
0\longrightarrow \Sel_{\mathcal{F}_\ell(n)}(K,T)\longrightarrow \Sel_{\mathcal{F}(n)}(K,T)\longrightarrow H^1_{\unr}(K_\ell,T),
\] 
where, as before, $\Sel_{\mathcal{F}_\ell(n)}(K,T)$ denotes the restriction of $\Sel_{\mathcal{F}(n)}(K,T)$ at $\ell$ and the length of the image of the rightmost map is $d_0$ because $H^1_{\unr}(K_\ell,T)\simeq R$ and all summands of $\Sel_{\mathcal{F}(n)}(K,T)$ have lengths smaller than $d_0$. Then 
\[
\length\biggl(\frac{\Sel_{\mathcal{F}(n)}(K,T)}{\Sel_{\mathcal{F}_\ell(n)}(K,T)}\biggr)= d_0,
\]
so $\length\bigl(\Sel_{\mathcal{F}_\ell(n)}(K,T)\bigr)=d-d_0$. Since $\loc_\ell:\Sel_{\mathcal{F}(n)}(K,T)\rightarrow H^1_{\unr}(K,T)$ is an isomorphism 
on one of the two copies of $R/\mathfrak{m}^{d_0}$, and any proper submodule 
$\bigoplus_{i\geq 1}(R/\mathfrak{m}^{d_i})^2\oplus (R/\mathfrak{m}^{d_0})$ has length strictly smaller than $d-d_0$, it follows that there is an isomorphism
\[
\Sel_{\mathcal{F}_\ell(n)}(K,T)\simeq\bigoplus_{i\geq1}\bigl(R/\mathfrak{m}^{d_i}\bigr)^2\oplus (R/\mathfrak{m}^{d_0}).
\]
Moreover, by \cite[Proposition 2.2.9]{howard2012bipartite} there is an equality 
\[
\length\biggl(\frac{\Sel_{\mathcal{F}(n\ell)}(K,T)}{\Sel_{\mathcal{F}_\ell(n)}(K,T)}\biggr)=k-d_0.
\]  
It follows that 
\[
\length\bigl(\Sel_{\mathcal{F}(n\ell)}(K,T)\bigr)=k+d-2d_0=k+\sum_{i\geq 1}d_i.
\] 
On the other hand, the isomorphism $\Sel_{\mathcal{F}(n\ell)}(K,T)\simeq R\oplus M_{n\ell}\oplus M_{n\ell}$ yields the equality 
\[
\length\bigl(\Sel_{\mathcal{F}(n\ell)}(K,T)\bigr)=k+2\length(M_{n\ell}),
\]
so $\length(M_{n\ell})=\sum_{i\geq 2}d_i$. Finally, $\Sel_{\mathcal{F}(n\ell)}(K,T)$ contains the submodule $\bigoplus_{i\geq 1}(R/\mathfrak{m}^{d_i})^2$ of $\Sel_{\mathcal{F}_\ell(n)}(K,T)$ and we conclude that $M_{n\ell}=\bigoplus_{i\geq 1}R/\mathfrak{m}^{d_i}$. \end{proof}
\color{black}

Given a bipartite Euler system $(\kappa,\lambda)$ and an integer $j\geq1$ with $j\equiv\epsilon+1\pmod2$, set
\begin{equation} \label{partiallambda}
\begin{split}
\partial^{(j)}(\lambda)&\defeq\min\Bigl\{\length\Bigl(R\big/\bigl(I_n+(\lambda_n)\bigr)\!\Bigr)\;\Big|\;\nu(n)=j\Bigr\},\\
\partial(\lambda)&\defeq\min\bigl\{\partial^{(j)}(\lambda)\mid j\equiv\epsilon+1\pmod2\bigr\}.
\end{split}
\end{equation}

\begin{theorem} \label{artSel}
Let $(\kappa, \lambda)$ be a non-trivial free bipartite Euler system of parity $\epsilon$. Assume that for some $e\in\{0,1\}$ there is an isomorphism of $R$-modules
\[
\Sel_{\mathcal{F}}(K,T) \simeq R^e \oplus\bigoplus_{i\geq0}\bigl(R/\mathfrak{m}^{d_i}\bigr)^2,
\]
with ${(d_i)}_{i\geq0}$ a non-increasing sequence of non-negative integers such that $d_i=0$ for $i \gg 0$. The equality
\[
\partial^{(j)}(\lambda)=\min\Biggl\{ k, \partial(\lambda) + \sum_{i \geq \frac{j-e}{2}} d_i\Biggr\}
\]
holds for all integers $j\geq1$ such that $j\equiv\epsilon+1\pmod 2$.
\end{theorem}

\begin{proof} We divide our arguments into several steps. 

\texttt{Step 1.} By \cite[Theorem 2.5.1]{howard2012bipartite} (see also \cite[Theorem 4.18]{kim2024higher}), we have
\[
\partial^{(j)}(\lambda)=\min\Bigl\{k, \partial (\lambda) + \length (M_n)\;\Big|\; j=\nu(n)\equiv\epsilon+1\pmod{2}\Bigr\}.
\]
To check this formula, following the notation in \cite[Section 2]{howard2012bipartite}, let $\ind(\lambda_n)$ be the index of divisibility of $\lambda_n$ in $R$, \emph{i.e.}, the
largest integer $b\leq k$ such that $\lambda_n\in\mathfrak{m}^b$ (\emph{cf.} \cite[Section 2]{howard2012bipartite}), and write $\delta$ for the minimum of $\mathrm{ind}(\lambda_n)$ for all $n\in\mathcal{N}^{\deff}$. Then there are equalities
\[
\begin{split}\delta&=\min\bigl\{\mathrm{ind}(\lambda_n)\mid n\in \mathcal{N}^{\deff}\bigr\}=\min\bigl\{k,\length(R/\lambda_nR)\mid n\in \mathcal{N}^{\deff}\bigr\}\\
&=\min_{j\equiv\epsilon+1\;(\text{mod $2$})}\Bigl\{\min\bigl\{k,\length(R/\lambda_nR)\mid\nu(n)=j, n\in\mathcal{N}^{\deff}\bigr\}\!\Bigr\}\\
&=\min_{j\equiv\epsilon+1\;(\text{mod $2$})}\Bigl\{\min\bigl\{k,\partial^{(j)}(\lambda_n)\mid\nu(n)=j\bigr\}\!\Bigr\}=\min\bigl\{k,\partial(\lambda)\bigr\}.
\end{split}
\]
For $n\in \mathcal{N}^{\deff}$, put 
\[
\mathrm{Stub}(n)\defeq\mathfrak{m}^{\mathrm{length}(M_n)}R.
\]
By \cite[Theorem 2.5.1]{howard2012bipartite}, $\lambda_n$ generates $\mathfrak{m}^\delta\cdot\mathrm{Stub}(n)$ for all $n\in\mathcal{N}^\mathrm{def}$. Therefore, there are equalities 
\[
\ind(\lambda_n)=\delta+\length(M_n)=\min\bigl\{k,\partial(\lambda)+\length(M_n)\bigr\},
\]
and then we get
\[
\begin{split}\partial^{(j)}(\lambda)&=\min\bigl\{\mathrm{ind}(\lambda_n,R)\mid\text{$\nu(n)=j$, $n\in \mathcal{N}^{\deff}$}\bigr\}\\
&=\min\bigl\{k,\partial(\lambda)+\length(M_n)\mid\text{$\nu(n)=j$, $n\in \mathcal{N}^{\deff}$}\bigr\}.
\end{split}
\]

\texttt{Step 2.} We prove the basis of the induction. If $e=0$, \emph{i.e.}, $1\in\mathcal{N}^{\deff}$ and $\epsilon=1$, then we have the equality
\[
\partial^{(0)}(\lambda)=\min\Biggl\{k, \partial(\lambda)+ \sum_{i\geq0} d_i\Biggr\}.
\]
If $e=1$, \emph{i.e.}, $1\in\mathcal{N}^{\ind}$ and $\epsilon=0$, then we choose $\ell \in \mathcal{P}$ such that the restriction of
\[
\loc_\ell:\Sel_{\mathcal{F}}(K,T) \longrightarrow H^1_{\unr}(K_{\ell},T)
\]
to the $R$-component is an isomorphism (such an $\ell$ exists by \cite[Lemma 2.3.3]{howard2012bipartite}; \emph{cf.} also \cite[Lemma 4.14]{kim2024higher}). 
By Lemma \ref{lemmaind}, $M_\ell\simeq M_1$. On the other hand, by \cite[Corollary 2.2.12]{howard2012bipartite} (\emph{cf.} also \cite[Corollary 4.9]{kim2024higher}) there is an inequality $\mathrm{length}(M_1)\leq \mathrm{length}(M_q)$ for all $q \in \mathcal{P}$, so $\length (M_\ell) \leq \length (M_q)$ for all $q \in \mathcal{P}$. In particular, we get
\[
\partial^{(1)}(\lambda)=\min \Biggl\{k,\partial(\lambda)+\sum_{i\geq0}d_i\Biggr\}.
\]

\texttt{Step 3.} We assume that the desired statement holds for $j-2$ and prove it for $j$. Let $n \in \mathcal{N}$ with $\nu(n)=j$ and consider the localization map
\[
\Sel_{\mathcal{F}}(K,T) \longrightarrow \prod_{\ell | n}H^1_{\unr}(K_\ell,T),
\]
where the right-hand side is free of rank $j$. Then the image of this map has length at most $2\cdot\sum_{i\leq\frac{j-e}{2}-1}d_i+ek$, so its kernel has length at least $2\cdot \sum_{i\geq\frac{j-e}{2}}d_i$. Since $M_n$ is contained in $\Sel_{\mathcal{F}(n)}(K,T)$, it follows that $\length(M_n)\geq\sum_{i\geq\frac{j-e}{2}}d_i$, so there is an inequality
\[
\partial^{(j)}(\lambda)\geq\min\Biggl\{k,\partial(\lambda)+\sum_{i\geq\frac{j-e}{2}}d_i \Biggr\}.
\]
Now we prove the opposite inequality. Choose $\ell_1\in\mathcal{P}$ such that the restriction of
\[
\loc_{\ell_1}: \Sel_{\mathcal{F}}(K,T) \longrightarrow H^1_{\unr}(K_{\ell_1},T)
\]
to one of the components of biggest length is injective and pick $\ell_2\in\mathcal{P}\smallsetminus\{\ell_1\}$ such that the restriction of
\[
\loc_{\ell_2}: \Sel_{\mathcal{F}(\ell_1)}(K,T) \longrightarrow H^1_{\unr}(K_{\ell_2},T)
\]
to one of the components of biggest length is injective (as before, the existence of $\ell_1$ and $\ell_2$ is guaranteed by \cite[Lemma 2.3.3]{howard2012bipartite}). 

Suppose that $e=1$ (so $1\in\mathcal{N}^{\ind}$, $\epsilon=0$ and $j$ is odd). By Lemma \ref{lemmaind}, we have $M_1\simeq M_{\ell_1}$, so Lemma \ref{lemmadef} allows us to conclude that there is an isomorphism
\[
\Sel_{\mathcal{F}(\ell_1 \ell_2)}(K,T)\simeq R\oplus\bigoplus_{i\geq1} \bigl(R/\mathfrak{m}^{d_{i}}\bigr)^2.
\]
Similarly, if $e=0$ (so $1\in\mathcal{N}^{\deff}$, $\epsilon=1$, and $j$ is even),
by Lemma \ref{lemmadef} there is an isomorphism
\[
\Sel_{\mathcal{F}(\ell_1)}(K,T) \simeq R\oplus\bigoplus_{i\geq1}\bigl(R/\mathfrak{m}^{d_{i}}\bigr)^2,
\]
while by Lemma \ref{lemmaind} we have $M_{\ell_1\ell_2}\simeq M_{\ell_1}$; in conclusion, there is an isomorphism
\[
\Sel_{\mathcal{F}(\ell_1 \ell_2)}(K,T)\simeq\bigoplus_{i\geq1}\bigl(R/\mathfrak{m}^{d_{i}}\bigr)^2.
\]
Therefore, in any case we have 
\[
\Sel_{\mathcal{F}(\ell_1 \ell_2)}(K,T)\simeq R^e\oplus\bigoplus_{i\geq0} \bigl(R/\mathfrak{m}^{d_{i+1}}\bigr)^2.
\]
Now set $\mathcal{P}'\defeq\mathcal{P}\smallsetminus\{\ell_1,\ell_2\}$. Consider the bipartite Euler system $(\kappa',\lambda')$ for $(T,\mathcal{F}(\ell_1 \ell_2),\mathcal{P}')$ where $\kappa'_n\defeq\kappa_{n \ell_1 \ell_2}$ and $\lambda'_n\defeq\lambda_{n \ell_1 \ell_2}$ for all square-free products $n$ of prime numbers in $\mathcal{P}'$; observe that the ordinary conditions at $\ell_1$ and $\ell_2$ are both self-dual and cartesian. Then the pair $(\kappa',\lambda')$ is a non-trivial free bipartite Euler system. 

With notation as in Definition \ref{C-def}, choose $m\in\mathcal{N}^{\deff}$ such that $\ind(\lambda_m)=\partial(\lambda)$ and (using \cite[Corollary 2.2.12]{howard2012bipartite}) pick $q_1 \in \mathcal{P}$ such that the restriction of
\[
\loc_{q_1}:\Sel_{\mathcal{F}(m\ell_1)}(K,T) \longrightarrow H^1_{\unr}(K_{q_1},T)
\]
to $C_{m \ell_1}$ is an isomorphism. Moreover, choose $q_2 \in \mathcal{P}$ such that the restriction of
\[
\loc_{q_2}:\Sel_{\mathcal{F}(m \ell_1 \ell_2q_1 )}(K,T) \longrightarrow H^1_{\unr}(K_{q_2},T)
\]
to $C_{m \ell_1 \ell_2q_1 }$ is an isomorphism. Let us write $\ind(\kappa_n)$ for the index of divisibility of $\kappa_n$ in $\Sel_{\mathcal{F}(n)}(K,T)$, \emph{i.e.}, the largest integer $b\leq k$ such that $\kappa_n\in \mathfrak{m}^b\Sel_{\mathcal{F}(n)}(K,T)$. Then there are the following three series of equalities and inequalities:
\[
\begin{aligned}
\partial(\lambda) &=\ind(\lambda_m) &  \text{\small{(choice of $m$)}}\\
&=\ind\bigl(\loc_{\ell_1}(\kappa_{m \ell_1})\bigr) & \text{\small{(first explicit reciprocity law)}}\\
&\geq\ind(\kappa_{m \ell_1});& \text{}\\
\ind(\kappa_{m \ell_1})&=\ind\bigl(\loc_{q_1}(\kappa_{m \ell_1 })\bigr) &\text{\small{(choice of $q_1$)}}\\
&=\ind(\lambda_{m \ell_1 q_1})& \text{\small{(second explicit reciprocity law)}}\\
&=\ind\bigl(\loc_{\ell_2}(\kappa_{m \ell_1 q_1 \ell_2})\bigr) &\text{\small{(first explicit reciprocity law)}}\\
&\geq \ind(\kappa_{m \ell_1 q_1 \ell_2}); &\text{}\\
\ind(\kappa_{m \ell_1 q_1 \ell_2})&=\ind\bigl(\loc_{q_2}(\kappa_{m \ell_1  q_1 \ell_2})\bigr) &\text{\small{(choice of $q_2$)}}\\
&=\ind(\lambda_{m \ell_1 \ell_2q_1q_2})  &\text{\small{(second explicit reciprocity law)}}\color{black}\\
&\geq \partial(\lambda')  &\text{\small{(definition of $\lambda'$)}}.
\end{aligned}
\]
Thus, we have shown that $\partial(\lambda')\leq \partial(\lambda)$. Finally, we obtain
\[
\partial^{(j)}(\lambda) \leq \partial^{(j-2)}(\lambda')=\min\Biggl\{k,\partial(\lambda')+ \sum_{i\geq\frac{j-e}{2}-1} d_{i+1}\Biggr\}\leq\min\Biggl\{k,\partial(\lambda)+\sum_{i \geq\frac{j-e}{2}} d_i\Biggr\},
\]
where the equality holds by our inductive hypothesis. This concludes the proof. \end{proof}

It is worth noticing that Theorem \ref{artSel} expresses the structure of Selmer groups 
in terms of the $\lambda$-component of $(\kappa,\lambda)$ both when $e=0$ and when $e=1$. The next corollary shows that, when $e=1$, one can also derive from Theorem \ref{artSel} a similar statement describing the structure of Selmer groups in terms of the $\kappa$-component of $(\kappa,\lambda)$. We first introduce the relevant notation. Given a bipartite Euler system $(\kappa, \lambda)$ of parity $0$ and an even integer $j$, define
\begin{equation} \label{partialk}
\partial^{(j)}(\kappa)\defeq\min\bigl\{\ind(\kappa_n)\mid\nu(n)=j\bigr\},\quad
\partial(\kappa)\defeq\min\bigl\{\partial^{(j)}(\kappa)\mid j\equiv0\pmod 2\bigr\}.
\end{equation}
In the following statement, keep the notation of Theorem \ref{artSel} in force.

\begin{theorem} \label{artKappa}
Let $(\kappa, \lambda)$ be a non-trivial free bipartite Euler system of parity $0$. Assume that there is an isomorphism
\[
\Sel_{\mathcal{F}}(K,T) \simeq R\oplus \bigoplus_{i \geq 0} \bigl( R/\mathfrak{m}^{d_i} \bigr)^2
\]
for a non-increasing sequence ${(d_i)}_{i\geq0}$ with $d_i=0$ for $i\gg0$. The equality 
\[
\partial^{(j)}(\kappa)=\min\Biggl\{ k,\partial(\kappa)+\sum_{i\geq\frac{j}{2}}d_i\Biggr\}
\]
holds for all even integers $j\geq0$.
\end{theorem}

\begin{proof} By Theorem \ref{artSel}, it suffices to check that $\partial^{(j)}(\kappa)=\partial^{(j+1)}(\lambda)$ for every even $j$, \emph{i.e.}, that the equality
\[
\min\bigl\{\ind(\lambda_n)\mid\nu(n)=j+1\bigr\}=\min\bigl\{\ind(\kappa_n)\mid\nu(n)=j\bigr\}
\]
holds for every even $j$. Lemma \ref{lemmaind} gives
\[
\min\bigl\{\ind(\kappa_n)\mid\nu(n)=j\bigr\}\leq\min\bigl\{\ind(\lambda_n)\mid\nu(n)=j+1\bigr\},
\]
and the other inequality follows from the second explicit reciprocity law (Definition \ref{defBES}). Indeed, for $\ell\in\mathcal{P}$ we have
\begin{equation} \label{ineq-def/ind}
\ind(\kappa_n)\geq\ind\bigl(\loc_\ell(\kappa_n)\bigr)=\ind(\lambda_{n\ell}),
\end{equation}
where the equality follows from the second explicit reciprocity law. \end{proof}

The following remark will not be used in the paper, but we add it for completeness.

\begin{remark}
Fix $n\in\mathcal{N}$. Let $q_1,q_2\in \mathcal{P}$ with $q_1\neq q_2$, $q_1\nmid n$, $q_2\nmid n$. For simplicity, put $\Sel_{\star}\defeq\Sel_\star(K,T)$ for $\star\in\bigl\{\mathcal{F}(n),\mathcal{F}(nq_1),\mathcal{F}(nq_2)\bigr\}$. By \cite[Corollary 2.2.12]{howard2012bipartite}, there are equalities
\[ 
\length(M_1)=\begin{cases}
\length(M_{q_1})+\length\bigl(\loc_{q_1}(\Sel_{\mathcal{F}})\bigr) & \text{if $e=0$},\\[2mm]
\length(M_{q_1})-\length\bigl(\loc_{q_1}(\Sel_{\mathcal{F}(q_1)})\bigr) & \text{if $e=1$}
\end{cases}
\]
and
\[
\length(M_{q_1})=\begin{cases}
\length(M_{q_1q_2})-\length\bigl(\loc_{q_2}(\Sel_{\mathcal{F}(q_1q_2)})\bigr) & \text{if $e=0$},\\[2mm]
\length(M_{q_1q_2})+\length\bigl(\loc_{q_2}(\Sel_{\mathcal{F}(q_1)})\bigr) & \text{if $e=1$}.
\end{cases}
\]
Thus, we get
\[
\text{{\small$\length(M_{q_1q_2})=
\begin{cases}
\length(M_1)-\length\bigl(\loc_{q_1}(\Sel_{\mathcal{F}})\bigr)+\length\bigl(\loc_{q_2}(\Sel_{\mathcal{F}(q_1q_2)})\bigr) & \text{if $e=0$},\\[2mm]
\length(M_1)+\length\bigl(\loc_{q_1}(\Sel_{\mathcal{F}(q_1)})\bigr)-\length\bigl(\loc_{q_2}(\Sel_{\mathcal{F}(q_1)})\bigr) & \text{if $e=1$}.
\end{cases}$}}
\]
Now choose $\ell_1\in\mathcal{P}$ such that $\ell_1\nmid n$ and the restriction of
\[
\loc_{\ell_1}: \Sel_{\mathcal{F}(n)}(K,T) \longrightarrow H^1_{\unr}(K_{\ell_1},T)
\]
to one of the components of largest length is injective; choose also $\ell_2 \in \mathcal{P}$ such that $\ell_2\nmid n\ell_1$ and the restriction of
\[
\loc_{\ell_2}: \Sel_{\mathcal{F}(n\ell_1)}(K,T) \longrightarrow H^1_{\unr}(K_{\ell_2},T)
\]
has the same property. The existence of $\ell_1$ and $\ell_2$ is ensured, as before, by \cite[Lemma 2.3.3]{howard2012bipartite}. If $e=1$, then $\loc_{\ell_1}\bigl(\Sel_{\mathcal{F}(n\ell_1)}\bigr)=0$ by \cite[Corollary 2.2.12]{howard2012bipartite}, as $\length\bigl(\loc_{\ell_1}(\Sel_{\mathcal{F}(n)})\bigr)=k$ for our choice of $\ell_1$; moreover, our choice of $\ell_2$ gives $\length\bigl(\loc_{\ell_2}(\Sel_{\mathcal{F}(n\ell_1)})\bigr)=d_0$. Similarly, if $e=0$, then $\length\bigl(\loc_{\ell_1}(\Sel_{\mathcal{F}(n)})\bigr)=d_0$ by our choice of $\ell_1$ 
and $\length\bigl(\loc_{\ell_2}(\Sel_{\mathcal{F}(n\ell_1\ell_2)})\bigr)=0$ by 
\cite[Corollary 2.2.12]{howard2012bipartite}, as $\length\bigl(\loc_{\ell_2}(\Sel_{\mathcal{F}(n\ell_1)})\bigr)=k$. Therefore, in both cases there is an equality $\length(M_{n\ell_1\ell_2})=\length(M_n)-d_0$. In particular, for all $q_1,q_2\in\mathcal{P}$ we get $\length(M_{n\ell_1\ell_2})\leq \length(M_{nq_1q_2})$.
\end{remark}
\color{black}

\subsection{Bipartite Euler systems over a DVR}

Let $R$ be a complete discrete valuation ring with maximal ideal $\mathfrak{m}=(\pi)$ and quotient field $F$, which is assumed to be a finite extension of $\Q_p$ inside $\overline{\Q}_p$. Set 
\[
V\defeq T\otimes_RF,\quad A\defeq V/T=T\otimes_R(F/R).
\]
Given a Selmer structure $(\mathcal{F},\Sigma_\mathcal{F})$ on $V$ as in Definition \ref{SelStructureT}, we introduce a Selmer structure $(\mathcal{F},\Sigma_\mathcal{F})$ on $T$ by defining $H^1_{\mathcal{F}}(K_\mathfrak{l},T)$ to be the inverse image of $H^1_{\mathcal{F}}(K_\mathfrak{l},V)$ under the map induced by the inclusion $T \hookrightarrow V$ for all primes $\mathfrak{l}$ of $K$. It is actually a Selmer structure since $H^1_{\mathcal{F}}(K_\mathfrak{l},T)=H^1_{\unr}(K_\mathfrak{l},T)$ for all $\mathfrak{l} \notin \Sigma_\mathcal{F}$ by \cite[Lemma 1.1.9]{MR}.
Analogously, let $H^1_\mathcal{F}(K_\mathfrak{l},A)$ be the Selmer structure on $A$ obtained by propagation from $H^1_\mathcal{F}(K_\mathfrak{l},V)$ via the projection map $V\twoheadrightarrow A$ and define
\[
\Sel_\mathcal{F}(K,A)\defeq\ker\Biggl(H^1(K,A)\longrightarrow \prod_\mathfrak{l} \frac{H^1(K_\mathfrak{l},A)}{H^1_\mathcal{F}(K_\mathfrak{l},A)}\Biggr).
\]
For every integer $k\geq1$, the Selmer structure $(\mathcal{F},\Sigma_\mathcal{F})$ on $T$ induces by propagation via the projection map $T\twoheadrightarrow T/\mathfrak{m}^kT$ (respectively, via the inclusion $A[\mathfrak{m}^k]\subset A$) a Selmer structure on $T/\mathfrak{m}^kT$ (respectively, $A[\mathfrak{m}^k]$); for each of these Selmer structures, we can consider the corresponding Selmer groups. To simplify our notation, set $T_k\defeq T/\mathfrak{m}^kT$ and $A_k\defeq A[\mathfrak{m}^k]$. 

\begin{definition} 
Let $k\geq1$ be an integer. A prime number $\ell$ is \emph{$k$-admissible} for $T$ over $R$ if $I_\ell \subset \mathfrak{m}^k$.
\end{definition} 

We denote by $\mathcal{P}_k$ the set of $k$-admissible primes in $\mathcal{P}$ and by $\mathcal{N}_k$ the set of square-free products of primes in $\mathcal{P}_k$. For all integers $k\geq1$ and all $n\in\mathcal{N}_k$, we still write $\mathcal{F}$ for the induced Selmer structure for $T/\mathfrak{m}^kT$. Set 
\[
S_n^{(k)}\defeq\Sel_{\mathcal{F}(n)}(K,T_k)
\] 
and, as in Theorem \ref{eSel}, denote by $e^{(k)}(n)\in\{0,1\}$ the parity of the rank of $S_n^{(k)}$.

\begin{assumption}\label{ass}
\begin{enumerate}
\item $\overline{T}=T/\mathfrak{m}T$ is absolutely irreducible;
\item $T$ admits a perfect and symmetric $R$-bilinear pairing
\[
(\cdot, \cdot):T \times T \longrightarrow R(1)
\]
such that $(x^\sigma,y^{\tau \sigma \tau})=(x,y)^\sigma$ for any $\sigma \in G_K$, where $\tau \in G_\Q$ is a fixed complex conjugation;
\item $\mathcal{F}$ is cartesian in the category $\mathrm{Quot}(T)$ (\emph{cf.} \cite[\S1.2]{HoHeeg}) and self-dual;
\item for any $k \geq 1$ and any $c \in H^1(K,\overline{T})$ with $c\neq 0$, there are infinitely many $\ell \in \mathcal{P}_k$ such that $\loc_\ell(c) \neq 0$;
\item for each $k\geq 0$, the following \emph{control theorem} holds:
\[
\Sel_{\mathcal{F}(n)}(K,A_k)\simeq \Sel_{\mathcal{F}(n)}(K,A)[\mathfrak{m}^k].
\]
\end{enumerate}    
\end{assumption}
Observe that Assumption \ref{ass} immediately implies that for all integers $k \geq 1$ the triple $(T_k,\mathcal{F},\mathcal{P}_k)$ satisfies Assumption \ref{assArt}.

\begin{lemma} \label{lemma-indep}
For all $n \in \mathcal{N}$, the value of $e^{(k)}(n)$ is independent of $k$.
\end{lemma}

\begin{proof} For all integers $k\geq1$, set $R_k\defeq R/\mathfrak{m}^k$ and write 
\[
\Sel_{\mathcal{F}(n)}(K,T_k)=R_k^{e^{(k)}(n)}\oplus\bigoplus_{i\geq1}\Bigl(R_k/\mathfrak{m}^{d_i^{(k)}(n)}\Bigr)^2.
\] 
Taking $\mathfrak{m}$-torsion, we get an equality 
\[
\dim_{R/\mathfrak{m}}\bigl(\Sel_{\mathcal{F}_n}(K,T_k)[\mathfrak{m}]\bigr)=e^{(k)}(n)+2\cdot\sum_{d_i^{(k)}(n)\neq 2} i\equiv e^{(k)}(n)\pmod{2}.
\]
Since the Selmer condition $\mathcal{F}(n)$ is cartesian, \cite[Lemma 2.2.6]{howard2012bipartite} ensures that for all pairs of integers $(k,k')$ with $k\geq k'\geq 1$ there is an isomorphism
\[
\Sel_{\mathcal{F}(n)}(K,T_{k'})\simeq\Sel_{\mathcal{F}(n)}(K,T_{k})\bigl[\mathfrak{m}^{k'}\bigr].
\]
Thus, since $T_k[\mathfrak{m}]\simeq T_1$ (an isomorphism being induced by the multiplication-by-$\pi^{k-1}$ map), we have 
\[
\dim_{R/\mathfrak{m}}\bigl(\Sel_{\mathcal{F}_n}(K,T_k)[\mathfrak{m}^{k-1}]\bigr)=
\dim_{R/\mathfrak{m}}\bigl(\Sel_{\mathcal{F}_n}(K,T_1)\bigr)\equiv e^{(1)}(n)\pmod{2}.
\]
Therefore, $e^{(k)}(n)\equiv e^{(1)}(n)\pmod{2}$ for all $k\geq 1$, which concludes the proof. \end{proof}

Thanks to Lemma \ref{lemma-indep}, we can define $e(n)\defeq e^{(k)}(n)$ for any choice of $k$; set $e\defeq e(1)$. We also define $\mathcal{N}^{\ind}$ and $\mathcal{N}^{\deff}$ as in the case of principal artinian rings. Notice that there exists $\epsilon \in \Z/2 \Z$ such that $\mathcal{N}^{\ind}=\mathcal{N}^\epsilon$ and $\mathcal{N}^{\deff}=\mathcal{N}^{\epsilon +1}$.

Given a bipartite Euler system $(\kappa, \lambda)$ of parity $\epsilon$ for $(T,\mathcal{F},\mathcal{P})$, for any integer $k\geq 1$ we obtain a bipartite Euler system $\bigl(\kappa^{(k)},\lambda^{(k)}\bigr)$ for $(T_k,\mathcal{F},\mathcal{P}_k)$ by projection: namely, for $n\in\mathcal{N}_k^{\ind}$ the class $\kappa_n^{(k)}$ is the image of $\kappa_n\in\Sel_{\mathcal{F}(n)}(K,T/I_nT)$ in $\Sel_{\mathcal{F}(n)}(K,T_k)$ via the canonical projection map $T/I_nT\twoheadrightarrow T_k=T/\mathfrak m^kT$, while for $n\in\mathcal{N}_k^{\deff}$ the class $\lambda_n^{(k)}$ is the image of $\lambda_n\in R/I_n$ in $R_k=R/\mathfrak{m}^k$ via the canonical projection map $R/I_n\twoheadrightarrow R_k$ (recall that the inclusion $I_n\subset\mathfrak{m}^k$ holds because $n$ is a product of $k$-admissible primes). The bipartite Euler system $\bigl(\kappa^{(k)},\lambda^{(k)}\bigr)$ for $(T_k,\mathcal{F},\mathcal{P}_k)$ thus constructed may not be free, and we modify it as follows. For integers $t\geq k >0$, denote by $\bigl(\kappa^{(k),t},\lambda^{(k),t}\bigr)$ the bipartite Euler system for 
$(T_k,\mathcal{F},\mathcal{P}_t)$ whose elements $\lambda^{(k),t}_n$ for $n\in\mathcal{N}_j^{\deff}$ and $\kappa^{(k),t}_n$ for $n\in \mathcal{N}_j^{\ind}$ are defined again by projection: namely, for $n\in \mathcal{N}^{\deff}$ the element 
$\lambda_n^{(k),t}$ is the image of $\lambda_n^{(t)}$ via the map canonical projection 
$R_t\rightarrow R_k$, while for $n\in\mathcal{N}^{\ind}$ the element $\kappa_n^{(k),t}$ is the image of $\kappa_n^{(t)}$ under the map $S^{(t)}_n\rightarrow S^{(k)}_n$ induced by the canonical projection. 

\begin{lemma} \label{freeness}
The bipartite Euler system $\bigl(\kappa^{(k),t},\lambda^{(k),t}\bigr)$ is free for all $t \geq 2k$.
\end{lemma}

\begin{proof} This is similar to that of \cite[Lemma 3.3.6]{howard2012bipartite} (\emph{cf.} also \cite[Lemma 4.5]{BD-IMC}): we reproduce it because of the slight notational difference with respect to \cite{howard2012bipartite} in treating the ramification index. Take $n$ such that $e(n)=1$; we must show that there is a free $R_k$-submodule of $\Sel_{\mathcal{F}(n)}(K,T_k)$ containing $\kappa^{(k),t}_n$. By Theorem \ref{artSel}, there are isomorphisms $\Sel_{\mathcal{F}(n)}(K,T_j)\simeq R_j\oplus N\oplus N$ and $\Sel_{\mathcal{F}(n)}(K,T_k)\simeq R_k\oplus M\oplus M$ for finite $R$-modules $N$ and $M$. Since $R_k$ has length $k$, if $\mathfrak{m}^{k-1}M\neq0$, then $\kappa^{(k),t}_n=0$ by \cite[Proposition 2.3.5]{howard2012bipartite} and there is nothing to prove. Therefore, assume that $\mathfrak{m}^{k-1}M=0$. By \cite[Lemma 2.2.6]{howard2012bipartite}, there is a commutative triangle
\[ 
\xymatrix@C=40pt@R=40pt{\Sel_{\mathcal{F}(n)}(K,T_j)\ar[r]^-{\pi^{j-k}\cdot} \ar[d]& \Sel_{\mathcal{F}(n)}(K,T_j)[\mathfrak{m}^{k}]\\ \Sel_{\mathcal{F}(n)}(K,T_k)\ar[ru]^-\simeq} 
\] 
in which the horizontal arrow is multiplication by $\pi^{j-k}$, the vertical arrow is induced, as before, by the canonical projection and the diagonal map is an isomorphism by the control theorem in Assumption \ref{ass}. Since $\mathfrak{m}^{k-1}M=0$, the module $\mathfrak{m}^{k-1}\Sel_{\mathcal{F}(n)}(K,T_k)$ is cyclic, isomorphic to $\mathfrak{m}^{k-1}R_k/R_k\simeq R/\mathfrak{m}$; moreover, the diagonal isomorphism shows that $\mathfrak{m}^{k-1}\Sel_{\mathcal{F}(n)}(K,T_j)[\mathfrak{m}^{k}]$ is also cyclic. In particular, it follows that $\mathfrak{m}^{j-k}N=0$. Since $j\geq 2k$, the image of $N$ under the vertical arrow is trivial, so the image of the vertical map is free of rank $1$. Since $\kappa_n^{(k),t}$ lies in this image, this concludes the proof. \end{proof}

From now on, we consider the bipartite Euler system $\bigl(\boldsymbol{\kappa}^{(k)},\boldsymbol{\lambda}^{(k)}\bigr)=\bigl(\kappa^{(k),2k},\lambda^{(k),2k}\bigr)$ for $(T_k, \mathcal{F}, \mathcal{P}_{2k})$, which is free by Lemma \ref{freeness}. We still write $\lambda^{(k)}_n=\boldsymbol{\lambda}^{(k)}_n$ and $\kappa^{(k)}_n=\boldsymbol{\kappa}^{(k)}_n$ for the classes corresponding to $n\in\mathcal{N}_j^{\deff}$ and $n\in \mathcal{N}_j^{\ind}$, respectively.    
Observe that, given an integer $j\geq1$ such that $j \equiv \epsilon +1 \pmod 2$, the sequences $\bigl(\partial(\boldsymbol{\lambda}^{(k)})\bigr)_{k\geq1}$ and $\bigl(\partial^{(j)}(\boldsymbol{\lambda}^{(k)})\bigr)_{k\geq1}$ are non-decreasing; define
\begin{equation} \label{deltas}
\delta(\lambda)\defeq\lim_{k \rightarrow + \infty} \partial(\boldsymbol{\lambda}^{(k)}),\quad\delta^{(j)}(\lambda)\defeq\lim_{k \rightarrow + \infty} \partial^{(j)}(\boldsymbol{\lambda}^{(k)}).
\end{equation}

\begin{theorem} \label{thmDVR}
Let $(\kappa,\lambda)$ be a bipartite Euler system of parity $\epsilon$. Assume $\lambda_1 \neq 0$ if $e=0$ and $\kappa_1 \neq 0$ if $e=1$. There is an isomorphism of $R$-modules
\[
\Sel_{\mathcal{F}}(K,A) \simeq \left( F/R \right) ^e \oplus \bigoplus_{i\geq0} \bigl(R/\mathfrak{m}^{d_i}\bigr)^2,
\]
where $d_i\defeq\delta^{(2i+e)}(\lambda)-\delta^{(2i+2+e)}(\lambda)$. Furthermore, the sequence ${(d_i)}_{i\geq0}$ is non-increasing.
\end{theorem}

\begin{proof} By \cite[Proposition B.2.7]{RubinES}, the $R$-module $\Sel_{\mathcal{F}}(K,A)$ is cofinitely generated, so there is an isomorphism
\[
\Sel_{\mathcal{F}}(K,A) \simeq \left( F/R \right) ^{r} \oplus \bigoplus_{i \geq 0 }  R/ \mathfrak{m}^{c_i} 
\]
with ${(c_i)}_{i\geq0}$ a non-increasing sequence of non-negative integers with $c_i=0$ for $i\gg0$. By Assumption \ref{ass}, for all integers $k\geq1$ there are isomorphisms
\begin{equation} \label{eqCT}
S_1^{(k)} \simeq \Sel_{\mathcal{F}}(K,A)[\mathfrak{m}^k] \simeq\bigl(R/\mathfrak{m}^k \bigr)^r \oplus \bigoplus_{i \geq 0} R/\mathfrak{m}^{\min \{k,c_i \}}.
\end{equation} 
If we choose an integer $k\geq1$ such that $k > c_i$ for all $i$, then we can rewrite \eqref{eqCT} as 
\begin{equation} \label{eqCT1}
S_1^{(k)}\simeq\bigl(R/\mathfrak{m}^k\bigr)^r \oplus  
\bigoplus_{i \geq 0} R/\mathfrak{m}^{c_i}.
\end{equation}
We also pick an integer $k$ such that $k>\ind(\lambda_1)$; then $\bigl(\boldsymbol{\kappa}^{(k)},\boldsymbol{\lambda}^{(k)}\bigr)$ is non-trivial and free. Choose integers $d_i\geq 0$ as in Theorem \ref{artSel}, so that 
\begin{equation} \label{eqAR}
S_1^{(k)} \simeq (R/\mathfrak{m}^k)^e\oplus\Biggl(\bigoplus_{i\geq 0}R/\mathfrak{m}^{d_i}\Biggr)^2
\end{equation} 
with 
\begin{equation}  \label{equality2}
\partial^{(j)}(\boldsymbol{\lambda}^{(k)})=\min\Biggl\{k,\partial (\boldsymbol{\lambda}^{(k)})+\sum_{i\geq\frac{j-e}{2}} d_i\Biggr\}
\end{equation}
for all integers $j\geq1$ such that $j\equiv\epsilon+1\pmod 2$. Then there is an equality 
\[
d_i=\begin{cases}k&\text{for $0\leq i\leq\displaystyle{\frac{r-e}{2}}-1$},\\[3mm]c_{2i-r+e}=c_{2i-r+e+1}&\text{for $i\geq\displaystyle{\frac{r-e}{2}}$}. \end{cases}
\]
Therefore, there is an isomorphism
\begin{equation} \label{eqsel1}
S_1^{(k)}\simeq\bigl(R/\mathfrak{m}^k\bigr)^{e}\oplus\Bigl(\bigl(R/\mathfrak{m}^k\bigr)^2 \Bigr)^{\frac{r-e}{2}}\oplus\bigoplus_{i \geq \frac{r-e}{2}}\bigl(R/\mathfrak{m}^{d_i} \bigr)^2.
\end{equation}
We claim that, under our assumptions, the sequences $\bigl(\partial(\boldsymbol{\lambda}^{(k)})\bigr)_{k\geq1}$ and $\bigl(\partial^{(e)}(\boldsymbol{\lambda}^{(k)})\bigr)_{k\geq1}$ are bounded from above. This is obvious if $e=0$, as in this case $\lambda_1\neq0$ in $R$. Thus, assume $e=1$ and $\kappa_1 \neq 0$. By \cite[Lemma 2.3.3]{howard2012bipartite}, which we can apply thanks to Assumption \ref{ass}, for any integer $k\geq1$ and any free $R_k$-submodule $C\subset \Sel_{\mathcal{F}}(K,T_k)$ of rank $1$ there are infinitely many $\ell \in \mathcal{P}_k$ such that $\loc_\ell$ takes $C$ isomorphically onto $H^1_{\unr}(K_\ell,T_k)$. For any $k$, fix a prime number $\ell_k$ satisfying the previous property and notice that $\ind(\lambda_{\ell_k},R_k)\leq\ind(\kappa_1)$ for $k \gg 0$ (see \eqref{ineq-def/ind} for this inequality). This implies that $\bigl(\partial^{(1)}(\boldsymbol{\lambda}^{(k)})\bigr)_{k\geq1}$ is bounded from above, and then the same is true of $\bigl(\partial(\boldsymbol{\lambda}^{(k)})\bigr)_{k\geq1}$, as desired. We also notice that, since $\bigl(\partial^{(j)}(\boldsymbol{\lambda}^{(k)})\bigr)_{j\geq1}$ is non-increasing, the sequence $\bigl(\partial^{(j)}(\boldsymbol{\lambda}^{(k)})\bigr)_{k\geq1}$ is bounded for all $j$ with $j\equiv\epsilon+1\pmod 2$.

Since $\partial^{(e)}(\boldsymbol{\lambda}^{(k)})$ is upper bounded, from \eqref{equality2} with $j=e$  
we find that $d_i < k$ for any $i \geq 0$ (and for $k\gg0$) and thus it implies $r=e$.
Moreover, since the sequence $\bigl(\partial^{(j)}(\boldsymbol{\lambda}^{(k)})\bigr)_{k\geq1}$ is bounded, equality \eqref{equality2} shows that for $k\gg0$ the equality
\begin{equation}\label{eqpart}
\partial^{(j)}(\boldsymbol{\lambda}^{(k)})=\partial (\boldsymbol{\lambda}^{(k)}) + \sum_{i \geq \frac{j-e}{2}} d_i
\end{equation}
holds for any integer $j\geq1$ with $j\equiv\epsilon+1\pmod 2$. Letting $k \rightarrow + \infty$ in \eqref{eqpart} yields the equality
\begin{equation} \label{eqpart1}
\delta^{(j)}(\lambda)=\delta(\lambda) + \sum_{i \geq \frac{j-e}{2}}d_i,
\end{equation}
so for $i=\frac{j-e}{2}$ we get $d_i=\delta^{(j)}(\lambda)-\delta^{(j+2)}(\lambda)$,
which (using the equality $j= 2i+e$) can be rewritten as 
\[
d_i=\delta^{(2i+e)}(\lambda)-\delta^{(2i+2+e)}(\lambda).
\]
The proof of the theorem is complete. \end{proof}

\color{black}

Denote by
\[
\Sha_\mathcal{F}(K,A)\defeq\Sel_\mathcal{F}(K,A)\big/\divv\bigl(\Sel_\mathcal{F}(K,A)\bigr)
\]
the Shafarevich--Tate group of $A$ over $K$ with respect to $\mathcal{F}$, where
$\divv\bigl(\Sel_\mathcal{F}(K,A)\bigr)$ is the maximal divisible subgroup of $\Sel_\mathcal{F}(K,A)$. Moreover, write $X(T)$ for the Pontryagin dual of $\Sha_\mathcal{F}(K,A)$, \emph{i.e.}, set
\[
X(T)\defeq\Hom\Bigl(\Sha_\mathcal{F}(K,A),\Q_p/\Z_p\Bigr).
\]
Theorem \ref{thmDVR} implies that there is an isomorphism
\begin{equation}\label{ShaX(T)}
X(T) \simeq \bigoplus_{i\geq0}\bigl(R/\mathfrak{m}^{d_i}\bigr)^2
\end{equation}
with $d_i\defeq\delta^{(2i+e)}(\lambda)-\delta^{(2i+2+e)}(\lambda)$. 
In particular, by Theorem \ref{FittingDVR}, for all integers $i\geq0$ there is an equality
\begin{equation} \label{fittingDVR}
\Fitt_i\bigl(X(T)\bigr)=\mathfrak{m}^{2(\delta^{(i+e)}(\lambda)-\delta(\lambda))}.
\end{equation}
It follows from the equality $\delta(\lambda)= \min_j \delta^{(j)}(\lambda)$.

In the indefinite case, we can express the result above in terms of classes in $\kappa$. For all even integers $j\geq1$, define
\[
\delta(\kappa)\defeq\lim_{k\rightarrow+\infty}\partial\bigl(\boldsymbol{\kappa}^{(k)}\bigr),\quad\delta^{(j)}(\kappa)\defeq\lim_{k\rightarrow+\infty}\partial^{(j)}\bigl(\boldsymbol{\kappa}^{(k)}\bigr).
\]

\begin{theorem} \label{dvrKappa}
Let $(\kappa,\lambda)$ be a bipartite Euler system of parity $0$ and assume $\kappa_1 \neq 0$. There is an isomorphism of $R$-modules
\[
\Sel_{\mathcal{F}}(K,A)\simeq(F/R)\oplus\bigoplus_{i\geq0}\bigl(R/\mathfrak{m}^{d_i}\bigr)^2,
\]
where $d_i\defeq\delta^{(2i)}(\kappa)-\delta^{(2i+2)}(\kappa)$. Furthermore, the sequence ${(d_i)}_{i\geq0}$ is non-increasing.
\end{theorem}

\begin{proof} One can proceed as in the proof of Theorem \ref{thmDVR}, using Corollary \ref{artKappa} in place of Theorem \ref{artSel}. \end{proof}

\subsection{Bipartite Euler systems over $\Lambda$} \label{BESLambda}

Denote by $h_K$ the class number of $K$ and assume $p\nmid D_Kh_K$. For any integer $n\geq1$, write $K[n]$ for the ring class field of $K$ of conductor $n$. Since $p \nmid \Gal(K[p]/K)$ and $\Gal(K[p^{m+1}]/K[p]) \simeq \Z/p^m \Z$, there is an isomorphism
\[
\Gal(K[p^{m+1}]/K) \simeq \Gamma \times \Delta,
\]
with $\Gamma \simeq \Z/p^m \Z$ and $\Delta\defeq\Gal(K[p]/K)$. Set
\[ 
K_m\defeq K[p^{m+1}]^{\Delta},\quad\Gamma_m\defeq\Gal(K_m/K) \simeq \Z/p^m \Z
\]
for any integer $m\geq1$ and
\[
K_\infty\defeq\bigcup_{m\geq1} K_m,\quad\Gamma_\infty\defeq\varprojlim_{m} \Gamma_m=\Gal(K_\infty/K)\simeq\Z_p.
\]
The field $K_\infty$ is the anticyclotomic $\Z_p$-extension of $K$.

In what follows, let $\mathcal{O}$ be a complete discrete valuation ring with maximal ideal $\mathfrak{m}$ and residual characteristic $p$; set
\[
\Lambda_m\defeq\mathcal{O}[\Gamma_m],\quad\Lambda\defeq\varprojlim_m \Lambda_m = \mathcal{O} \llbracket \Gamma_\infty \rrbracket.
\]
Then there is a (non-canonical) isomorphism of $\mathcal{O}$-algebras 
$\Lambda \simeq \mathcal{O}\llbracket T \rrbracket$. Let $\mathfrak{m}_\Lambda\defeq(T,\pi)$ be the maximal ideal of $\Lambda$. A finitely generated $\Lambda$-module $M$ is \emph{$\mathfrak{m}_\Lambda$-divisible} if $\mathfrak{m}_\Lambda M=M$; we denote by $\divv_{\mathfrak{m}_\Lambda}(M)$ the maximal $\mathfrak{m}_\Lambda$-divisible submodule of $M$ (see, \emph{e.g.}, \cite[(4.1.20)]{NP}). 

Let $\mathbf{T}$ be a free $\Lambda$-module of rank $2$ endowed with a continuous $G_K$-action and define $\mathbf{A}\defeq \Hom_{\mathrm{cont}}(\mathbf{T}^\iota, \Bmu_{p^\infty})$  where $\iota:\Lambda \rightarrow \Lambda$ is the involution sending $\sigma$ to $\sigma^{-1}$ for all $\sigma\in\Gamma_\infty$.
Denote by $\mathcal{F}$ and $\mathcal{P}$ a Selmer structure and a set of admissible primes for $\mathbf{T}$, respectively; write $\mathcal{F}$ also for the Selmer structure for $\mathbf{A}$ such that $H^1_\mathcal{F}(K_\mathfrak{l},\mathbf{A})$ is the orthogonal complement of $H^1_\mathcal{F}(K_{\bar{\mathfrak{l}}},\mathbf{T})$ under the local Tate pairing for all primes $\mathfrak{l}$ of $K$. Set
\[
\Sel_\mathcal{F}(K,\mathbf{A})\defeq \ker \Biggl(H^1(K,\mathbf{A})\longrightarrow \prod_\mathfrak{l} \frac{H^1(K_\mathfrak{l},\mathbf{A})}{H^1_\mathcal{F}(K_\mathfrak{l},\mathbf{A})} \Biggr).
\]
For an integer $k\geq1$, a prime number $\ell \in \mathcal{P}$ is \emph{$k$-admissible} if $I_\ell \subset \mathfrak{m}^k$; we denote by $\mathcal{P}_k$ the set of $k$-admissible primes and write $\mathcal{N}_k$ for the set of $k$-admissible integers, \emph{i.e.}, square-free products of $k$-admissible primes. By arguments analogous to those in \cite[Proposition 2.1.5 and Lemma 2.3.3]{MR}, one can show that $\Sel_\mathcal{F}(K,\mathbf{T})$ 
and $\Sel_\mathcal{F}(K,\mathbf{A})$ are finitely and cofinitely generated $\Lambda$-modules. We also consider the Shafarevich--Tate group
\[
\Sha_\mathcal{F}(K,\mathbf{A})\defeq\Sel_\mathcal{F}(K,\mathbf{A})\big/\divv_{\mathfrak{m}_\Lambda}\bigl(\Sel_\mathcal{F}(K,\mathbf{A})\bigr)
\]
of $\mathbf{A}$ over $K$ with respect to $\mathcal{F}$. Finally, let
\begin{equation} \label{X-sha-eq}
X\defeq\Hom\Bigl(\Sha_\mathcal{F}(K,\mathbf{A}),\Q_p/\Z_p\Bigr).
\end{equation}
be the Pontryagin dual of $\Sha_\mathcal{F}(K,\mathbf{A})$.

Let $\mathfrak{P}$ be a height $1$ prime ideal of $\Lambda$. Let $\mathcal{O}_{\mathfrak{P}}$ be the integral closure of $\Lambda/\mathfrak{P}$, which is a discrete valuation ring such that $[\mathcal{O}_{\mathfrak{P}}:\Lambda/\mathfrak{P}]<\infty$; let $\pi_\mathfrak{P}$ be a uniformizer of $\mathcal O_{\mathfrak P}$. The composition of the canonical projection $\Lambda\twoheadrightarrow\Lambda/\mathfrak{P}$ and inclusion $\Lambda/\mathfrak{P}\hookrightarrow\mathcal O_{\mathfrak P}$ induces a map $s_\mathfrak{P}:\Lambda\rightarrow\mathcal{O}_\mathfrak{P}$. Let $F_{\mathfrak P}$ be the fraction field of $\mathcal O_{\mathfrak P}$ and set 
\begin{equation} \label{T-V-A-eq}
T_{\mathfrak{P}}\defeq\mathbf{T} \otimes_\Lambda \mathcal{O}_{\mathfrak{P}},\quad V_\mathfrak{P}\defeq T_\mathfrak{P}\otimes_{\mathcal{O}_\mathfrak{P}}F_\mathfrak{P},\quad A_\mathfrak{P}\defeq V_\mathfrak{P}/T_\mathfrak{P},
\end{equation}
where the tensor product in the definition of $T_{\mathfrak{P}}$ is taken with respect to $s_{\mathfrak P}$. It follows that $T_{\mathfrak{P}}$ is a free $\mathcal{O}_{\mathfrak{P}}$-module of rank $2$ equipped with a continuous action of $G_K$. By an abuse of notation, the obvious map $\mathbf{T}\rightarrow T_\mathfrak{P}$ and the map $H^1(K,\mathbf{T})\rightarrow H^1(K,T_\mathfrak{P})$ induced in cohomology will be still denoted by $s_\mathfrak{P}$.

We assume that, for each height $1$ prime ideal $\mathfrak{P}$ of $\Lambda$, we are given a Selmer structure $\mathcal{F}_\mathfrak{P}$ on $V_\mathfrak{P}$; then, for all integers $n\geq1$, we define by propagation Selmer structures on $A_\mathfrak{P}$ and $T_\mathfrak{P}$, still denoted by $\mathcal{F}_\mathfrak{P}$. Moreover, we assume that $T_\mathfrak{P}$ is endowed with a perfect, symmetric, $\mathcal{O}_\mathfrak{P}$-bilinear pairing
\[
{(\cdot, \cdot)}_\mathfrak{P}:T_\mathfrak{P} \times T_\mathfrak{P} \longrightarrow \mathcal{O}_\mathfrak{P}(1)
\]
such that $(x^\sigma,y^{\tau \sigma \tau})=(x,y)^\sigma$ for all $\sigma \in G_K$, where $\tau \in G_\Q$ is a fixed complex conjugation. In particular, it induces a perfect, $G_K$-equivariant and $\mathcal{O}_\mathfrak{P}$-bilinear pairing
\[
{(\cdot, \cdot)}_\mathfrak{P}:T^\iota_{\mathfrak{P}^\iota} \times A_\mathfrak{P} \longrightarrow \Bmu_{p^\infty}.
\]
In addition, we assume that $\mathcal{F}_\mathfrak{P}$ is self-dual. Observe that the dual of the map $s^\iota_{\mathfrak{P}^\iota}:\mathbf{T}^\iota\rightarrow T^\iota_{\mathfrak{P}^\iota}$ induces a canonical map $A_\mathfrak{P}\rightarrow \mathbf{A}$.

\begin{assumption} \label{assLambda}
For every height $1$ prime ideal $\mathfrak{P}$ of $\Lambda$, the properties below hold true:
\begin{enumerate}
\item the triple $(T_\mathfrak{P},\mathcal{F}_\mathfrak{P},\mathcal{P})$ satisfies Assumption \ref{ass};
\item the Selmer structure $\mathcal{F}_\mathfrak{P}$ is compatible 
with the Selmer structure $\mathcal{F}$ on $\mathbf{T}$ in the sense that 
the specialization map $s_\mathfrak{P}:H^1(K,\mathbf{T})\rightarrow H^1(K,T_\mathfrak{P})$ induces a map 
\begin{equation} \label{sp}
s_\mathfrak{P}:\Sel_{\mathcal{F}}(K,\mathbf{T})/\mathfrak{P}\longrightarrow \Sel_{\mathcal{F}_\mathfrak{P}}(K,T_\mathfrak{P})
\end{equation} 
and the specialization map $i_\mathfrak{P}:H^1(K,A_\mathfrak{P}) \rightarrow H^1(K,\mathbf{A})$ induces a map
\begin{equation} \label{cosp}
i_\mathfrak{P}:\Sel_{\mathcal{F}_\mathfrak{P}}(K,A_\mathfrak{P})\longrightarrow \Sel_\mathcal{F}(K,\mathbf{A})[\mathfrak{P}];\end{equation}
\item the following \emph{control theorem} holds: the map $s_\mathfrak{P}$ in \eqref{sp}
is injective and there is a finite set $\Sigma_\mathrm{CT}$ of height $1$ prime ideals of $\Lambda$ such that for all $\mathfrak{P}\notin \Sigma_\mathrm{CT}$ the kernel and the cokernel of the maps $s_\mathfrak{P}$ and $i_\mathfrak{P}$ in \eqref{sp} and \eqref{cosp} are finite, of order bounded by a constant that depends on $[\mathcal{O}_\mathfrak{P}:\Lambda/\mathfrak{P}]$ but not on $\mathfrak{P}$ itself. 
\end{enumerate}
\end{assumption}

Let us fix $\epsilon \in \Z/2\Z$ and assume that there exists a non-trivial bipartite Euler system $(\kappa, \lambda)$ of parity $\epsilon$ for $(\mathbf{T},\mathcal{F},\mathcal{P})$. 
Then we have an element $\lambda_1\in \Lambda$ if $1\in\mathcal{N}^{\epsilon+1}$ and a class 
$\kappa_1\in \Sel_\mathcal{F}(K,\mathbf{T})$
if $1\in\mathcal{N}^{\epsilon}$.

\begin{assumption}[Bipartite Euler systems] \label{assBESLambda}
    $\lambda_1\neq 0$ if $1\in\mathcal{N}^{\epsilon+1}$ and $\kappa_1\neq 0$ if $1\in\mathcal{N}^{\epsilon}$. 
\end{assumption}

Let $\mathfrak{P}$ be a height $1$ prime ideal of $\Lambda$. For all $n\in\mathcal{N}$, denote by $I_{\mathfrak{P},n}$ the usual ideal of $\mathcal{O}_\mathfrak{P}$ attached to $T_\mathfrak{P}$ and define a set of classes $(\kappa_{\mathfrak{P}},\lambda_{\mathfrak{P}})$ as follows:  
\begin{itemize}
\item if $n\in\mathcal{N}^{\epsilon}$, then $\kappa_n\in\Sel_{\mathcal{F}(n)}(K,\mathbf{T}/I_n)$; set $\kappa_{\mathfrak{P},n}\defeq s_\mathfrak{P}(\kappa_n)\in\Sel_{\mathcal{F}_\mathfrak{P}(n)}(K,T_{\mathfrak{P}}/I_{\mathfrak{P},n})$; 
\item if $n\in \mathcal{N}^{\epsilon+1}$, then $\lambda_n\in\Lambda/I_n$; set $\lambda_{\mathfrak{P},n}\defeq s_\mathfrak{P}(\lambda_n)\in \mathcal{O}_\mathfrak{P}/I_{\mathfrak{P},n}$. 
\end{itemize}
Here $I_n$ is the ideal introduced in \eqref{I-eq}.
\begin{proposition} \label{non-triv}
There exists a finite set $\Sigma_\mathrm{EX}$ of height $1$ prime ideals of $\Lambda$ such that for all $\mathfrak P\notin\Sigma_\mathrm{EX}$ the following conditions hold true: 
\begin{enumerate}\label{parprop}
\item $\lambda_{\mathfrak{P},1}\neq 0$ if $1\in\mathcal{N}^{\epsilon+1}$ and $\kappa_{\mathfrak{P},1}\neq0$ if $1\in\mathcal{N}^{\epsilon}$;
\item $(\kappa_{\mathfrak{P}},\lambda_{\mathfrak{P}})$ is a non-trivial bipartite Euler system of parity $\epsilon$ for $(T_{\mathfrak{P}},\mathcal{F}_\mathfrak{P},\mathcal{P})$;
\item the $\mathcal{O}_\mathfrak{P}$-corank of $\Sel_{\mathcal{F}_\mathfrak{P}}(K,\mathbf{A})$ is $0$ if $1\in\mathcal{N}^{\epsilon+1}$ and is $1$ if $1\in\mathcal{N}^{\epsilon}$.
\end{enumerate}
\end{proposition}
\begin{proof} The class $\lambda_1$ if $1\in\mathcal{N}^{\epsilon+1}$ or the class $\kappa_1$ if $1\in\mathcal{N}^{\epsilon}$ can be divisible only by a finite number of height $1$ prime ideals $\mathfrak{P}$ of $\Lambda$ (in the second case, use the fact that $\Sel_\mathcal{F}(K,\mathbf{T})$ is a finitely generated $\Lambda$-module), so condition (1) holds for all but finitely many $\mathfrak P$ as above and shows that, for such a $\mathfrak P$, the pair $(\kappa_{\mathfrak{P}},\lambda_{\mathfrak{P}})$ is a non-trivial bipartite Euler system for $(T_{\mathfrak{P}},\mathcal{F}_\mathfrak{P},\mathcal{P})$. For any $\mathfrak{P}$ as before, by Theorem \ref{thmDVR} the $\mathcal{O}_\mathfrak{P}$-corank of $\Sel_{\mathcal{F}_\mathfrak{P}}(K,A_\mathfrak{P})$ is $1$ if $\kappa_{\mathfrak{P},1}\neq 0$ and is $0$ if $\lambda_{\mathfrak{P},1}\neq 0$; by construction, the first (respectively, second) case occurs when $1\in\mathcal{N}^{\epsilon}$ (respectively, $1\in\mathcal{N}^{\epsilon+1}$), which proves condition (3). \end{proof}

\begin{corollary} \label{corolambdarank}
The $\Lambda$-corank of $\Sel_\mathcal{F}(K,\mathbf{A})$ is $0$ if $1\in\mathcal{N}^{\epsilon+1}$ and is $1$ if $1\in\mathcal{N}^{\epsilon}$. 
\end{corollary}

\begin{proof} Since part (3) of Proposition \ref{parprop} holds for all but finitely many height $1$ prime ideals of $\Lambda$, the structure theorem for $\Lambda$-modules combined with the control theorem in part (3) of Assumption \ref{assLambda} shows that the $\Lambda$-corank of $\Sel_\mathcal{F}(K,\mathbf{A})$ is $0$ if $\lambda_1\neq 0$ and is $1$ if $\kappa_1\neq 0$. The corollary follows from Assumption \ref{assBESLambda}. \end{proof}

Motivated by Proposition \ref{parprop}, set 
\begin{equation} \label{N-sets-eq}
\mathcal{N}^{\deff}\defeq\mathcal{N}^{\epsilon+1},\quad\mathcal{N}^{\ind}\defeq\mathcal{N}^{\epsilon}. 
\end{equation}
For $\mathfrak{P}\notin\Sigma_\mathrm{EX}$, the sets in \eqref{N-sets-eq} are the definite and indefinite sets for the bipartite Euler system $(\kappa_\mathfrak{P},\lambda_\mathfrak{P})$, as introduced in \S\ref{BES-Artinian}; however, for $\mathfrak{P}\in\Sigma_\mathrm{EX}$ this may no longer be true. 

The map $i_\mathfrak{P}$ in \eqref{cosp} takes the maximal divisible subgroup of $\Sel_{\mathcal{F}_\mathfrak{P}}(K,A_\mathfrak{P})$ to the maximal $\mathfrak{m}_\Lambda$-divisible subgroup $\mathrm{div}_{\mathfrak{m}_\Lambda}\bigl(\Sel_\mathcal{F}(K,\mathbf{A})\bigr)$ of $\Sel_\mathcal{F}(K,\mathbf{A})$; thus, $i_{\mathfrak P}$ induces a map
\[
i_\mathfrak{P}:\Sha_{\mathcal{F}_\mathfrak{P}}(K,A_\mathfrak{P})\longrightarrow \Sha_\mathcal{F}(K,\mathbf{A})[\mathfrak{P}],
\]
still denoted by the same symbol. Therefore, with notation as in \eqref{X-sha-eq}, Pontryagin duality gives a map 
\begin{equation} \label{cospShadual}
i^\vee_\mathfrak{P}:X /\mathfrak{P} X\longrightarrow X({T}_{\mathfrak{P}}).
\end{equation}
Taking $\mathfrak{P}$-torsion in the short exact sequence of $\Lambda$-modules 
\[
0\longrightarrow\mathrm{div}_{\mathfrak{m}_\Lambda}\bigl(\Sel_\mathcal{F}(K,\mathbf{A})\bigr)\longrightarrow\Sel_\mathcal{F}(K,\mathbf{A})\longrightarrow\Sha_\mathcal{F}(K,\mathbf{A})\longrightarrow0
\]
defining $\Sha_\mathcal{F}(K,\mathbf{A})$, we obtain an exact sequence of $\Lambda$-modules
\begin{equation} \label{shaP}
0\longrightarrow\mathrm{div}_{\mathfrak{m}_\Lambda}\bigl(\Sel_\mathcal{F}(K,\mathbf{A})\bigr)[\mathfrak{P}]\longrightarrow\Sel_\mathcal{F}(K,\mathbf{A})[\mathfrak{P}]\longrightarrow\Sha_\mathcal{F}(K,\mathbf{A})[\mathfrak{P}].
\end{equation}

\begin{lemma} \label{right-onto-lemma}
The rightmost map in \eqref{shaP} is surjective. 
\end{lemma}

\begin{proof} Fix $x\in \Sha_\mathcal{F}(K,\mathbf{A})[\mathfrak{P}]$ and pick $y\in \Sel_\mathcal{F}(K,\mathbf{A})$ mapping to $x$. Choose $F\in\Lambda$ such that $\mathfrak{P}=(F)$. Since $Fy$ maps to $Fx=0$ in $\Sha_\mathcal{F}(K,\mathbf{A})$, we see that $Fy$ belongs to $\mathrm{div}_{\mathfrak{m}_\Lambda}\bigl(\Sel_\mathcal{F}(K,\mathbf{A})\bigr)$. In particular, the maximal $\mathfrak{m}_\Lambda$-divisible subgroup of $\Sel_\mathcal{F}(K,\mathbf{A})$ is $F$-divisible; in other words, the multiplication-by-$F$ map is surjective, so there exists $z\in\mathrm{div}_{\mathfrak{m}_\Lambda}\bigl(\Sel_\mathcal{F}(K,\mathbf{A})\bigr)$ with $Fy=Fz$. This shows that $y=z+w$ for some $w\in \Sel_\mathcal{F}(K,\mathbf{A})$ with $Fw=0$. It follows that $w$ is an element of $\Sel_\mathcal{F}(K,\mathbf{A})[\mathfrak{P}]$ mapping to $x$ under the rightmost map in \eqref{shaP}. \end{proof}

\begin{proposition} \label{CTSha1}
There is a finite set $\Sigma$ of height $1$ prime ideals of $\Lambda$ such that for all height $1$ prime ideals $\mathfrak{P}$ of $\Lambda$ not in $\Sigma$ the kernel and the cokernel of the map $i^\vee_\mathfrak{P}$ in $\eqref{cospShadual}$ are finite of order bounded by a constant depending on $[\mathcal{O}_{\mathfrak{P}}:\Lambda/\mathfrak{P}]$. 
\end{proposition}

Notice that the constant whose existence is stated in the proposition does not depend on the ideal $\mathfrak{P}$ itself, but only on $[\mathcal{O}_{\mathfrak{P}}:\Lambda/\mathfrak{P}]$.

\begin{proof} With notation as in Assumption \ref{assLambda} and Proposition \ref{parprop}, we prove the analogous result for $i_\mathfrak{P}$ with $\mathfrak{P}$ outside the finite set $\Sigma\defeq\Sigma_\mathrm{CT}\cup\Sigma_\mathrm{EX}$ of height $1$ prime ideals of $\Lambda$: the desired result for $i^\vee_\mathfrak{P}$ will follow by Pontryagin duality. 

Let us fix a height $1$ prime ideal $\mathfrak P$ of $\Lambda$ such that $\mathfrak P\notin\Sigma$. Writing $\varphi$ and $\psi$ for the maps induced by $i_\mathfrak{P}$, in light of Lemma \ref{right-onto-lemma} there is a commutative diagram  with exact rows
\[
\xymatrix{
0\ar[r] & \mathrm{div}\bigl(\Sel_{\mathcal{F}_\mathfrak{P}}(K,A_\mathfrak{P})\bigr)\ar[r]\ar[d]^-{\varphi} & \Sel_{\mathcal{F}_\mathfrak{P}}(K,A_\mathfrak{P})\ar[r]\ar[d]^-{i_\mathfrak{P}} & \Sha_{\mathcal{F}_\mathfrak{P}}(K,A_\mathfrak{P})\ar[d]^-{\psi}\ar[r]&0\\
0\ar[r]&\mathrm{div}_{\mathfrak{m}_\Lambda}\bigl(\Sel_\mathcal{F}(K,\mathbf{A})\bigr)[\mathfrak{P}]\ar[r]&\Sel_\mathcal{F}(K,\mathbf{A})[\mathfrak{P}]\ar[r]&\Sha_\mathcal{F}(K,\mathbf{A})[\mathfrak{P}]\ar[r]&0
}
\]
that gives rise to an exact sequence 
\[
\begin{split}
0\longrightarrow \ker(\varphi)&\longrightarrow\ker(i_\mathfrak{P})\longrightarrow \ker(\psi)\\&\hskip3mm\longrightarrow 
\mathrm{coker}(\varphi)\longrightarrow \mathrm{coker}(i_\mathfrak{P})\longrightarrow\mathrm{coker}(\psi)\longrightarrow0.
\end{split}
\]
Since $\mathfrak{P}\notin\Sigma_\mathrm{CT}$, Assumption \ref{assLambda} forces $\ker(i_\mathfrak{P})$ and $\mathrm{coker}(i_\mathfrak{P})$, hence $\mathrm{coker}(\psi)$ as well, to possess the required properties. It remains to show that the same is true of 
$\mathrm{coker}(\varphi)$. We actually show that the fact that $\mathfrak{P}\notin\Sigma_\mathrm{EX}$ implies the surjectivity of $\varphi$. First of all, by Corollary \ref{corolambdarank}, the Pontryagin dual of $\mathrm{div}_{\mathfrak{m}_\Lambda}\bigl(\Sel_\mathcal{F}(K,\mathbf{A})\bigr)$ is isomorphic to $\Lambda^r$ with $r=1$ if $1\in\mathcal{N}^{\epsilon}$ and $r=0$ if $1\in\mathcal{N}^{\epsilon+1}$, so the dual of $\mathrm{div}_{\mathfrak{m}_\Lambda}\bigl(\Sel_\mathcal{F}(K,\mathbf{A})\bigr)[\mathfrak{P}]$ is isomorphic to $(\Lambda/\mathfrak{P})^r$. Similarly, Proposition \ref{parprop} ensures that if $\mathfrak{P}\notin \Sigma$, then the Pontryagin dual
of $\mathrm{div}\bigl(\Sel_{\mathcal{F}_\mathfrak{P}}(K,A_\mathfrak{P})\bigr)$ is isomorphic to $\mathcal{O}_\mathfrak{P}^r$ for the same $r$ as before and the dual map $\varphi^\vee$ corresponds to the injection $(\Lambda/\mathfrak{P})^r\hookrightarrow\mathcal{O}_\mathfrak{P}^r$, so it is 
injective. By Pontryagin duality, it follows that $\varphi$ is surjective, as desired. \end{proof}

The map $i_\mathfrak{P}^\vee$ in \eqref{cospShadual} induces a map
\begin{equation} \label{ipvee}
i_\mathfrak{P}^\vee: X/\mathfrak{P}X \otimes_{\Lambda/\mathfrak{P}} \mathcal{O}_\mathfrak{P} \longrightarrow X(T_\mathfrak{P}) \otimes_{\Lambda/\mathfrak{P}} \mathcal{O}_\mathfrak{P},
\end{equation}
denoted by the same symbol. Composing with the map $X(T_\mathfrak{P}) \otimes_{\Lambda/\mathfrak{P}} \mathcal{O}_\mathfrak{P}\rightarrow 
X(T_\mathfrak{P})$ defined on pure tensors by $x\otimes a\mapsto ax$, we obtain a further map
\begin{equation} \label{ipvee2}
i_\mathfrak{P}^\vee: X/\mathfrak{P}X \otimes_{\Lambda/\mathfrak{P}} \mathcal{O}_\mathfrak{P} \longrightarrow X(T_\mathfrak{P}).
\end{equation}
For a height $1$ prime ideal $\mathfrak{P}$ of $\Lambda$ and an integer $j \geq 0$, we introduce the following notation:  
\begin{itemize}
\item if $\mathfrak{P} = (f)$ with $f \neq \pi$, set $\mathfrak{P}_j\defeq(f+\pi^j)$ and write $e_\mathfrak{P}$ for the ramification index of $\Frac(\mathcal{O}_{\mathfrak{P}})$ over $\Frac(\mathcal{O})$;
\item if $\mathfrak{P} = (\pi)$, set $\mathfrak{P}_j\defeq(\pi + T^j)$ and $e_\mathfrak{P}\defeq1$.
\end{itemize}
There exists an integer $N(\mathfrak{P})\geq1$ such that for all integers $j \geq N(\mathfrak{P})$ the following properties hold:
\begin{itemize}
\item $\mathfrak{P}_j$ is a height $1$ prime ideal of $\Lambda$;
\item $\mathfrak{P}_j\notin\Sigma$;
\item if $\mathfrak{P}\neq(\pi)$, then there is an isomorphism $\Lambda/\mathfrak{P}_j\simeq\Lambda/\mathfrak{P}$.
\end{itemize}
For all integers $j \geq N(\mathfrak{P})$, set
\[
\mathcal{O}_j\defeq\mathcal{O}_{\mathfrak{P}_j},\quad T_j\defeq {T}_{\mathfrak{P}_j},\quad \lambda_j\defeq\lambda_{\mathfrak{P}_j},\quad \lambda_{j,n}\defeq(\lambda_{\mathfrak{P}_j})_n,\quad s_j=s_{\mathfrak{P}_j},\quad i_j=s_{\mathfrak{P}_j},\quad i_j^\vee=i_{\mathfrak{P}_j}^\vee.
\]
Moreover, denote by $\pi_j$ a uniformizer of $\mathcal{O}_j$ and by $e_j$ the ramification index of $\Frac(\mathcal{O}_j)$ over $\Frac(\mathcal{O})$. Finally, let
\begin{equation} \label{s-j-eq}
s_j:\Lambda\longrightarrow\Lambda/\mathfrak{P}_j\longmono\mathcal{O}_j
\end{equation}
be the specialization map.

Now we give a first result on the structure of $X$. We need some preliminary lemmas.

\begin{lemma} \label{kerCokerComposition}
Let $R$ be a commutative ring and for all $j\in\N$ let $A_j,B_j,C_j$ be $R$-modules. If for all $j\in\N$ the maps $f_j:A_j \rightarrow B_j$ and $g_j:B_j \rightarrow C_j$ are $R$-modules homomorphisms with finite kernel and cokernel bounded independently of $j$, then the same holds for $g_j\circ f_j$ for all $j\in\N$.
\end{lemma}

\begin{proof} Fix $j\in\N$. There are equalities
\[
\ker(g_j \circ f_j)=f_j^{-1}(\ker g_j)=\bigcup_{y \in \ker(g_j)} f_j^{-1}(y)
\]
and, since $\# f_j^{-1}(y) \leq \# \ker (f_j)$ for all $y \in B_j$, there is an inequality
\[
\#\ker(g_j \circ f_j) \leq \sum_{y \in \ker (g_j)} \#\ker(f_j)=\# \ker(g_j)\cdot\# \ker(f_j),
\]
which proves the claim of the proposition for $\ker(g_j \circ f_j)$. On the other hand, we also have
\[
\begin{aligned}
\# \coker(g_j \circ f_j)&=\#\Bigl(C_j\big/g_j\bigl(f_j(A_j)\bigr)\!\Bigr)=\#\bigl(C_j/g_j(B_j)\bigr)\cdot \#\Bigl(g_j(B_j)\big/g_j\bigl(f_j(A_j)\bigr)\!\Bigr)\\
& \leq\#\coker(g_j)\cdot\#\coker(f_j),
\end{aligned}
\]
which concludes the proof. \end{proof}

\begin{lemma} \label{extScalars}
Let $\mathfrak{P}$ be a height $1$ prime of $\Lambda$ not contained in $\Sigma$. The kernel and the cokernel of $i_j^\vee$ are finite and bounded by a constant depending on $\mathfrak{P}$ but not on $j$ for $j\geq N(\mathfrak{P})$.
\end{lemma}

\begin{proof} The lemma is trivial for $\mathfrak{P}=(\pi)$, as, in this case, $\Lambda/\mathfrak{P}_j=\mathcal{O}_j$ for all integers $j \geq 1$. Thus, we can assume $\mathfrak{P} \neq (\pi)$. The claim for the cokernel follows from Proposition \ref{CTSha1}, as the cokernel of \eqref{ipvee2} is a quotient of the cokernel of \eqref{cospShadual}; therefore, we need to prove the statement for the kernel only. To lighten the notation, for the rest of the proof we set
\begin{equation} \label{j-notation}
R_j\defeq\Lambda/\mathfrak{P}_j,\quad S_j\defeq\mathcal{O}_{\mathfrak{P}_j},\quad M_j\defeq X/\mathfrak{P}_jX,\quad N_j\defeq X(T_{\mathfrak{P}_j}),\quad h_j:M_j \rightarrow N_j.
\end{equation}
By Lemma \ref{kerCokerComposition}, we just need to prove the statement for the two maps
\[
f_j:M_j \otimes_{R_j} S_j\longrightarrow N_j \otimes_{R_j} S_j,\quad g_j:N_j \otimes_{R_j} S_j \longrightarrow N_j,
\]
where $f_j$ is the $S_j$-linear extension of $i_j^\vee$ and $g_j$ is defined on pure tensors by $n\otimes s\mapsto ns$. 
    
We first prove the statement for $f_j$. From the tautological short exact sequences
\[
0 \longrightarrow \ker(h_j) \longrightarrow M_j \longrightarrow \im (h_j) \longrightarrow 0
\]
and
\[
0 \longrightarrow \im(h_j) \longrightarrow N_j \longrightarrow \coker(h_j) \longrightarrow 0,
\]
we get exact sequences
\begin{equation} \label{a-eq}
\ker(h_j) \otimes_{R_j} S_j \longrightarrow M_j \otimes_{R_j} S_j \longrightarrow \im(h_j) \otimes_{R_j} S_j \longrightarrow 0
\end{equation}
and
\begin{equation} \label{b-eq}
\begin{split}
\Tor_1^{R_j}\bigl(\coker(h_j),S_j\bigr)& \longrightarrow \im(h_j) \otimes_{R_j} S_j\\ &\hskip 3mm\longrightarrow N_j \otimes_{R_j} S_j \longrightarrow \coker(h_j) \otimes_{R_j} S_j \longrightarrow 0.
\end{split}
\end{equation}
Notice that $f_j$ is the composition of the second maps in \eqref{a-eq} and \eqref{b-eq}, so, again by Lemma \ref{kerCokerComposition}, we just need to prove that the orders of $\ker(h_j)\otimes_{R_j} S_j$ and $\Tor_1^{R_j}\bigl(\coker(h_j),S_j\bigr)$ are bounded independently of $j$. The claim for $\ker(h_j) \otimes_{R_j} S_j$ follows from Proposition \ref{CTSha1} because $[S_j:R_j]$ is constant over all $j\geq N(\mathfrak{P})$. As for the second claim, the rings $R_j$ (and, consequently, the rings $S_j$) are isomorphic as $\Z_p$-algebras for $j \geq N(\mathfrak{P})$, so we can fix finite presentations
\[
R_j^m \xlongrightarrow{\varphi} R_j^k \longrightarrow S_j \longrightarrow 0
\]
in which $m$ and $k$ do not depend on $j$. Then for all $j$ there is a short exact sequence
\[
0 \longrightarrow R_j^m/\ker(\varphi) \longrightarrow R_j^k \longrightarrow S_j \longrightarrow 0;
\]
tensoring with $\coker(h_j)$ over $R_j$, we get an exact sequence
\[
\begin{split}
0 \longrightarrow\Tor_1^{R_j}\bigl(\coker(h_j),S_j\bigr) &\longrightarrow R_j^m/\ker(\varphi) \otimes \coker(h_j)\\ &\hskip 3mm\longrightarrow \coker(h_j)^k \longrightarrow S_j \otimes\coker(h_j) \longrightarrow 0
\end{split}
\]
(here the unadorned $\otimes$ stands for $\otimes_{R_j}$). Therefore, there are inequalities
\[
\#\Tor_1^{R_j}\bigl(\coker(h_j),S_j\bigr) \leq \#\bigl(R_j^m/\ker(\varphi) \otimes_{R_j} \coker(h_j)\bigr) \leq \#\coker(h_j)^m
\]
and our claim follows from Proposition \ref{CTSha1}. This establishes the desired result for $f_j$.

Now we prove the statement for $g_j$. The map $g_j$ is clearly surjective, so we only need to check that its kernel has order bounded independently of $j$. First of all, consider the map of $S_j$-modules $S_j \otimes_{R_j} S_j\rightarrow S_j$ given on pure tensors by $s\otimes t\mapsto st$; denote by $J_j$ the kernel of this map. If we write $K_j$ for the fraction field of $S_j$, then there is a short exact sequence
\[
0 \longrightarrow  J_j \otimes_{R_j} K_j \longrightarrow (S_j \otimes_{R_j} S_j) \otimes_{R_j} K_j \longrightarrow S_j \otimes_{R_j} K_j \longrightarrow 0.
\]
Since $S_j/R_j$ is finite, $(S_j \otimes_{R_j} S_j) \otimes_{R_j} K_j\simeq S_j\otimes_{R_j}K_j\simeq K_j$ and the second map, being surjective, is an isomorphism. Therefore, $J_j \otimes_{R_j} K_j=0$, which implies that $J_j$ is a finitely generated torsion $R_j$-module and, hence, a finitely generated torsion $S_j$-module. Thus, $J_j$ is finite and, since the $R_j$'s (and, consequently, the $S_j$'s) are isomorphic as $\Z_p$-algebras to each other for $j\geq N(\mathfrak{P})$, its order is bounded by a constant that does not depend on $j$. Now, we observe that $M_j$ is finitely generated over $R_j$ by a number of elements independent of $j$, as it is smaller than or equal to the number of generators of $X$ over $\Lambda$. On the other hand, $\coker(h_j)$ is finite and bounded by a constant independent of $j$, so $N_j$ too is finitely generated over $R_j$ by a number of elements independent of $j$, and the same holds true over $S_j$. Then there is an isomorphism
\[
N_j\simeq\bigoplus_{i=1}^t S_j/\pi_j^{k_i} S_j
\]
of $S_j$-modules, where $t$ does not depend on $j$. Therefore, it suffices to prove that, for an integer $k\geq1$, the kernel of the natural map $g_j:\bigl(S_j/\pi_j^k S_j\bigr)\otimes_{R_j}S_j\rightarrow S_j/\pi_j^k S_j$ is bounded independently of $j$. To this end, consider the commutative diagram with exact rows 
\begin{equation} \label{S-comm-eq}
\xymatrix@C=42pt{&S_j\otimes_{R_j}S_j \ar[r]^-{(\pi_j^k \otimes 1)\cdot} \ar@{->>}[d] &S_j\otimes_{R_j}S_j \ar[r] \ar@{->>}[d] &S_j/\pi_j^k S_j\otimes_{R_j}S_j \ar[r] \ar@{->>}[d]^-{g_j} &0 \\
    0 \ar[r] &S_j \ar[r]^-{\pi_j^k\cdot} & S_j \ar[r] & S_j/\pi_j^kS_j \ar[r] &0}
\end{equation}
in which the vertical arrows are given by $x\otimes s\mapsto sx$ on pure tensors.
By the snake lemma and the surjectivity of the left vertical map in \eqref{S-comm-eq}, $\ker(g_j)$ is a quotient of $J_j$. The claim follows because we already proved that the order of $J_j$ is bounded independently of $j$. \end{proof}

\begin{lemma} \label{pseudoKerCoker}
Fix a height $1$ prime ideal $\mathfrak{P}$ of $\Lambda$. Let $M$ and $N$ be finitely generated torsion $\Lambda$-modules and let $f:M \rightarrow N$ be a pseudo-isomorphism. For all integers $j\geq0$ such that $\mathfrak{P}_j\notin\Supp(M)$, the map
\[ 
f_j:M \otimes_\Lambda \mathcal{O}_j \longrightarrow N \otimes_\Lambda \mathcal{O}_j
\]
has finite kernel and cokernel of order bounded independently of $j$.
\end{lemma}

\begin{proof} Since $M$ and $N$ are pseudo-isomorphic, there are exact sequences of $\Lambda$-modules
\[
M \overset{f}\longrightarrow N \longrightarrow B \longrightarrow 0,\quad N \longrightarrow M \longrightarrow B' \longrightarrow 0,
\]
where $B$ and $B'$ are finite. Fix an integer $j \geq 0$ such that $\mathfrak{P}_j \notin \Supp(M)$. There are exact sequences of finitely generated torsion $\mathcal{O}_j$-modules
\[
M \otimes_\Lambda \mathcal{O}_j \overset{f_j}\longrightarrow N \otimes_\Lambda \mathcal{O}_j \longrightarrow B \otimes_\Lambda \mathcal{O}_j \longrightarrow 0
\]
and
\[ N \otimes_\Lambda \mathcal{O}_j \longrightarrow M \otimes_\Lambda \mathcal{O}_j \longrightarrow B' \otimes_\Lambda \mathcal{O}_j \longrightarrow 0,
\]
where $B \otimes_\Lambda \mathcal{O}_j$ and $B' \otimes_\Lambda \mathcal{O}_j$ are finite of order bounded independently of $j$. Indeed:
\begin{itemize}
\item if $\mathfrak{P} \neq (\pi)$, then the desired boundedness is a consequence of the canonical surjection $B \otimes_{\Z_p}\mathcal{O}_j\twoheadrightarrow B \otimes_\Lambda \mathcal{O}_j$ and the isomorphism $B \otimes_{\Z_p}\mathcal{O}_j\simeq B\otimes_{\Z_p}\mathcal{O}_{\mathfrak{P}}$ induced by the isomorphism of rings $\mathcal{O}_\mathfrak{P} \simeq \mathcal{O}_j$;      
\item if $\mathfrak{P}=(\pi)$, then $B \otimes_\Lambda \mathcal{O}_j=B \otimes_\Lambda \Lambda/\mathfrak{P}_j$ is a quotient of $B$.
\end{itemize}
This proves the claim of the lemma on the cokernel of $f_j$; it remains to deal with its kernel. Set $a_j\defeq\#\ker(f_j)$ and for $(\xi,\Xi)\in\bigl\{(m,M),(n,N),(b,B),(b',B')\bigr\}$ put $\xi_j\defeq\#(\Xi\otimes_\Lambda \mathcal{O}_j)$. Choose an integer $k\geq0$ such that $b_j,b'_j \leq k$ for all $j$. Then 
\[
a_j=\frac{m_j b_j}{n_j} \leq \frac{m_j}{n_j} \cdot k \leq k^2,
\]
where the last inequality follows from the inequality $n_jb'_j\geq m_j$. Since $k^2$ is independent of $j$, this proves the statement of the lemma for the kernel of $f_j$. \end{proof}

\begin{lemma} \label{parityNExp}
Let $R$ be a discrete valuation ring with uniformizer $\pi$. Assume that for any integer $m\gg0$ we are assigned finite $R$-modules $A_m$, $B_m$, $M_m$, integers $e_i,g_i\geq1$ for $i\in\{1,\dots,n\}$ and an exact sequence
\[
0\longrightarrow A_m \longrightarrow \bigoplus_{i=1}^n  (R/\pi^{me_i}R)^{g_i} \xlongrightarrow{f_m} M_m \oplus M_m \longrightarrow B_m \longrightarrow 0
\]
of finite $R$-modules, where $\max\{\#A_m,\#B_m\}\leq C$ for some integer $C\geq0$ independent of $m$. Then $g_i$ is even for all $i\in\{1,\dots,n\}$.
\end{lemma}

\begin{proof} We shall repeatedly use the following two facts, both of which are easy consequences of the structure theorem for finitely generated modules over principal ideal domains:
\begin{enumerate}
\item if $M$ is a finitely generated torsion $R$-module and $N$ is an $R$-submodule of $M$, then there exists an injective map of $R$-modules $M/N\hookrightarrow M$;
\item if $M \simeq \bigoplus_{i=1}^n R/\pi^{a_i}R$ and $N\simeq \bigoplus_{i=1}^n R/\pi^{b_i}R$ are two finitely generated torsion $R$-modules with $a_1\geq a_2\geq\dots\geq a_n$ and $b_1\geq b_2\geq\dots\geq b_n$, then there is an injection of $R$-modules $N \hookrightarrow M$ if and only if $b_i\leq a_i$ for all $i\in\{1,\dots,n\}$.
\end{enumerate} 
Without loss of generality, we can assume that $g_i$ is odd for $i\in\{1,\dots,s\}$ and even for $i\in\{s+1,\dots,n\}$. We also assume $e_1\geq e_2\geq\dots\geq e_s$. Since $\#A_m \leq C$, there exists an integer $k \geq 0$ such that $\pi^k A_m=0$, and then there are inclusions
\[
A_m\subset\bigoplus_{i=1}^n\bigl(\pi^{me_i-k}R/\pi^{me_i}R\bigr)^{g_i}\subset\bigoplus_{i=1}^n\bigl(R/\pi^{me_i}R\bigr)^{g_i}.
\]
By (1), there are two injections $\varphi$ and $\psi$ sitting in the chain of maps of $R$-modules 
\[
\bigoplus_{i=1}^n\bigl(R/\pi^{me_i-k}R\bigr)^{g_i}\overset\varphi\longmono \bigoplus_{i=1}^n\bigl(R/\pi^{me_i}R\bigr)^{g_i}/A_m\simeq\im(f_m)\overset\psi\longmono \bigoplus_{i=1}^n\bigl(R/\pi^{me_i}R\bigr)^{g_i}.
\]
It follows from (2) that there is an isomorphism
\[
\im (f_m) \simeq \bigoplus_{i=1}^n \bigoplus_{j=1}^{g_i} R/\pi^{c_{i,j}(m)}R,
\]
where $me_i-k \leq c_{i,j}(m) \leq me_i$ for all $j\in\{1,\dots,g_i\}$. Without loss of generality, we can assume that $c_{i,1}(m)\geq c_{i,2}(m)\geq\dots\geq c_{i,g_i}(m)$ for all $i\in\{1,\dots,n\}$. In light of the inclusion $\im(f_m) \subset M_m \oplus M_m$, again by (2) we obtain the inequality
\[
\length(M_m)\geq\sum_{i=1}^s\sum_{j=1}^{(g_i-1)/2} c_{i,2j-1}(m)+\sum_{i=1}^{s'}c_{2i-1,g_i}(m)+\sum_{i=s+1}^n\sum_{j=1}^{g_i/2}c_{i,2j-i}(m),
\]
where
\[
s'\defeq \begin{cases} s/2 &\text{if $s$ is even},\\[2mm](s+1)/2 & \text{if $s$ is odd}. \end{cases}
\]
Then we get
\begin{equation} \label{diseq}
\begin{split}
\length(B_m) &=2\cdot\length(M_m) - \length(\im(f_m))\\
&\geq 2\cdot\sum_{i=1}^s \sum_{j=1}^{(g_i-1)/2} c_{i,2j-1}(m) + 2\cdot\sum_{i=1}^{s'} c_{2i-1,g_i}(m) + 2\cdot\sum_{i=s+1}^n \sum_{j=1}^{g_i/2} c_{i,2j-i}(m) 
\\&\quad-\sum_{i=1}^n \sum_{j=1}^{g_i} c_{i,j}(m).
\end{split}
\end{equation}
If $s>1$, then \eqref{diseq} implies
\[
\length(B_m) \geq c_{1,g_1}(m)-c_{2,g_2}(m) \geq me_1-k -me_2=m(e_1-e_2)-k,
\]
which produces a contradiction. On the other hand, if $s=1$, then \eqref{diseq} gives
\[
\length(B_m) \geq c_{1,g_1}(m) \geq me_1 -k,
\]
which also yields a contradiction. It follows that $s=0$, and the lemma is proved. \end{proof}

In the lines below, recall the specialization map $s_j:\Lambda\rightarrow\mathcal O_j$ from \eqref{s-j-eq}.

\begin{proposition} \label{propX=M+M}
There exists a torsion $\Lambda$-module $M$ such that $X\sim M\oplus M$.
\end{proposition}

\begin{proof} Fix a prime ideal $\mathfrak P$ of $\Lambda$ such that $\mathfrak{P}\,|\,\charr(X)$. We assume $X \sim N=N' \oplus N_\mathfrak{P}$ where $\charr (N')$ is coprime with $\mathfrak{P}$ and $N_\mathfrak{P} \simeq \bigoplus_{i=1}^n \Lambda/\mathfrak{P}^{e_i}$. We need to show that all the $e_i$ are even. 
    
\texttt{Step 1.} Keep notation as in \eqref{j-notation}. We claim that for all $j\geq N(\mathfrak{P})$ the composition
\begin{equation} \label{Str1}
N_\mathfrak{P} \otimes_\Lambda \mathcal{O}_j \longrightarrow N\otimes_\Lambda \mathcal{O}_j \longrightarrow X \otimes_\Lambda \mathcal{O}_j \longrightarrow X(T_j)=M_j \oplus M_j
\end{equation}
has finite kernel and cokernel bounded independently of $j$ (note that the last equality follows from \eqref{ShaX(T)}). By Lemma \ref{kerCokerComposition}, we need to prove the previous claim for each of the involved maps. The short exact sequence
\[
0 \longrightarrow N_\mathfrak{P} \longrightarrow N \longrightarrow N' \longrightarrow 0
\]
splits, so it induces for all $j$ a short exact sequence
\[
0 \longrightarrow N_\mathfrak{P} \otimes_\Lambda \mathcal{O}_j \longrightarrow N \otimes_\Lambda \mathcal{O}_j\longrightarrow N' \otimes_\Lambda \mathcal{O}_j \longrightarrow 0.
\]
To prove that the first map in \eqref{Str1} satisfies the claim, we need to check that $N' \otimes_\Lambda\mathcal{O}_j$ is finite of order independent of $j$. We notice that $N' \otimes_\Lambda\mathcal{O}_j$ is a finite direct sum of modules of the form $\mathcal{O}_j\big/s_j(\mathfrak{p}^N)\mathcal{O}_j$, where $\mathfrak{p} \neq \mathfrak{P}$. All these modules are finite and their orders are bounded independently of $j$, as we explain below.

$\bullet$\quad Suppose $\mathfrak{P}\neq(\pi)$. Then, writing $\mathfrak{p}^N=(g)$ and $\mathfrak{P}=(f)$, there exists an integer $S \geq 0$ such that $\pi^S \in \mathfrak{p} + \mathfrak{P}$, \emph{i.e.}, there exist $\lambda_1,\lambda_2 \in \Lambda$ such that $\pi^S=\lambda_1g+\lambda_2f$. Then
\[
\lambda_1g+\lambda_2 (f+\pi^j)=\pi^S + \lambda_2\pi^j \in \mathfrak{p}^N+\mathfrak{P}_j
\]
and $\bigl(\pi_j^{S e_\mathfrak{P}}\bigr)\subset s_j(\mathfrak{p}^N)$. The order of $\mathcal{O}_j\big/\bigl(\pi_j^{S e_\mathfrak{P}}\bigr)$ is independent of $j$ because $\mathcal{O}_j\simeq\mathcal{O}$ as rings. 

$\bullet$\quad Suppose $\mathfrak{P} =(\pi)$. Then $s_j(T)=\pi_j$ and, writing $\mathfrak{p}^{N}=(g)$ with $g=T^S + \pi g'$ a distinguished polynomial, we have $s_j(\mathfrak{p}^N) \mathcal{O}_j=(\pi_j^S)$.

Then the first map in \eqref{Str1} enjoys the claimed property. On the other hand, by Lemma \ref{pseudoKerCoker}, the second map in \eqref{Str1} also satisfies the claim. Finally, the last map in \eqref{Str1} has the desired property by Lemma \ref{extScalars}.
    
\texttt{Step 2.} The proof now follows from Lemma \ref{parityNExp}. As before, we distinguish two cases. 

$\bullet$\quad Suppose $\mathfrak{P}\neq(\pi)$. Then $\mathfrak{P}^{e_i}+\mathfrak{P}_j=(\pi^{je_i})+\mathfrak{P}_j$, so $s_j(\mathfrak{P}^{e_i}) \mathcal{O}_j=(\pi_j^{e_\mathfrak{P} j e_i})$. It follows that there is an isomorphism $N_\mathfrak{P} \otimes_\Lambda \mathcal{O}_j \simeq \bigoplus_{i=1}^n \mathcal{O}_\mathfrak{P}\big/\bigl(\pi_\mathfrak{P}^{e_\mathfrak{P}je_i}\bigr)$ of $\mathcal{O}_\mathfrak{P}$-modules, and we deduce from Lemma \ref{parityNExp} that the number of $i$ such that $e_i=e$ is even for all integers $e\geq1$. 

$\bullet$\quad Suppose $\mathfrak{P} = (\pi)$. Then $\mathcal{O}_j$ is totally ramified of degree $j$ over $\mathcal{O}$, so $s_j(\mathfrak{P}^{e_i}) \mathcal{O}_j=(\pi_j^{je_i})$. Therefore, there is an isomorphism $N_\mathfrak{P} \otimes_\Lambda \mathcal{O}_j \simeq \bigoplus_{i=1}^n \mathcal{O}/(\pi^{je_i})$ of $\mathcal{O}$-modules, and we deduce from Lemma \ref{parityNExp} that the number of $i$ such that $e_i=e$ is even for all integers $e\geq1$. 
    
This concludes the proof of the proposition. \end{proof}

\section{Higher Fitting ideals of $X$} \label{sec::4}

In what follows, let $K$ be an imaginary quadratic field of discriminant $D_K$ and class number $h_K$. Let $\mathcal{O}$ be a discrete valuation ring with maximal ideal $\mathfrak{m}$ and residue field $\mathcal O/\mathfrak m$ of characteristic $p \nmid D_K h_K$, then set $\Lambda\defeq \mathcal{O} \llbracket \Gamma_\infty \rrbracket$. Fix a free $\Lambda$-module $\mathbf{T}$ of rank $2$ endowed with a continuous $G_K$-action. We assume that $(\kappa,\lambda)$ is a non-trivial bipartite Euler system of parity $\epsilon \in \Z/2 \Z$ for $(\mathbf{T},\mathcal{F},\mathcal{P})$, where $\mathcal{F}$ is a Selmer structure and $\mathcal{P}$ is a set of admissible primes for $\mathbf{T}$. Regarding $\epsilon$ as an element of $\{0,1\}$, set $e\defeq1-\epsilon$. We keep the notation of \S \ref{BESLambda} in force and work under Assumptions \ref{assLambda} and \ref{assBESLambda}. 

\begin{assumption} \label{kolyvagin}
For any height $1$ prime ideal $\mathfrak{P}$ of $\Lambda$, there is an integer $k_\mathfrak{P} \geq 1$ such that for all integers $k \geq  k_\mathfrak{P}$ the set
\[
\bigl\{\lambda_n\in\Lambda/I_n\Lambda\mid n\in\mathcal{N}_{k}^{\deff}\bigr\}
\]
contains an element with non-trivial image in $\Lambda\big/\bigl(\mathfrak{P}+(\pi^{k_\mathfrak{P}})\bigr)$ (respectively, $\Lambda\big/\bigl(\mathfrak{P}+(T^{k_\mathfrak{P}})\bigr)$) if $\mathfrak{P}\neq(\pi)$ (respectively, $\mathfrak{P}=(\pi)$).
\end{assumption}


\subsection{Construction of the ideals $\mathfrak{C}_i$} \label{ss:construction}

Fix an even integer $i\geq0$. Let $k \geq 1$ be an integer and recall that $\mathfrak{m}$ is the maximal ideal of $\mathcal{O}$. Set
\[
\mathfrak{C}_i(k)\defeq\Bigl(\lambda_n^{(k)}\Bigr)_{n \in \mathcal{N}_{2k,}\nu(n) \leq i+e}\subset \Lambda/\mathfrak{m}^k\Lambda,
\]
where $\lambda_n^{(k)}$ is the projection of $\lambda_n$ to $\Lambda/\mathfrak{m}^k \Lambda$.
If $k_2\geq k_1$, then the canonical projection map $\Lambda/\mathfrak{m}^{k_2}\Lambda\twoheadrightarrow \Lambda/\mathfrak{m}^{k_1}\Lambda$ takes $\lambda_n^{(k_2)}$ to $\lambda_n^{(k_1)}$, so we can consider
\[
\mathfrak{C}_i\defeq\varprojlim_{k} \mathfrak{C}_i(k),
\]
which is obviously an ideal of $\varprojlim_{k}\Lambda/\mathfrak{m}^k \Lambda=\Lambda$.

\subsection{Main result}

Given two sequences of real numbers ${(x_j)}_{j\geq1}$ and ${(y_j)}_{j\geq1}$, we write $x_j \prec y_j$ if $\liminf (y_j - x_j) \neq -\infty$; moreover, we write $x_j \sim y_j$ if $x_j \prec y_j$ and $y_j \prec x_j$.

To begin with, we need a preliminary result.

\begin{proposition} \label{fittingspec}
Let $M$ be a finitely generated torsion $\Lambda$-module and let $\mathfrak{P}$ be a height $1$ prime ideal of $\Lambda$. If there are equalities
\[
\Fitt_i(M)\Lambda_\mathfrak{P}=\mathfrak{P}^m \Lambda_\mathfrak{P},\quad\Fitt_{\mathcal{O}_j,i}(M \otimes_\Lambda \mathcal{O}_j)=\pi_j^{m_j}\mathcal{O}_j,
\]
then $m_j \sim m e_\mathfrak{P} j$.
\end{proposition}

\begin{proof} We assume that $M$ is pseudo-isomorphic to the $\Lambda$-module $
N\defeq\bigoplus_{s=1}^n \bigoplus_{t=1}^l\Lambda\big/\mathfrak{p}_s^{k_{s,t}}$, where the sequence of positive integers ${(k_{s,t})}_t$ is non-decreasing for all $s$ and the prime ideals $\mathfrak{p}_s$ are all distinct. By Theorem \ref{Thm-pseudo}, for all $i\in\{0,\dots,l\}$ there are equalities
\[
\Fitt_i(M) \Lambda_{\mathfrak{p}_s}=\Fitt_i(N)\Lambda_{\mathfrak{p}_s}=\Biggl(\bigoplus_{t=1}^{l-i
    } \Lambda/\mathfrak{p}_s^{k_{s,t}}\Biggr) \Lambda_{\mathfrak{p}_s}.
\]
In particular, by Theorem \ref{FittingDVR}, the equality $\mathfrak{P}=\mathfrak{p}_{s_\mathfrak{P}}$ holds for some $s_\mathfrak{P}\in\{1,\dots,n\}$ and the equality $m=\sum_{t=1}^{l-i}k_{s_\mathfrak{P},t}$ holds for all integers $0\leq i<l$. On the other hand, the equality $\Fitt_i(M) \Lambda_\mathfrak{P}=\Lambda_\mathfrak{P}$ holds for $i\geq l$, so $m=0$. In conclusion, we have
\begin{equation} \label{m}
m=\begin{cases}\displaystyle{\sum_{t=1}^{l-i}k_{s_\mathfrak{P},t}} & \text{for $0\leq i<l$},\\[7mm]
    0 & \text{for $i\geq l$.}\end{cases} 
\end{equation}
Since $M$ and $N$ are pseudo-isomorphic, there are exact sequences of $\Lambda$-modules
\[
M \longrightarrow N \longrightarrow B \longrightarrow 0,\quad N \longrightarrow M \longrightarrow B' \longrightarrow 0,
\]
with $B$ and $B'$ finite. Recall that $\mathcal O_j\defeq\mathcal O_{\mathfrak P_j}$. For all $j$ such that $\mathfrak{P}_j \notin \Supp(M)$, there are exact sequences
\begin{equation} \label{seq1}
M \otimes_\Lambda \mathcal{O}_j \longrightarrow N \otimes_\Lambda \mathcal{O}_j \longrightarrow B \otimes_\Lambda \mathcal{O}_j \longrightarrow 0
\end{equation}
and
\begin{equation} \label{seq2}
N \otimes_\Lambda \mathcal{O}_j \longrightarrow M \otimes_\Lambda \mathcal{O}_j \longrightarrow B' \otimes_\Lambda \mathcal{O}_j \longrightarrow 0
\end{equation}
of finitely generated torsion $\mathcal{O}_j$-modules; moreover, by Lemma \ref{pseudoKerCoker}, $B \otimes_\Lambda \mathcal{O}_j$ and $B' \otimes_\Lambda \mathcal{O}_j$ are finite with length bounded independently of $j$. Write
\[
\Fitt_i(N \otimes_\Lambda \mathcal{O}_j)=\bigl(\pi_j^{n_j}\bigr),\quad\Fitt_i(B \otimes_\Lambda\mathcal{O}_j)=\bigl(\pi_j^{b_j}\bigr),\quad\Fitt_i(B'\otimes_\Lambda\mathcal{O}_j)=\Bigl(\pi_j^{b'_j}\Bigr);
\]
by Proposition \ref{pfi} applied to our exact sequences \eqref{seq1} and \eqref{seq2}, there is an integer $b\geq0$ such that $b_j,b'_j\leq b$ and
\[
m_j +b_j \geq n_j,\quad n_j + b'_j \geq m_j.
\]
It follows that 
\begin{equation} \label{m_jn_j}
m_j \sim n_j.
\end{equation}
Furthermore, there are also isomorphisms
\[
N \otimes_\Lambda \mathcal{O}_j\simeq \bigoplus_{s=1}^n \bigoplus_{t=1}^l \Lambda\big/\mathfrak{p}_s^{k_{s,t}}\otimes_\Lambda\mathcal{O}_j\simeq \bigoplus_{s=1}^n \bigoplus_{t=1}^l\mathcal{O}_j\big/s_j(\mathfrak{p}_s^{k_{s,t}})\mathcal{O}_j. 
\]
We distinguish two cases, each of which comprises two subcases.

$\bullet$\quad Suppose $\mathfrak{P} \neq (\pi)$. Then, for each $s=1,\dots,n$ there are two subcases: $\mathfrak{p}_s \neq \mathfrak{P}$ (\emph{i.e.}, $s\neq s_\mathfrak{P}$) and $\mathfrak{p}_s =\mathfrak{P}$ (\emph{i.e.}, $s= s_\mathfrak{P}$). When $\mathfrak{p}_s\neq\mathfrak{P}$, the argument in \texttt{Step 1} of the proof of Proposition \ref{propX=M+M} shows that there is an integer $S\geq 0$ independent of $j$ such that $\bigl(\pi_j^{S e_\mathfrak{P}}\bigr)\subset s_j\bigl(\mathfrak{p}_s^{k_{s,t}}\bigr)$. When $\mathfrak{p}_s =\mathfrak{P}$, there is an equality $\mathfrak{p}_s^{k_{s,t}}+\mathfrak{P}_j=(\pi^{jk_{s,t}})+\mathfrak{P}_j$, so $s_j\bigl(\mathfrak{p}_s^{k_{s,t}}\bigr)\mathcal{O}_j=\bigl(\pi_j^{e_\mathfrak{P} j k_{s,t}}\bigr)$.

$\bullet$\quad Suppose $\mathfrak{P} =(\pi)$. Again, for each $s=1,\dots,n$ there are two subcases: $\mathfrak{p}_s \neq (\pi)$ (\emph{i.e.}, $s\neq s_\mathfrak{P}$) and $\mathfrak{p}_s=(\pi)$ (\emph{i.e.}, $s= s_\mathfrak{P}$). When $\mathfrak{p}_s \neq (\pi)$, the arguments in \texttt{Step 1} of the proof of Proposition \ref{propX=M+M} show that $s_j\bigl(\mathfrak{p}_s^{k_{s,t}}\bigr)\mathcal{O}_j=(\pi_j^S)$ for an integer $S\geq 0$ independent of $j$. When $\mathfrak{p}_s = (\pi)$, the ring $\mathcal{O}_j$ is totally ramified of degree $j$ over $\mathcal{O}$ and, since $e_\mathfrak{P}=1$ in this case, there are equalities $s_j(\mathfrak{p}_s^{k_{s,t}})\mathcal{O}_j=(\pi_j^{jk_{s,t}})=(\pi_j^{e_\mathfrak{P} jk_{s,t}})$.

In any case, we can write
\[
N\otimes_\Lambda\mathcal{O}_j\simeq P\oplus \Biggl(\bigoplus_{t=1}^l \mathcal{O}_j\big/\bigl(\pi_j^{e_\mathfrak{P} jk_{s,t}}\bigr)\!\Biggr)
\]
for an $\mathcal{O}_j$-module $P$ of finite length $L$ independent of $j$. Then, by Theorem \ref{FittingDVR}, we have $n_j=L+\sum_{t=1}^{l-i}e_\mathfrak{P}jk_{s,t}$ if $0 \leq i <l$ and $0 \leq n_j \leq L$ if $i \geq l$. It follows that 
\begin{equation}\label{n_j}
n_j \sim \begin{cases} \displaystyle{\sum_{t=1}^{l-i}e_\mathfrak{P}jk_{s,t}}& \text{for $0 \leq i <l$},\\[7mm]
0 &\text{for $i \geq l$.}
\end{cases}\end{equation}
Now the statement of the proposition follows by combining \eqref{m}, \eqref{m_jn_j} and \eqref{n_j}. \end{proof}

Now fix a height $1$ prime ideal $\mathfrak{P}$ of $\Lambda$ and set
\begin{equation} \label{alphabeta}
\Fitt_i(X) \Lambda_\mathfrak{P}\defeq\mathfrak{P}^{\alpha_i} \Lambda_\mathfrak{P},\quad\mathfrak{C}_i^2 \Lambda_\mathfrak{P}\defeq\mathfrak{P}^{\beta_i} \Lambda_\mathfrak{P}.
\end{equation}
Moreover, for all integers $j\geq N(\mathfrak{P})$ put
\begin{equation} \label{ajbj}
(\pi_j^{a_j})\defeq\Fitt_{\mathcal{O}_j,i}(X \otimes_\Lambda \mathcal{O}_j),\quad
b_j\defeq\length_{\mathcal{O}_j}\bigl((\Lambda / \mathfrak{C}_i^2) \otimes_\Lambda \mathcal{O}_j\bigr).
\end{equation}
By Proposition \ref{fittingspec}, we have
\begin{equation} \label{a_jb_jsim}
a_j \sim \alpha_i e_\mathfrak{P} j,\quad b_j \sim \beta_i e_\mathfrak{P} j.
\end{equation}
On the other hand, by Proposition \ref{extScalars}, for each $j$ there is an exact sequences of $\mathcal{O}_j$-modules 
\begin{equation} \label{KXTC} 
0 \longrightarrow K_j \longrightarrow X \otimes_\Lambda \mathcal{O}_j \longrightarrow X(T_j) \longrightarrow C_j \longrightarrow 0,
\end{equation}
where $K_j$ and $C_j$ are finite, of order bounded independently of $j$. If we set
\[
\Fitt_0(K_j)\defeq\bigl(\pi_j^{k_j}\bigr),\quad\Fitt_i\bigl(X(T_j)\bigr)\defeq\bigl(\pi_j^{t_j}\bigr),\quad\Fitt_0(C_j)\defeq\bigl(\pi_j^{c_j}\bigr),
\]
then there is an integer $c\geq0$ independent of $j$ such that $k_j,c_j \leq c$. Combining the short exact sequence 
\[
X\otimes_\Lambda\mathcal{O}_j\longrightarrow X(T_j)\longrightarrow C_j\longrightarrow0
\] 
and Proposition \ref{pfi}, we get $a_j + c_j \geq t_j$, hence $t_j\prec a_j$. Taking Pontryagin duals of exact sequence \eqref{KXTC}, we obtain the exact sequence 
\[
0\longrightarrow C_j^\vee\longrightarrow X(T_j)^\vee\longrightarrow (X\otimes_\Lambda\mathcal{O}_j)^\vee \longrightarrow K_j^\vee\longrightarrow0,
\]
where for an $\mathcal{O}_j$ module $M$ we let $M^\vee\defeq\Hom_{\mathcal{O}_j}(M,F_j/\mathcal{O}_j)$ be its Pontryagin dual, with $F_j$ the fraction field of $\mathcal{O}_j$. Thus, we get another exact sequence 
\[
X(T_j)^\vee\longrightarrow (X\otimes_\Lambda\mathcal{O}_j)^\vee\longrightarrow K_j^\vee\longrightarrow0.
\]
The $\mathcal{O}_j$-modules $K_j$, $X\otimes_\Lambda\mathcal{O}_j$, $X(T_j)$ are
finite, so they are isomorphic to their Pontryagin duals as $\mathcal{O}_j$-modules; then $\Fitt_0(K_j)=\Fitt_0(K_j^\vee)$, $\Fitt_i(X\otimes_\Lambda\mathcal{O}_j)=\Fitt_i\bigl((X\otimes_\Lambda\mathcal{O}_j)^\vee\bigr)$, $\Fitt_i\bigl(X(T_j)\bigr)=\Fitt_i\bigl(X(T_j)^\vee\bigr)$. Therefore, by Proposition \ref{pfi}, we have $t_j + k_j \geq a_j$, hence $a_j\prec t_j$ and we conclude that 
\begin{equation} \label{maineq2}
a_j \sim t_j = 2\delta^{(i+e)}(\lambda_j)-2\delta(\lambda_j),
\end{equation} 
where the equality follows from Theorem \ref{thmDVR}.

For any integers $j\geq N(\mathfrak{P})$ and $k\gg0$, denote by $I(k) \subset \Lambda$ the inverse image of $\mathfrak{C}_i(k)$ via the natural projection. Observe that the sequence $\bigl(I(k)\bigr)_k$ is non-increasing under inclusion and $\mathfrak{C}_i=\bigcap_{k \gg 0}I(k)$.

\begin{lemma} \label{kj0}
For every $j\geq N(\mathfrak P)$, there is an integer $k_{j,0} >0$ such that $s_j\bigl(I(k_{j,0})\bigr)\mathcal{O}_j = s_j(\mathfrak{C}_i) \mathcal{O}_j$.
\end{lemma}

\begin{proof} Since $s_j(\mathfrak{C}_i) \mathcal{O}_j$ is not trivial, there are integers $k_{j,0},N \geq 0$ such that $s_j\bigl(I(k)\bigr)\mathcal{O}_j=(\pi_j^N)$ for all $k\geq k_{j,0}$. Set
\[
J_N\defeq s_j^{-1}\bigl((\pi_j^N)\bigr),\quad J_{N+1}\defeq s_j^{-1}\bigl((\pi_j^{N+1})\bigr)
\]
and notice that for all $k\geq k_{j,0}$ we have:
\begin{itemize}
\item $I(k)\subset J_N$;
\item $I(k)\cap(\Lambda\smallsetminus J_{N+1})\neq\emptyset$.
\end{itemize}
It follows, in particular, that the descending sequence of sets $\bigl(I(k)\cap(\Lambda\smallsetminus J_{N+1})\bigr)_{k\geq k_{j,0}}$ has the finite intersection property. Considering the $\mathfrak{m}_\Lambda$-adic topology on $\Lambda$, we also observe that
\begin{itemize}
\item $J_{N+1}$ is open, as it is the inverse image of an open set via a continuous function;
\item $I(k)$ is closed, as it is an ideal in a complete noetherian local ring.
\end{itemize}
Then the sets in our descending sequence are closed and, since $\Lambda$ is compact, we conclude that their intersection is non-empty. Finally, there is $\lambda \in \mathfrak{C}_i$ such that $s_j(\lambda)\mathcal{O}_j=(\pi_j^N)$; in particular, $s_j\bigl(I(k_{j,0})\bigr)\mathcal{O}_j=s_j(\mathfrak{C}_i)\mathcal{O}_j$, which completes the proof. \end{proof}

\begin{theorem} \label{highFittLambda}
If Assumptions \ref{assLambda} and \ref{assBESLambda} are satisfied, then $\mathfrak{C}_i^2\prec\Fitt_i(X)$ for all even integers $i\geq0$. If, in addition, Assumption \ref{kolyvagin} is satisfied, then $\Fitt_i(X) \sim \mathfrak{C}_i^2$ for all even integers $i\geq0$.
\end{theorem}

\begin{proof} Suppose that Assumptions \ref{assLambda} and \ref{assBESLambda} are satisfied. Recall that, in our notation, $\boldsymbol{\lambda}_j^{(k)}\defeq\lambda_j^{(k),2k}$. Moreover, recall the integers $\alpha_i$ and $\beta_i$ from \eqref{alphabeta} and the integers $a_j$ and $b_j$ from \eqref{ajbj}. For all $j \geq N(\mathfrak{P})$, fix an integer $k_j\gg0$ such that $\delta^{(i+e)}(\lambda_j)=\partial^{(i+e)}\bigl(\boldsymbol{\lambda}_j^{(e_jk_j)}\bigr)$ and $k_j \geq k_{j,0}$, where $k_{j,0}$ is as in Lemma \ref{kj0}. Then
\begin{equation} \label{maineq1}
\begin{split}
\beta_i e_\mathfrak{P} j \sim b_j &=\length_{\mathcal{O}_j}\bigl((\Lambda/\mathfrak{C}_i^2) \otimes_\Lambda \mathcal{O}_j \bigr)=2\cdot\length_{\mathcal{O}_j}\bigl((\Lambda/\mathfrak{C}_i)\otimes_\Lambda \mathcal{O}_j\bigr)\\
&=2\cdot\length_{\mathcal{O}_j}\bigl(\mathcal{O}_j/s_j(\mathfrak{C}_i)\mathcal{O}_j\bigr)= 2\cdot\length_{\mathcal{O}_j}\bigl(\mathcal{O}_j/s_j(I(k_j))\mathcal{O}_j\bigr), 
\end{split}
\end{equation} 
where the equivalence $\sim$ follows from \eqref{a_jb_jsim} and the last equality is a consequence of Lemma \ref{kj0} and of our choice of $k_j$. Since $\pi^{k_j} \in I(k_j)$, there is an isomorphism
\begin{equation} \label{maineq2}
\mathcal{O}_j\big/s_j\bigl(I(k_j)\bigr)\mathcal{O}_j\simeq\bigl(\mathcal{O}_j\big/\pi_j^{e_j k_j}\mathcal{O}_j\bigr)\big/\pi_{k_j}\bigl(s_j(I(k_j))\bigr),
\end{equation}
where $\pi_{k_j}:\mathcal{O}_j\twoheadrightarrow\mathcal{O}_j\big/\pi_j^{e_j k_j} \mathcal{O}_j$ is the natural projection. By definition of $\mathfrak{C}_i(k_j)$, there exists $n \in \mathcal{N}_{2k_j}$ with $\nu(n)\leq i+e$ such that $\pi_{k_j}\Bigl(s_j\bigl(I(k_j)\bigr)\!\Bigr)=\bigl(\lambda_{j,n}^{(k_j)}\bigr)$. The sequence $\bigl(\partial^{(t)}(\boldsymbol{\lambda}_j^{(e_jk_j)})\bigr)_{t\geq1}$ is non-increasing, so we can assume $\nu(n)=i+e$; we obtain the equality
\begin{equation} \label{eqIndPart}
\ind\bigl(\lambda_{j,n}^{(k_j)}\bigr)=\partial^{(i+e)}\bigl(\boldsymbol{\lambda}_j^{(e_jk_j)}\bigr).
\end{equation}
Combining \eqref{eqIndPart} with \eqref{maineq1} and \eqref{maineq2}, we get
the first equality in the chain
\begin{equation} \label{maineq3}
\begin{aligned}
\beta_i e_\mathfrak{P} j \sim b_j &= 2\cdot\length_{\mathcal{O}_j}\Bigl(\!\bigl(\mathcal{O}_j / \pi_j^{e_j k_j}\mathcal{O}_j\bigr)\big/\bigl(\lambda_{j,n}^{(k_j)}\bigr)\!\Bigr)= 2\cdot\partial^{(i+e)}\bigl(\boldsymbol{\lambda}_j^{(e_jk_j)}\bigr)\\
&=2\cdot\delta^{(i+e)}(\lambda_j)\geq 2\cdot\delta^{(i+e)}(\lambda_j)-2\cdot\delta(\lambda_j) \sim a_j \sim \alpha_i e_\mathfrak{P} j,
\end{aligned}\end{equation}
while the second equality follows from \eqref{eqIndPart}, the third follows from our choice of $k_j$, the first $\sim$ is \eqref{maineq2} 
and the second $\sim$ is 
\eqref{a_jb_jsim}. Thus, we conclude that $\alpha_i\prec\beta_i$, completing the proof of the first statement. 

To prove the second claim of the theorem, we need to check that, under Assumption \ref{kolyvagin}, the sequence $\bigl(\delta (\lambda_j)\bigr)_{j\geq1}$ is bounded. Again, we distinguish the cases $\mathfrak{P}\neq(\pi)$ and $\mathfrak{P}=(\pi)$. 

$\bullet$\quad Suppose $\mathfrak{P} \neq (\pi)$. Using Assumption \ref{kolyvagin}, we choose $\lambda_{n(k)}\in\Lambda$, with $n(k) \in \mathcal{N}_k^{\deff}$, having non-trivial image in $\Lambda\big/\bigl(\mathfrak{P}+(\pi^{k_\mathfrak{P}})\bigr)$ for all $k \geq k_\mathfrak{P}$. Then $\lambda_{n(k)}$ has non-trivial image in $\Lambda\big/\bigl(\mathfrak{P}_j+(\pi^{k_\mathfrak{P}})\bigr)$ for all $j\geq k_\mathfrak{P}$, where the integer $k_\mathfrak{P}\geq0$ is as in \emph{loc. cit.} Let $C_j$ be the cokernel of $\Lambda/\mathfrak{P}_j \hookrightarrow \mathcal{O}_j$ and notice that it is finite and independent of $j$. Then fix an integer $k_1\geq k_\mathfrak{P}$ such that $\pi^{k_1-k_\mathfrak{P}}$ kills $C_j$; by the snake lemma, there is a commutative diagram with exact rows
\[
\xymatrix@C=30pt{
C_j[\pi^{k_1}] \ar[r] \ar[d] &\Lambda\big/\bigl(\mathfrak{P}_j+(\pi^{k_1})\bigr) \ar[r] \ar[d] & \mathcal{O}_j\big/\bigl(\pi_j^{e_jk_1}\bigr) \ar[d]\\
C_j[\pi^{k_\mathfrak{P}}] \ar[r] & \Lambda\big/\bigl(\mathfrak{P}_j+(\pi^{k_\mathfrak{P}})\bigr) \ar[r] & \mathcal{O}_j\big/\bigl(\pi_j^{e_jk_\mathfrak{P}}\bigr).
}
\]
Since $\pi^{k_1-k_\mathfrak{P}}C_j=0$, the left vertical arrow is trivial, so 
$\lambda_{n(k)}$ has non-trivial image in $\mathcal{O}_j/(\pi_j^{e_jk_1})$ for all $k \geq k_1$. It follows (\emph{cf.} the definition of $\partial(\boldsymbol{\lambda}_j^{(k)})$ in \eqref{partiallambda}) that
\[
\partial\bigl(\boldsymbol{\lambda}_j^{(k)}\bigr)\leq\ind\Bigl(\lambda_{j,n(2k)},\mathcal{O}_j\big/\bigl(\pi_j^{2e_jk}\bigr)\!\Bigr)\leq e_jk_1=e_\mathfrak{P}k_1
\]
for $k\gg0$.

$\bullet$\quad Suppose $\mathfrak{P}=(\pi)$. Then, again using Assumption \ref{kolyvagin}, we choose $\lambda_{n(k)}\in\Lambda$, with $n(k) \in \mathcal{N}_k^{\deff}$, having non-trivial image in $\Lambda\big/\bigl(\mathfrak{P}+(T^{k_\mathfrak{P}})\bigr)$. It follows that $\lambda_{n(k)}$ has non-trivial image in $\Lambda\big/\bigl(\mathfrak{P}_j+(T^{k_\mathfrak{P}})\bigr)=\mathcal{O}_j\big/\pi_j^{k_\mathfrak{P}} \mathcal{O}_j$ for all $j \geq k_\mathfrak{P}$, so there are inequalities
\[
\partial\bigl(\boldsymbol{\lambda}_j^{(k)}\bigr)\leq \ind\Bigl(\lambda_{j,n(2k)},\mathcal{O}_j\big/\bigl(\pi_j^{2k}\bigr)\!\Bigr)\leq k_\mathfrak{P}
\]
for $k\gg0$ (again, \emph{cf.} \eqref{partiallambda}).

Comparing with the definition of $\delta(\lambda_j)$ in \eqref{deltas}, we conclude that $\bigl(\delta (\lambda_j)\bigr)_{j\geq1}$ is bounded, which concludes the proof. \end{proof}

In the indefinite case, we can give this result in terms of classes in $\kappa$. Namely, for all even integers $i\geq0$ set
\[
\begin{split}
\mathfrak{D}_i(k)&\defeq\Bigl(\Bigl\{f\bigl(\kappa_n^{(k)}\bigr)\;\Big|\;\text{$f\in \Hom_{\Lambda/(\pi^k)}\bigl(H^1(K,\mathbf{T}/\pi^k \mathbf{T}),\Lambda/(\pi^k)\bigr)$, $n \in\mathcal{N}_{2k}$, $\nu(n)\leq i$}\Bigr\}\Bigr)\\&\;\subset\Lambda\big/(\pi^k),
\end{split}
\]
where $\kappa_n^{(k)}$ is the image of $\kappa_n$ in $H^1(K,\mathbf{T}/\pi^k \mathbf{T})$. Moreover, put $\mathfrak{D}_i\defeq\varprojlim_k \mathfrak{D}_i(k)$, which is an ideal of $\Lambda$.

\begin{theorem} \label{lambdaKappa}
Suppose that Assumptions \ref{assLambda}, \ref{assBESLambda} and \ref{kolyvagin} are satisfied and $e=1$. Then $\Fitt_i(X) \sim \mathfrak{D}_i^2$ for all even integers $i\geq0$.
\end{theorem}

\begin{proof} Proceed as in the proof of Theorem \ref{highFittLambda}, replacing Theorem \ref{dvrKappa} with Theorem \ref{thmDVR}. \end{proof}

\section{Elliptic curves} \label{sec::5}

In this section, we specialize our previous results to (the Pontryagin duals of) Selmer and Shafarevich--Tate groups of elliptic curves.

Let $E$ be an elliptic curve over $\Q$ of conductor $N$. Let $K$ be an imaginary quadratic field of discriminant $D_K$ coprime with $N$ and class number $h_K$. Let us write $N=N^+N^-$ where $N^+$ is divisible only by primes that split in $K$ and $N^-$ is divisible only by primes that are inert in $K$. We assume that $N^-$ is square-free and set
\begin{equation} \label{e-elliptic-eq}
e\defeq\begin{cases}0&\text{if $N^-$ is divisible by an \emph{odd} number of primes},\\[2mm]1&\text{if $N^-$ is divisible by an \emph{even} number of primes}. \end{cases}
\end{equation}
The first (respectively, second) case in \eqref{e-elliptic-eq} is called \emph{definite} (respectively, \emph{indefinite}). From here on, fix a prime number $p$ such that
\begin{itemize}
\item $p\nmid6ND_Kh_K$ (in particular, $E$ has good reduction at $p$);
\item $E$ has \emph{ordinary} reduction at $p$ (\emph{i.e.}, $\#\bar E_p(\F_p)\not\equiv1\pmod{p}$, where $\bar E_p$ denotes the reduction of $E$ at $p$).
\end{itemize}
We write $T=T_p(E)$ for the $p$-adic Tate module of $E$ and let $A\defeq E[p^\infty]$ be its $p$-primary torsion subgroup. Then $A\simeq \Hom(T,\Bmu_{p^\infty})$ and for every integer $k\geq 0$ there is a canonical isomorphism of $G_\Q$-modules $T_k\simeq A_k$, 
where $T_k\defeq T/p^kT$ and $A_k\defeq A[p^k]=E[p^k]$ is the $p^k$-torsion subgroup of $E$. Now set 
\[ 
\mathbf{T}\defeq\varprojlim_m\mathrm{Ind}_{K_m/K}(T),
\] 
where $\mathrm{Ind}_{K_m/K}(T)$ is the $G_K$-module induced from the $G_{K_m}$-module $T$ and the inverse limit is taken with respect to the corestriction maps. We also put 
\[
\mathbf{A}\defeq\varinjlim_{m}\mathrm{Ind}_{K_m/K}(A),
\]
the direct limit being computed with respect to the restriction maps. 

Let $\bar{\rho}_{E,p}:G_\Q \rightarrow \mathrm{Aut}_{\F_p}(E[p])$ be the residual mod $p$ representation attached to $E$ and, as customary, set $a_p(E)\defeq1+p-\#\bar E_p(\F_p)\in\Z$. 

\begin{assumption} \label{asselliptic}
\begin{enumerate}
\item $\bar{\rho}_{E,p}$ is surjective;
\item $\bar{\rho}_{E,p}$ is ramified at $q$ for any $q|N^+$;
\item if $q\,|\,N^-$ and $q\equiv \pm 1 \pmod p$, then $\bar{\rho}_{E,p}$ is ramified at $q$;
\item $a_p(E) \not\equiv 1 \pmod p $ if $p$ splits in $K$ and $a_p(E) \not\equiv \pm 1 \pmod p$ if $p$ is inert in $K$.
\end{enumerate}
\end{assumption}
Observe that, by the ``open image theorem'' of Serre (\cite{serre}), condition (1) in Assumption \ref{asselliptic} rules out only finitely many prime numbers $p$. With notation as above, from now until the end of this paper we take $\mathcal O=\Z_p$, so that $\Lambda=\Z_p[\![T]\!]$.

\subsection{Selmer groups}

For a prime number $q\,|\,N^-$, let $\mathfrak{q}$ be the only prime of $K$ lying over $q$ and write $K_\mathfrak{q}$ for the completion of $K$ at $\mathfrak{q}$; the group $G_{K_\mathfrak{q}}\defeq\Gal(\overline{K}_\mathfrak{q}/K_\mathfrak{q})$ acts on $T$ as
$\bigl(\begin{smallmatrix}\chi_p&*\\0&1\end{smallmatrix}\bigr)$, where $\chi_p$ is the $p$-adic cyclotomic character. Let $F_q^+T\subset T$ be the $\Z_p$-line on which $G_{K_\mathfrak{q}}$ acts as $\chi_p$ and, given a finite extension $L/K_\mathfrak{q}$, set
\[
H^1_{\ord}(L,T)\defeq\im\Bigl(H^1(L,F_q^+T) \longrightarrow H^1(L,T)\!\Bigr).
\]
Since $E$ has good ordinary reduction at $p$, there is a $\Z_p$-line $F_p^+T\subset T$ on which the inertia subgroup $I_p$ of $\Gal(\overline{\Q}_p/\Q_p)$ acts as $\chi_p$. Again, for any finite extension $L$ of $\Q_p$ we define
\[
H^1_{\ord}(L,T)\defeq\im (H^1(L,F_p^+T) \longrightarrow H^1(L,T)).
\]
The \emph{Greenberg Selmer structure} on $\mathbf{T}$ is
\[
H^1_{\mathcal{F}_{\Gr}}(K_v,\mathbf{T})\defeq\begin{cases}
    \varprojlim_m H^1_{\unr}(K_{m,v},T) & \text{if $v\nmid N^-p$},\\[3mm]
    \varprojlim_m H^1_{\ord}(K_{m,v},T) & \text{if $v\,|\,N^-p$}, 
\end{cases}
\]
where $H^1(K_{m,v},T)\defeq\bigoplus_{w|v} H^1(K_{m,w},T)$ the isomorphism $H^1(K_v,\mathbf{T}) \simeq \varprojlim_m H^1(K_{m,v},T)$ is a consequence of Shapiro's lemma. Denote by $\Sel(K,\mathbf{T})$ the Selmer group defined by the local conditions $H^1_{\mathcal{F}_{\Gr}}(K_v,\mathbf{T})$. Similarly, we set 
\[
H^1_{\mathcal{F}_{\Gr}}(K_v,\mathbf{A})\defeq\begin{cases}
    \varinjlim_m H^1_{\unr}(K_{m,v},A) & \text{if $v \nmid N^-p$}, \\[3mm]
    \varinjlim_m H^1_{\ord}(K_{m,v},A) & \text{if $v\,|\,N^-p$}, 
\end{cases}
\]
and denote by $\Sel(K,\mathbf{A})$ the Selmer group defined by the local conditions 
$H^1_{\mathcal{F}_{\Gr}}(K_v,\mathbf{A})$. By \cite[Proposition 2.2.4]{HoHeeg}, there is an isomorphism $\mathbf{A} \simeq \Hom_{\mathrm{cont}}(\mathbf{T}^\iota,\Bmu_{p^\infty})$; moreover, for all primes $v$ of $K$ the group $H^1_{\mathcal{F}_{\Gr}}(K_v,\mathbf{A})$ is the orthogonal complement of $H^1_{\mathcal{F}_{\Gr}}(K_{\bar{v}},\mathbf{T})$ under the local Tate pairing.

Let $\mathfrak{P}$ be a height $1$ prime ideal of $\Lambda$. Let $F_\mathfrak{P}$ be the fraction field of $\mathcal{O}_\mathfrak{P}$ and let $\pi_\mathfrak{P}$ be a uniformizer of $\mathcal{O}_\mathfrak{P}$. Define $T_{\mathfrak P}$, $V_{\mathfrak P}$, $A_{\mathfrak P}$ as in \eqref{T-V-A-eq}. For each prime number $q\,|\,N^-p$ and each prime $\mathfrak{q}$ of $K$ lying over $q$, set
\[
H^1_{\ord}(K_\mathfrak{q},V_\mathfrak{P})\defeq\im\Bigl(H^1(K_\mathfrak{q},F^+_qT \otimes_{\mathcal{O}_\mathfrak{P}}F_\mathfrak{P})\longrightarrow H^1(K_\mathfrak{q},V_\mathfrak{P})\!\Bigr).
\]
The \emph{Greenberg Selmer structure} on $V_\mathfrak{P}$ is
\[
H^1_{\mathcal{F}_{\Gr,\mathfrak{P}}}(K_v,V_\mathfrak{P})\defeq\begin{cases}
    H^1_{\unr}(K_v,V_\mathfrak{P}) & \text{if $v \nmid N^-p$},\\[3mm]
    H^1_{\ord}(K_v,V_\mathfrak{P}) & \text{if $v\,|\,N^-p$},
\end{cases}
\]
and we also define $H^1_{\mathcal{F}_{\Gr,\mathfrak{P}}}(K_v,T_\mathfrak{P})$ and $H^1_{\mathcal{F}_{\Gr,\mathfrak{P}}}(K_v,A_\mathfrak{P})$ by propagation. Thus, we obtain Selmer groups $\Sel_{\Gr}(K,V_\mathfrak{P})$, $\Sel_{\Gr}(K,T_\mathfrak{P})$ and $\Sel_{\Gr}(K,A_\mathfrak{P})$. 

\subsection{Bipartite Euler systems for elliptic curves}

Let $\mathcal{P}$ denote the set of admissible primes for $\mathbf{T}$ with respect to $\mathcal{F}_\mathrm{Gr}$, as in Definition \ref{admissibleprimes}, and let $\mathcal{N}$ be the set of square-free products of admissible primes. For any $n \in \mathcal{N}$ we denote by the same symbol $I_n$ the ideals defined in \S\ref{bipartite} for $T$ and $\mathbf{T}$; in our case, since $\ell\mathcal{O}_K$ splits completely in $K_\infty$, $I_n=(p^k)$ where $p^k$ divides $a_\ell\pm(\ell+1)$ exactly. Under Assumption \ref{asselliptic}, the work of Bertolini--Darmon on the Iwasawa main conjectures over anticyclotomic $\Z_p$-extensions (\cite{BD-IMC}) produces a bipartite Euler system $(\kappa,\lambda)$ using the theory of congruences between modular forms over the rings $\Z/p^k\Z$. The reader is referred to \cite[Section 7]{BD-IMC} and \cite{BLV} for details; to facilitate the comparison with \cite{BLV}, notice that the construction of the elements $\lambda_n$ (denoted by $\mathcal{L}_{g,n}$ in \cite{BLV}) is treated in \cite[\S 4.2]{BLV}, while the construction of the elements $\kappa_n$ (denoted by $\kappa_n(g)$ in \cite{BLV}) is discussed in \cite[\S 6.1.1]{BLV} (in \cite{BLV}, for a given $n\in\mathcal{N}$ with $I_n=(p^k)$, the symbol $g$ indicates a modular form on a suitable quaternion algebra that is congruent modulo $p^k$ to the newform attached to $E$: see \cite[\S3]{BLV} for details). For a discussion of the explicit reciprocity laws satisfied 
by this bipartite Euler system, see \cite[Theorems 4.1 and 4.2]{BD-IMC} and \cite[\S6.2]{BLV}. Observe that, in particular, Assumption \ref{assBESLambda} is satisfied: if $e=0$, then $\lambda_1\neq0$ by work of Vatsal (\cite{vatsal2002uniform}), while if $e=1$, then $\kappa_1\neq0$ thanks to results of Cornut (\cite{cornut2002mazur}) and Vatsal (\cite{vatsal2003special}).

\subsection{Structure of Selmer groups of anticyclotomic twists}

Fix a height $1$ prime ideal $\mathfrak{P}$ of $\Lambda$. Equip $T_\mathfrak{P}$ with the pairing
\[
{(\cdot,\cdot)}_\mathfrak{P}:T_\mathfrak{P} \times T_\mathfrak{P} \longrightarrow \mathcal{O}_\mathfrak{P}(1)
\]
defined as
\[
{(x \otimes \alpha,y \otimes \beta)}_\mathfrak{P}\defeq\alpha \beta{(x,y^\tau)}_{\Weil},
\]
where $x,y \in T$, $\alpha,\beta\in\mathcal{O}_\mathfrak{P}$ and ${(\cdot, \cdot)}_{\Weil}$ is the Weil pairing on $T$.

\begin{proposition} \label{assTwists}
The triple $(T_\mathfrak{P},\mathcal{F}_\mathfrak{P},\mathcal{P})$ satisfies Assumption \ref{ass}. 
\end{proposition}

\begin{proof} Conditions (1)--(4) in Assumption \ref{ass} follow from \cite[Lemma 3.3.4]{howard2012bipartite}, while (5) is a consequence of condition (1) in Assumption \ref{asselliptic}, which implies that $E[p]$ is an irreducible $G_K$-module. \end{proof}

By Proposition \ref{non-triv}, there is a finite set $\Sigma_{\mathrm{EX}}$ of height $1$ prime ideals of $\Lambda$ such that for all $\mathfrak{P} \notin \Sigma_{\mathrm{EX}}$ the pair $(\kappa,\lambda)$ induces a bipartite Euler system $(\kappa_\mathfrak{P},\lambda_\mathfrak{P})$ with $\lambda_{\mathfrak{P},1}\neq 0$ if $e=0$ and $\kappa_{\mathfrak{P},1}\neq0$ if $e=1$.

\begin{theorem} \label{mainthmDVR}
Assume $\mathfrak{P} \notin \Sigma_{\mathrm{EX}}$. There is an isomorphism
\[
\Sel_{\Gr}(K,A_\mathfrak{P}) \simeq (F_\mathfrak{P}/\mathcal{O}_\mathfrak{P})^e \oplus \bigoplus_{i\geq0}\bigl(\mathcal{O}_\mathfrak{P}\big/\pi_\mathfrak{P}^{d_i} \mathcal{O}_\mathfrak{P}\bigr)^2,
\]
where $d_i\defeq\delta^{(2i+e)}(\lambda_\mathfrak{P})-\delta^{(2i+2+e)}(\lambda_\mathfrak{P})$. Moreover, if $e=1$, then for every $i\geq0$ the integer $d_i$ is equal to $\delta^{(2i)}(\kappa_\mathfrak{P})-\delta^{(2i+2)}(\kappa_\mathfrak{P})$. Finally, there is an isomorphism 
\[
X(T_\mathfrak{P})\simeq\bigoplus_{i\geq0}\bigl(\mathcal{O}_\mathfrak{P}\big/\pi_\mathfrak{P}^{d_i} \mathcal{O}_\mathfrak{P}\bigr)^2.
\]
\end{theorem}

\begin{proof} Using Proposition \ref{assTwists}, this follows from Theorems \ref{thmDVR} and \ref{dvrKappa}. \end{proof}

We immediately obtain

\begin{corollary} \label{mainthmK}
Assume $(T) \notin \Sigma_{\mathrm{EX}}$. There is an isomorphism
\[
\Sel_{\Gr}\bigl(K,E[p^\infty]\bigr) \simeq (\Q_p/\Z_p)^e \oplus \bigoplus_{i \geq 0} (\Z_p/p^{d_i} \Z_p)^2,
\]
where $d_i\defeq\delta^{(2i+e)}(\lambda_{(T)})-\delta^{(2i+2+e)}(\lambda_{(T)})$.
Moreover, if $e=1$, then for every $i\geq0$ the integer $d_i$ is equal to $\delta^{(2i)}(\kappa_{(T)})-\delta^{(2i+2)}(\kappa_{(T)})$. Finally, there is an isomorphism
\[
\Sha_{\Gr}\bigl(K,E[p^\infty]\bigr)\simeq\bigoplus_{i \geq 0} (\Z_p/p^{d_i}\Z_p)^2.
\]
\end{corollary}

\begin{proof} Apply Theorem \ref{mainthmDVR} with $\mathfrak{P}=(T)$. \end{proof}


\begin{remark}
When $e=0$, Kim proved an analogous result (\cite[Theorem 4.24]{kim2024higher}) without the assumption $\lambda_{(T),1}\neq0$ under the hypothesis that all projections of the bipartite Euler system are free, which is not the case in general; Corollary \ref{mainthmK} holds without this freeness condition and therefore extends \cite[Theorem 4.24]{kim2024higher} under the assumption $\lambda_{(T),1}\neq0$. Generalizing the arguments developed in this paper, it might be possible to remove the freeness assumption in Kim's results.
\end{remark}

\begin{remark} 
When $e=1$, Corollary \ref{mainthmK} can be seen as an analogue of a result of 
Kolyvagin (\cite{KolStructure}; \emph{cf.} also \cite{McCallum}), which is stronger than ours in that it does not assume $\kappa_{(T),1}\neq0$, but only requires the non-triviality of the collection of Kolyvagin classes (\emph{Kolyvagin's conjecture}); under suitable technical conditions, Kolyvagin's conjecture was proved by W. Zhang (\cite{Zhang}). The formulations in terms of the elements $\delta^{(2j)}(\lambda_{(T)})$ and of the elements $\delta^{(2j)}(\kappa_{(T)})$ are both new, but while the second is somehow reminiscent of that of Kolyvagin (because both use Heegner points, albeit arising from different Shimura curves), the first is genuinely novel. 
\end{remark}

\subsection{Structure of anticyclotomic Shafarevich--Tate groups}

Under Assumption \ref{asselliptic}, we describe the higher Fitting ideals of the Pontryagin duals of Selmer and Shafarevich--Tate groups of $E$ over the anticyclotomic $\Z_p$-extension of $K$. 

We begin by showing that Assumptions \ref{assLambda} and \ref{kolyvagin} are satisfied in this case.

\begin{proposition}
Assumption \ref{assLambda} is satisfied.
\end{proposition}

\begin{proof} Proposition \ref{assTwists} shows that condition (1) in Assumption \ref{assLambda} is satisfied. On the other hand, conditions (2) and (3), \emph{i.e.}, the existence of the induced maps $s_\mathfrak{P}$, $i_\mathfrak{P}$ and the control theorem, follow from \cite[Proposition 3.3.1]{howard2012bipartite}. \end{proof}

\begin{proposition}
The system $(\kappa,\lambda)$ satisfies Assumption \ref{kolyvagin}.
\end{proposition}

\begin{proof} This is proved, with a different terminology, in \cite[\S7.2.4]{BLV} as a combination of Gross's formula for special values of $L$-functions, results by Skinner--Urban and level raising/rank lowering arguments (notice that such a result is claimed in \cite[Lemma 3.6]{BCK} too; however, as far as we understand, the proof in \cite{BCK} seems to be incomplete). For the reader's convenience, we replicate those arguments in our notation. 

Consider the newform $f\in S_2(\Gamma_0(N))$ attached to $E$ by modularity. For any $n \in \mathcal{N}_1$ we denote by $f_n \in S_2(\Gamma_0(Nn))$ the $n$-th level raising of $f$; the eigenform $f_n$ is obtained by combining \cite[Theorem 3.3]{BLV}, which allows us to first obtain an eigenform with coefficients in $\F_p=\Z/p\Z$, with \cite[Lemma 6.1]{DS}, which lifts our eigenform with coefficients in $\F_p$ to an eigenform in characteristic $0$ with coefficients in a finite extension $\mathcal{O}$ of $\Z_p$. In particular, if $\mathfrak{m}_\mathcal{O}$ is the maximal ideal of $\mathcal O$, then there is a congruence $f\equiv f_n\pmod{\mathfrak{m}_\mathcal{O}}$. Let $\F\defeq\mathcal{O}/\mathfrak m_{\mathcal O}$ be the residue field of $\mathcal O$, which is a finite extension of $\F_p$, and fix a uniformizer $\pi_{\mathcal O}$ of $\mathcal{O}$.
 
Let $A^0_{f_n}$ be the abelian variety attached to $f_n$ by the Eichler--Shimura construction and denote by $A_{f_n}$ be the abelian variety in the isogeny class of $A^0_{f_n}$ having real multiplication by $\mathcal{O}$. The Brauer--Nesbitt theorem, the Eichler--Shimura relation and the irreducibility of $E[p]$ show that there is an isomorphism of $\F[G_K]$-modules $E[p]\otimes_{\F_p}\F\simeq A_{f_n}[\pi_{\mathcal O}]$. 
Then there are equalities
\[
\dim_{\F_p}\Sel_{\mathcal{F}(n)}(K,T/pT)=
\dim_{\F}\Sel_{\mathcal{F}(n)}\bigl(K,E[p]\otimes_{\F_p}\F\bigr)=\dim_{\F}\Sel_\mathcal{F}\bigl(K,A_{f_n}[\pi_{\mathcal O}]\bigr),
\]
where $\mathcal{F}\defeq\mathcal{F}_{\Gr}$ (and we use the isomorphism 
$E[p]\otimes_{\F_p}\F\simeq A_{f_n}[\pi_{\mathcal O}]$ to define the Selmer structure on $A_{f_n}$).
Set 
\[
r_n \defeq \dim_{\F}\Sel_\mathcal{F}\bigl(K,A_{f_n}[\pi_{\mathcal O}]\bigr).
\]
Observe that, by Theorem \ref{artSel}, the integers $r_1$ and $e$ have the same parity. We will use the following fact that is a consequence of Lemma \ref{lemmaind}, Lemma \ref{lemmadef} and \cite[Lemma 2.3.3]{howard2012bipartite}:
\begin{equation} \label{lr}
\text{for all integers $k\geq1$ and all $n\in\mathcal{N}_k$, there is $\ell\in \mathcal{P}_k$ such that $r_{n\ell}=r_n-1$.}
\end{equation}
Fix an integer $k\geq1$: we claim that there exists $n \in \mathcal{N}_k$ such that $\lambda_n$ is non-trivial in $\Lambda/\mathfrak{m}_\Lambda$. If $r_1>0$, then iterating \eqref{lr} shows that there is $n \in \mathcal{N}_k$ with $\nu(n) \equiv r_1 \equiv e \pmod 2$ such that $r_n=0$. Denote by
\[
L^{\mathrm{alg}}(f_n/K,s)\defeq\frac{L(f_n/K,s)}{\Omega_{f_n}^{\mathrm{can}}}
\]
the algebraic $p$-adic $L$-function of $f_n$ over $K$, where $\Omega_{f_n}^{\mathrm{can}}$ is the canonical period of $f_n$ (\cite[(6.3)]{Zhang}, \cite[Section 4]{BBV}). 
By a result of Skinner--Urban (\cite[Theorem 2]{SU}), extended to modular abelian varieties (see, \emph{e.g.}, \cite[Theorem 7.1]{Zhang} or \cite[(7.20)]{BLV}), there is an equality 
\begin{equation} \label{L-alg-eq}
\ord_{\pi_{\mathcal O}}\bigl(L^{\mathrm{alg}}(f_n/K,1)\bigr)=\length_{\mathcal{O}}\Bigl(\Sel_\mathcal{F}\bigl(K,A_{f_n}[\pi_{\mathcal O}^\infty]\bigr)\!\Bigr) + \sum_{q\,|\,nN}t_q,
\end{equation}
where $t_q$ is the Tamagawa exponent of $A_{f_n}$ at $q$. Since, by condition (1) in Assumption \ref{asselliptic}, $A_{f_n}[\pi_{\mathcal O}]$ is an irreducible $G_K$-module, we have an isomorphism 
\[
\Sel_\mathcal{F}\bigl(K,A_{f_n}[\pi_{\mathcal O}^\infty]\bigr)[\pi_{\mathcal O}]\simeq\Sel_\mathcal{F}\bigl(K,A_{f_n}[\pi_{\mathcal O}]\bigr).
\]
Since $r_n=0$, we have $\Sel_\mathcal{F}\bigl(K,A_{f_n}[\pi_{\mathcal O}^\infty]\bigr)=0$ and, by condition (2) in Assumption \ref{asselliptic}, $t_q=0$ for each $q\,|\,N^+$. Therefore, equality \eqref{L-alg-eq} can be rewritten as
\begin{equation} \label{eqfinal1}
\ord_{\pi_{\mathcal O}}\bigl(L^{\mathrm{alg}}(f_n/K,1)\bigr)=\sum_{q\,|\,nN^-}t_q.
\end{equation}
Combining Gross's formula in \cite[Theorem 3.2]{berti2015congruences} with \cite[Lemma 4.2]{BLV}, and using condition (4) in Assumption \ref{asselliptic}, we obtain the equality 
\begin{equation} \label{eqfinal2}
\ord_{\pi_{\mathcal O}}\bigl(L^{\mathrm{alg}}(f_n/K,1)\bigr)=2\cdot\ord_{\pi_{\mathcal O}}(\lambda_{(T),n})+\sum_{q|nN^-} t_q.
\end{equation}
Finally, \eqref{eqfinal1} and \eqref{eqfinal2} show that $\lambda_{(T),n}$ is a $p$-adic unit, so $\lambda_n\not=0$ in $\Lambda/\mathfrak{m}_\Lambda$. \end{proof}

For all even integers $i\geq0$, define
\[
\mathfrak{C}_i(k)\defeq\bigl(\lambda_n^{(k)}\bigr)_{n\in\mathcal{N}_{2k},\nu(n)\leq i+e},\quad\mathfrak{C}_i\defeq\varprojlim_k \mathfrak{C}_i(k),
\]
so that $\mathfrak{C}_i$ is an ideal of $\Lambda$. Let $\mathcal X$ be the $\Lambda$-torsion submodule of the Pontryagin dual of $\Sel_{\Gr}\bigl(K_\infty,E[p^\infty]\bigr)$ and consider the (tautological) short exact sequence of $\Lambda$-modules
\begin{equation} \label{sha-iw-eq}
0\rightarrow\divv_{\mathfrak{m}_\Lambda}\!\Bigl(\Sel_{\Gr}\bigl(K_\infty,E[p^\infty]\bigr)\!\Bigr)\longrightarrow\Sel_{\Gr}\bigl(K_\infty,E[p^\infty]\bigr)\longrightarrow\Sha_{\Gr}\bigl(K_\infty,E[p^\infty]\bigr)\rightarrow0. 
\end{equation}
Taking Pontryagin duals in \eqref{sha-iw-eq} yields a canonical isomorphism
\[
\mathcal X\simeq\Sha_{\Gr}\bigl(K_\infty,E[p^\infty]\bigr)^\vee,
\]
which can be viewed as an identification.

\begin{theorem} \label{final-thm1}
The equivalence $\Fitt_i(\mathcal X)\sim\mathfrak{C}_i^2$ holds for all even integers $i\geq0$.
\end{theorem}

\begin{proof} This is a consequence of Theorem \ref{highFittLambda}. \end{proof}

In the indefinite setting, \emph{i.e.}, when $e=1$, we can also offer a result in terms of the classes $\kappa$. Namely, for all even integers $i\geq0$ put
\[
\begin{split}
\mathfrak{D}_i(k)&\defeq\Bigl(\Bigl\{f\bigl(\kappa_n^{(k)}\bigr)\;\Big|\;\text{$f \in \Hom_{\Lambda/(\pi^k)}\bigl(H^1(K,\mathbf{T}/\pi^k \mathbf{T}),\Lambda/(\pi^k)\bigr)$, $n \in \mathcal{N}^{2k}$, $\nu(n) \leq i$}\Bigr\}\Bigr)\\ &\;\subset\Lambda\big/(\pi^k).
\end{split}
\]
Moreover, define the ideal $\mathfrak{D}_i\defeq\varprojlim_k \mathfrak{D}_i(k)$ of $\Lambda$.

\begin{theorem} \label{final-thm2}
Suppose $e=1$. The equivalence $\Fitt_i(\mathcal X)\sim\mathfrak{D}_i^2$ holds for all even integers $i\geq0$.
\end{theorem}

\begin{proof} This is a consequence of Theorem \ref{lambdaKappa}. \end{proof}

We collect a consequence of Theorems \ref{final-thm1} and \ref{final-thm2} that describes the $\sim$-equivalence class $\bigl[\Fitt_i(\mathcal X)\bigr]$ of $\Fitt_i(\mathcal X)$ for all integers $i\geq0$ (\emph{cf.} Corollary \ref{pseudo-M-coro}). In the statement below, bear Remark \ref{pseudo-rem} in mind.

\begin{corollary} \label{final-coro}
Let $i\geq0$ be an integer. Then
\[
\bigl[\Fitt_i(\mathcal X)\bigr]=\begin{cases}\bigl[\mathfrak{C}_i^2\bigr]&\text{if $i$ is even},\\[3mm] \bigl[\mathfrak{C}_{i-1}\cdot\mathfrak{C}_{i+1}\bigr]&\text{if $i$ is odd}. \end{cases}
\]
Moreover, if $e=1$, then
\[
\bigl[\Fitt_i(\mathcal X)\bigr]=\begin{cases}\bigl[\mathfrak{D}_i^2\bigr]&\text{if $i$ is even},\\[3mm] \bigl[\mathfrak{D}_{i-1}\cdot\mathfrak{D}_{i+1}\bigr]&\text{if $i$ is odd}. \end{cases}
\]
\end{corollary}

\begin{proof} Combine Theorems \ref{final-thm1} and \ref{final-thm2} with Propositions \ref{even-odd-prop} and \ref{propX=M+M}. \end{proof}

Notice that Corollary \ref{final-coro} recovers Theorems A and B in the introduction.

\begin{remark}
We expect that, building on \cite{LPV} and \cite{LV}, analogous Iwasawa-theoretic results can be proved for the $p$-adic Galois representations attached by Deligne to newforms of higher (even) weight (\cite{Del-Bourbaki}): we plan to come back to these questions in a future project.
\end{remark}

\bibliographystyle{amsplain}
\bibliography{Bibliography}

\end{document}